\documentclass[11pt]{amsart}

\usepackage{amsmath,amssymb,amsfonts,mathtools}
\usepackage[margin=1.05in]{geometry}
\usepackage[colorlinks=true,linkcolor=blue,citecolor=blue,urlcolor=blue]{hyperref}
\usepackage[nameinlink,capitalize]{cleveref}

\newtheorem{theorem}{Theorem}[section]
\newtheorem{proposition}[theorem]{Proposition}
\newtheorem{lemma}[theorem]{Lemma}
\newtheorem{corollary}[theorem]{Corollary}
\theoremstyle{definition}
\newtheorem{definition}[theorem]{Definition}
\theoremstyle{remark}
\newtheorem{remark}[theorem]{Remark}

\newcommand{\R}{\mathbb R}
\newcommand{\cA}{\mathcal A}

\newcommand{\cE}{\mathcal E}
\newcommand{\cP}{\mathcal P}
\newcommand{\bfR}{\mathbf R}
\newcommand{\bfK}{\mathbf K}
\newcommand{\bfE}{\mathbf E}
\newcommand{\one}{\mathbf 1}
\newcommand{\supp}{\operatorname{supp}}
\newcommand{\diam}{\operatorname{diam}}
\newcommand{\Lip}{\operatorname{Lip}}
\newcommand{\dist}{\operatorname{dist}}
\newcommand{\qdist}{\operatorname{qdist}}

\title[A two-weight theorem on a $C^{1,\mathrm{Dini}}$ curve]
{Two-weight $T1$ theorem for fractional Riesz transforms
---\\ {\small
one measure supported on a $C^{1,\mathrm{Dini}}$ curve}}

\author{Ji Li}
\address{School of Mathematical and Physical Sciences, Macquarie University,
NSW 2109, Australia}
\email{ji.li@mq.edu.au}

\author{Chun-Yen Shen}
\address{Department of Mathematics, National Taiwan University,
Taipei 106, Taiwan}
\email{cyshen@math.ntu.edu.tw}

\author{Eric T. Sawyer}
\address{Department of Mathematics and Statistics, McMaster University,
Hamilton, Ontario L8S 4K1, Canada}
\email{sawyer@mcmaster.ca}

\author{Brett D. Wick}
\address{Department of Mathematics, Washington University in St. Louis,
One Brookings Drive, St. Louis, Missouri 63130-4899, USA}
\email{wick@math.wustl.edu}

\date{}
\subjclass[2020]{Primary 42B20; Secondary 42B25, 42B35, 35J05, 76B47}

\keywords{fractional Riesz transforms, two-weight inequalities,
$T1$ theorem, $C^{1,\mathrm{Dini}}$ curves, energy conditions,
common point masses, Poisson equation, Green potential, Biot--Savart law}

\begin{document}
\raggedbottom

\begin{abstract}
Let $0\leq\alpha<n$, and let $\sigma$ and $\omega$ be locally finite
positive Borel measures on $\R^n$, possibly with common point masses.
Suppose that one measure is supported on a compact regular
$C^{1,\mathrm{Dini}}$ curve. We prove that the $\alpha$-fractional Riesz
vector maps $L^2(\sigma)$ to $L^2(\omega)$ if and only if the two
$\cA_2^\alpha$ conditions with holes, the two punctured $A_2^\alpha$
conditions, and the forward and dual testing conditions over rotated
cubes hold. The theorem includes curves whose tangent is not
H\"older continuous of any positive order. Straightening such a curve
produces a kernel whose first derivatives need not satisfy the
H\"older estimates required by the energy-side-condition $T1$ theorem.
We prove both strong energy conditions in the straightened coordinates
using only a sufficiently small graph slope, and show that these
conditions are preserved under biLipschitz changes of variables.
The Dini condition makes the errors in the kernel and testing
comparisons summable. We then transfer the energy conditions back
and apply the $T1$ theorem to the original smooth Riesz kernel.

We also obtain a two-weight criterion for Green gradients with
curve-supported sources.  For arclength measure on the curve and
observation measure $\dist(x,\Gamma)^\beta\,dx$, the range is
$\beta>n-3$ in a bounded tube and $n-3<\beta<n-2$ in $\R^n$.
The endpoints fail logarithmically whenever the measure is locally
finite.  In dimension three, the corresponding whole-space range for
the Biot--Savart field with arbitrary vector densities is $0<\beta<1$.
\end{abstract}

\maketitle

\section{Introduction and statement of the result}

Let \(T\) be a singular integral operator, and let \(\sigma\) and
\(\omega\) be locally finite positive Borel measures.  The two-weight
problem asks for conditions on \((\sigma,\omega)\) under which \(\|T_\sigma f\|_{L^2(\omega)} := \|T(f\sigma)\|_{L^2(\omega)} \leq C\|f\|_{L^2(\sigma)}\). The measures need not be doubling or mutually absolutely continuous.
They may also have common point masses.  These features distinguish the
two-weight problem from the usual one-weight theory.

Testing conditions first appeared in the two-weight characterizations
for positive operators.  Sawyer proved such characterizations for the
Hardy--Littlewood maximal operator and for fractional and Poisson
integrals in \cite{SawyerMaximal,SawyerFractional}.  For singular
integrals, cancellation creates an additional difficulty.  The
David--Journ\'e \(T1\) theorem \cite{DavidJourne} gives a testing
criterion in the unweighted setting, but its direct two-weight analogue
does not follow from the Muckenhoupt and testing conditions alone by the
same proof.

A testing characterization for the Hilbert transform was obtained
through a series of papers.  Nazarov, Treil and Volberg
introduced random dyadic grids, corona decompositions and a pivotal
condition in their study of this problem; see
\cite{NTVHilbert}.  They conjectured that the two-weight inequality for
the Hilbert transform should be characterized by a Poisson
\(A_2\) condition and the forward and dual testing conditions.

Lacey, Sawyer, Shen and Uriarte-Tuero reduced the problem to a
local estimate in \cite{DukeI}.  Lacey proved that estimate in
\cite{DukeII}.  Thus, under the assumption
that the two measures have no common point masses, the Hilbert transform
is bounded if and only if the Poisson \(A_2\) and testing conditions
hold.  Hyt\"onen later removed the restriction on common point masses
in \cite{HytonenHilbert}.

Energy conditions played an important part in these proofs.  For the
Hilbert transform, they can be recovered from the operator by using the
identity \(\frac{1}{y-x}-\frac{1}{y-x'} = \frac{x-x'}{(y-x)(y-x')}\). When \(y\) lies outside the interval containing \(x\) and \(x'\), the
denominator on the right has a fixed sign.  This gives the energy
reversal used in \cite{HilbertEnergy,DukeI,DukeII}.

The higher-dimensional problem is different.  The corresponding
energy reversal fails in general for the fractional Riesz vector; see
\cite{FailureEnergy}.  For this reason, the general two-weight theorem
for fractional singular integrals in \cite{EnergyCommon} assumes the
forward and dual strong energy conditions in addition to the
Muckenhoupt and testing conditions.  The good-\(\lambda\) theorem in
\cite{GoodLambda} gives the same conclusion without the quasiweak
boundedness condition, while retaining the energy conditions.

There are geometric cases in which the energy conditions follow from
boundedness or testing.  Sawyer, Shen and Uriarte-Tuero proved this in \cite{CurvePaper}
when one measure is supported on a regular
\(C^{1,\delta}\) curve, where \(\delta>0\).  Small pieces of the curve
are written as graphs and straightened to a line.  The one-dimensional
geometry then gives the required energy reversal.  Combining this with
the energy-side-condition theorem gives a characterization by the
Muckenhoupt and testing conditions alone.

We replace the $C^{1,\delta}$ hypothesis in \cite{CurvePaper} by
$C^{1,\mathrm{Dini}}$.  A H\"older modulus $\vartheta(t)\lesssim t^\delta$
with $\delta>0$ satisfies the Dini condition.  The modulus
$\vartheta(t)=(\log(e/t))^{-1-\varepsilon}$, where $0<t\leq1$ and
$\varepsilon>0$, is Dini but satisfies no H\"older bound of positive
order.

The energy-side-condition theorem used in \cite{CurvePaper} requires
H\"older-continuous first derivatives of the kernel.  Straightening a
$C^{1,\mathrm{Dini}}$ curve gives a conformal kernel whose first
derivatives have Dini continuity.  We prove the energy conditions for
this kernel using the small $C^1$ norm of the graph, then transfer the
energy conditions back to the original coordinates.  We apply the
general $T1$ theorem there, to the smooth fractional Riesz kernel.
The Dini condition is used to sum the kernel errors in the change of
variables.

We now state the result.  We use the notation and the definitions of the
Muckenhoupt, punctured Muckenhoupt, and testing constants from
\cite[Sections~2.2--2.3]{CurvePaper}, and the strong energy constants
from \cite[Definition~5]{EnergyCommon}.  All operator estimates are
required uniformly in the truncation parameters.

Throughout the paper,
\begin{equation*}
 \mathfrak A_2^\alpha
 :=
 \mathcal A_2^\alpha
 +\mathcal A_2^{\alpha,*}
 +A_2^{\alpha,\mathrm{punct}}
 +A_2^{\alpha,*,\mathrm{punct}}.
\end{equation*}
Thus \(\mathfrak A_2^\alpha\) includes both one-tailed conditions with
holes and both punctured conditions.

\begin{definition}\label{def:modulus}
A function \(\vartheta:[0,1]\to[0,\infty)\) is a modulus of continuity
if it is nondecreasing, \(\vartheta(0)=0\), and
\(\lim_{t\downarrow0}\vartheta(t)=0\).  It is a Dini modulus if
\begin{equation}\label{eq:Dini}
 [\vartheta]_{\mathrm{Dini}}
 :=
 \int_0^1\vartheta(t)\,\frac{dt}{t}<\infty.
\end{equation}
For a continuously differentiable map \(F\), put
\begin{equation*}
 \vartheta_F(t)
 :=
 \sup\bigl\{\|DF(x)-DF(y)\|:|x-y|\leq t\bigr\},
 \qquad 0\leq t\leq1.
\end{equation*}
We write \(F\in C^{1,\mathrm{Dini}}\) if \(\vartheta_F\) is a Dini
modulus.  If \(F\) is defined on an interval, the supremum is taken
over that interval.  If its domain is a parameter circle, we use the
cyclic distance.
\end{definition}

The definition is local on compact sets.  The graph maps used below
have the form \(\Psi(x_1,x') = (x_1,x'+\psi(x_1))\), with  \(\psi\in C_c^{1,\mathrm{Dini}}(\R;\R^{n-1})\). Their derivative moduli are therefore defined on all of \(\R^n\).

\begin{definition}\label{def:regular-curve}
A compact curve \(\Gamma\subset\R^n\) is a regular
\(C^{1,\mathrm{Dini}}\) curve if
$
 \Gamma=\Phi([a,b]),
$
where \(\Phi\) is parametrized by arc length,
\(\Phi\in C^{1,\mathrm{Dini}}([a,b];\R^n)\), and \(\Phi\) is one-to-one,
apart from the possible equality \(\Phi(a)=\Phi(b)\) when the curve is
closed.  Thus
$
 |\Phi'(t)|=1,
 \ t\in[a,b].
$
In the closed case, we also require
\(\Phi'(a)=\Phi'(b)\), and we regard \(\Phi\) as a regular
\(C^{1,\mathrm{Dini}}\) embedding of
\(\R/(b-a)\mathbb Z\).  The derivative modulus is then computed using
the cyclic distance on this parameter circle.
\end{definition}

The fractional Riesz vector is
\begin{equation*}
 \bfR^{\alpha,n}
 =
 \bigl(R_1^{\alpha,n},\ldots,R_n^{\alpha,n}\bigr),
 \qquad
 \bfK^{\alpha,n}(x,y)
 =
c_{\alpha,n}\frac{y-x}{|y-x|^{n+1-\alpha}}.
\end{equation*}

A rotated cube is a set \(a+UQ\), where \(Q\) is a half-open
axis-parallel cube, \(a\in\R^n\), and \(U\) is orthogonal.  Its
position, side length and rotation are arbitrary.

\begin{theorem}\label{thm:main}
Let \(0\leq\alpha<n\).  Let \(\sigma\) and \(\omega\) be locally finite
positive Borel measures on \(\R^n\), possibly with common point masses.
Suppose that at least one of the measures is supported on a regular
\(C^{1,\mathrm{Dini}}\) curve.  Then \(\|\bfR^{\alpha,n}_\sigma f\|_{L^2(\omega)} \leq \mathfrak N_{\bfR^{\alpha,n}} \|f\|_{L^2(\sigma)}\) if and only if the following conditions hold:
\begin{enumerate}
\item the two one-tailed \(\cA_2^\alpha\) conditions with holes;
\item the two punctured \(A_2^\alpha\) conditions;
\item the testing conditions for \(\bfR^{\alpha,n}\) and its adjoint
over all rotated cubes.
\end{enumerate}
Moreover,
\begin{equation}\label{eq:norm-comparison}
 \mathfrak N_{\bfR^{\alpha,n}}
 \approx
 \sqrt{\mathfrak A_2^\alpha}
 +\mathfrak T_{\bfR^{\alpha,n}}
 +\mathfrak T_{\bfR^{\alpha,n}}^{\mathrm{dual}}.
\end{equation}
The constants depend on \(n\), \(\alpha\), and the fixed finite graph
cover of the curve.  More precisely, they may depend on the number of
graph pieces, their collar and separation constants, the biLipschitz
constants of the graph maps, and the Dini integrals of their derivative
moduli.  They may also depend on the fixed cutoff, the fixed goodness
exponent, and the resulting choices of \(A_0\), \(\gamma\), and \(r\).
\end{theorem}

The case \(\alpha=1\) applies to Poisson equations with sources
supported on the curve.  Such equations occur in models of line wells, thin heat sources,
and mixed-dimensional flow; see
\cite{GjerdeKumarNordbottenWohlmuth2019,
GjerdeKumarNordbotten2021,Ojea2024}.  If \(-\Delta u=f\,\sigma\) in a bounded domain, then the singular part of \(\nabla u\) is the
\(1\)-fractional Riesz vector.  The remaining part of the Green gradient
has a bounded kernel when the curve is separated from the boundary.
Hence \cref{thm:main} shows that the Green-gradient estimate is
equivalent to the four Muckenhoupt conditions and the forward and dual
Green testing conditions, for any finite observation measure.

For arclength measure on the curve and observation measure
\(\dist(x,\Gamma)^\beta\,dx\), we then determine the exact range for the
gradient estimate.  In a bounded tubular neighbourhood the range is \(\beta>n-3\), whereas in the whole space it is \(n-3<\beta<n-2\). At every endpoint compatible with local finiteness, the corresponding
integral diverges logarithmically.  In dimension three, the
Biot--Savart operator for arbitrary vector densities on the filament
has norm comparable to that of \(\bfR^{1,3}\), and its whole-space range
is exactly \(0<\beta<1\); see also \cite{Polanco2025} for Biot--Savart
integrals over curves.  These results apply, in particular, to regular
Dini curves whose tangent is not H\"older continuous of any positive
order.  They are proved in Section~\ref{sec:Poisson-application}.

Necessity follows from
\cite[Theorem~5(2)]{CurvePaper}.  We include the argument at the end of
Section~\ref{sec:7}.  The rest of the paper proves sufficiency.

\medskip
\noindent
\textit{Main points of the proof.}
The proof has five steps.

\smallskip
\noindent
\textit{1. Local graph maps.}
The continuity of the tangent vector allows us to divide the curve into
finitely many arcs.  After a rotation, each arc has the form \(\bigl\{(t,\psi_i(t)):t\in I_i\bigr\}, \quad \|\psi_i'\|_\infty\leq\varepsilon_0\), where \(\varepsilon_0>0\) is sufficiently small.  Using nested intervals
with a fixed collar, we extend \(\psi_i\) to a compactly supported
\(C^{1,\mathrm{Dini}}\) function on \(\R\).  The resulting graph shear
is globally biLipschitz, and the Dini integral of its derivative is
controlled by that of the original parametrization.

\smallskip
\noindent
\textit{2. The Dini change-of-variables estimate.}
The first-order remainder satisfies \(|\psi_i(s)-\psi_i(t)-\psi_i'(t)(s-t)| \leq |s-t|\,\vartheta_{\psi_i}(|s-t|)\). Hence the error in the pulled-back fractional Riesz kernel obeys \(|E_{\mathrm{near}}(x,y)| \leq C|x-y|^{\alpha-n} \vartheta_{\psi_i}(|x-y|)\). On the annulus
\(2^{-j-1}r_0<|x-y|\leq2^{-j}r_0\), its operator norm is bounded by
\(C\vartheta_{\psi_i}(2^{-j}r_0)\sqrt{\mathfrak A_2^\alpha}\).  The sum
converges because
\[
 \sum_{j\geq0}
 \vartheta_{\psi_i}(2^{-j}r_0)
 \lesssim
 \vartheta_{\psi_i}(r_0)
 +
 \int_0^{r_0}
 \vartheta_{\psi_i}(t)\,\frac{dt}{t}.
\]
Since the graph extension
is bounded, the error away from the diagonal satisfies \(|E_{\mathrm{far}}(x,y)| \leq C\|\psi_i\|_\infty |x-y|^{\alpha-n-1}\). The additional factor \(|x-y|^{-1}\) makes the outer annuli summable.
These bounds are uniform in the supports and truncation parameters.
They give the operator and testing comparisons between the Riesz and
conformal Riesz transforms.

\smallskip
\noindent
\textit{3. Energy on the straight line.}
After straightening an arc, the restricted curve measure is supported on
the first coordinate axis.  The derivative computations from
\cite[Theorem~13]{CurvePaper} use
\(\|\psi_i'\|_\infty\leq\varepsilon_0\), but no difference of the form
\(\psi_i'(s)-\psi_i'(t)\) occurs.  Thus the energy reversal needs neither
a H\"older nor a Dini modulus.  In the end region, the first coordinate
derivative has a fixed sign.  The side contribution is bounded by the one-tailed Muckenhoupt
condition and the overlap of the dilated cubes.  The weak and strong
backward reversals give the dual energy condition.
The calculations are given in Section~\ref{sec:5}.

\smallskip
\noindent
\textit{4. Transfer and localization.}
Strong energy is preserved by a globally biLipschitz change of variables.
This follows from the comparisons of the Poisson integrals and variances,
together with the correspondence of deep quasicubes.  We can therefore
transfer the energy estimates from the line to the curve.

It remains to obtain testing after restricting the curve measure to one
arc.  The arcs are half-open, so a common endpoint atom belongs to one
arc only.  Quantitative tube separation keeps the other parts of the
curve away from the straightened interval.  The terms near its endpoints
reduce to one-dimensional weighted Hardy inequalities, whose constants
are controlled by the one-tailed Muckenhoupt conditions with holes.  This
also gives testing for arbitrarily large quasicubes.

\smallskip
\noindent
\textit{5. Return to the smooth kernel.}
For each restricted pair, we have the four Muckenhoupt conditions, the
forward and dual testing conditions, and both strong energy conditions.
We apply \cite[Theorem~2]{GoodLambda} to each scalar component
\(R_j^{\alpha,n}\) of the original fractional Riesz vector.  This kernel is smooth, so the hypotheses of the general two-weight
theorem hold.  Summing over the finitely many
disjoint curve pieces proves the result.  

Section~\ref{sec:2} contains the estimates for Dini moduli and the graph
extensions.  Section~\ref{sec:3} proves the change-of-variables theorem, including
the comparison of the truncations.  Section~\ref{sec:4} proves the biLipschitz
invariance of the strong energy conditions.  Section~\ref{sec:5} proves the
forward and dual energy conditions when one measure is supported on a
line.  Section~\ref{sec:6} gives the geometric decomposition of the curve and the
localized testing estimates.  Section~\ref{sec:7} proves \cref{thm:main}.
Section~\ref{sec:8} proves the two-weight criterion for Green gradients with a
curve-supported source.  It then gives the exact distance-power ranges
in a bounded tubular neighbourhood and in the whole space, including
the logarithmic endpoint failures, and applies the three-dimensional
result to the Biot--Savart field.

\section{Dini estimates}\label{sec:2}

\begin{lemma}\label{lem:dyadic-Dini}
Let $\vartheta$ be a nondecreasing modulus.  Then
\begin{equation}\label{eq:dyadic-Dini}
 \sum_{k=0}^\infty \vartheta(2^{-k})
 \approx
 \vartheta(1)+\int_0^1\vartheta(t)\,\frac{dt}{t}.
\end{equation}
For each fixed $c\geq1$,
\begin{equation}\label{eq:dilated-Dini}
 \sum_{k=0}^\infty
 \vartheta\bigl(\min\{1,c2^{-k}\}\bigr)
 \leq C_c\left(\vartheta(1)+[\vartheta]_{\mathrm{Dini}}\right).
\end{equation}
\end{lemma}

\begin{proof}
For $k\geq0$ and $2^{-k-1}<t\leq2^{-k}$, monotonicity gives \(\vartheta(2^{-k-1}) \leq \vartheta(t) \leq \vartheta(2^{-k})\). Integration with respect to $dt/t$ and summation in $k$ give
\eqref{eq:dyadic-Dini}.  To obtain \eqref{eq:dilated-Dini}, choose an
integer $m_c$ such that $c\leq2^{m_c}$.  The terms with $k\leq m_c$ are
bounded by $(m_c+1)\vartheta(1)$, and the remaining sum is bounded by a
shift of the sum in \eqref{eq:dyadic-Dini}.
\end{proof}

\begin{lemma}\label{lem:Taylor-Dini}
Let $F\in C^1(\R^n;\R^n)$.  If $|x-y|\leq1$, then \(|F(y)-F(x)-DF(x)(y-x)| \leq |x-y|\,\vartheta_F(|x-y|)\). \end{lemma}

\begin{proof}
By the fundamental theorem of calculus,
\[
F(y)-F(x)-DF(x)(y-x)
 =\int_0^1\bigl(DF(x+t(y-x))-DF(x)\bigr)(y-x)\,dt.
\]
The norm of the integrand is at most
$\vartheta_F(|x-y|)|x-y|$.
\end{proof}

\begin{lemma}
\label{lem:graph-modulus}
Let \(\Gamma:[c,d]\to\R^n\) be parametrized by arc length.  Suppose,
after a rotation, that
\[
 \Gamma_1'(s)\geq c_0>0,
 \qquad c\leq s\leq d.
\]
Write \(u=\Gamma_1(s)\), and let \(\Gamma(s)=(u,g(u))\). Then
\begin{equation}\label{eq:graph-modulus}
 \vartheta_g(r)
 \leq C_{c_0}\vartheta_\Gamma(C_{c_0}r),
 \qquad
 0<r\leq C_{c_0}^{-1}.
\end{equation}
In particular, \(g\in C^{1,\mathrm{Dini}}\) whenever
\(\Gamma\in C^{1,\mathrm{Dini}}\).
\end{lemma}

\begin{proof}
Since \(u'(s)=\Gamma_1'(s)\geq c_0\), the inverse function
\(s=s(u)\) satisfies \(|s(u)-s(v)|\leq c_0^{-1}|u-v|\). Moreover, \(g'(u) =\frac{\Gamma_\perp'(s(u))}{\Gamma_1'(s(u))}\). The denominators are at least \(c_0\), and
\(|\Gamma'(s)|=1\).  Hence \(|g'(u)-g'(v)| \leq C_{c_0} |\Gamma'(s(u))-\Gamma'(s(v))| \leq C_{c_0} \vartheta_\Gamma(c_0^{-1}|u-v|)\). Choose \(C_{c_0}\geq\max\{1,c_0^{-1}\}\) large enough that the
preceding estimate has the form in \eqref{eq:graph-modulus}.  If
\(0<r\leq C_{c_0}^{-1}\), then \(C_{c_0}r\leq1\), so the argument of
\(\vartheta_\Gamma\) lies in its defining interval \([0,1]\).
This proves \eqref{eq:graph-modulus}.  For
\(C_{c_0}^{-1}\leq r\leq1\), we have
\(\vartheta_g(r)\leq2\|g'\|_\infty\).  A change of variables on the
remaining part of the Dini integral proves the last assertion.
\end{proof}

\begin{lemma}\label{lem:graph-extension}
Let \(I^*\) be a bounded open interval, and let
\(\widehat I\Subset I^*\).  Assume that
\begin{equation}\label{eq:fixed-collar}
 \dist(\widehat I,\R\setminus I^*)
 \geq \kappa |I^*|
\end{equation}
for some \(\kappa>0\).  Let
\(g:I^*\to\R^{n-1}\) be \(C^{1,\mathrm{Dini}}\).  Suppose that
\(t_0\in I^*\), \(g(t_0)=0\), and
\begin{equation}\label{eq:local-smallness}
 \|g'\|_{L^\infty(I^*)}\leq\varepsilon_0.
\end{equation}
There is a compactly supported function
$\widetilde g\in C^{1,\mathrm{Dini}}(\R;\R^{n-1})$ such that
$\widetilde g=g$ on $\widehat I$ and
\begin{equation}\label{eq:extension-slope}
 \|\widetilde g'\|_\infty\leq C_\kappa\varepsilon_0.
\end{equation}
If $\vartheta_g$ is the modulus of $g'$, then the derivative of
$\widetilde g$ has a modulus satisfying
\begin{equation}\label{eq:extension-modulus}
 \vartheta_{\widetilde g}(t)
 \leq C_\kappa\vartheta_g(t)
 +C_\kappa\varepsilon_0\min\{1,t/|I^*|\},\qquad 0<t\leq1.
\end{equation}
In particular, $\vartheta_{\widetilde g}$ is a Dini modulus.
\end{lemma}

\begin{proof}
By \(g(t_0)=0\) and \eqref{eq:local-smallness}, \(\|g\|_{L^\infty(I^*)} \leq \varepsilon_0|I^*|\). Condition \eqref{eq:fixed-collar} gives
\(\rho\in C_c^\infty(I^*)\) such that \(\rho=1\) on
\(\widehat I\) and \(\|\rho'\|_\infty\leq C_\kappa|I^*|^{-1}, \quad \|\rho''\|_\infty\leq C_\kappa|I^*|^{-2}\). Extend $\rho g$ by zero outside $I^*$ and put
$\widetilde g=\rho g$.  Then \(\widetilde g'=\rho g'+\rho'g\). The preceding bounds for \(g\), \(g'\), and \(\rho'\) give
\eqref{eq:extension-slope}.  If $|s-t|\leq r$, the difference of the first
term is bounded by \(\vartheta_g(r)+C\varepsilon_0 r/|I^*|\). The difference of the second term has the same bound, by the estimates
for $\rho'$, $\rho''$, $g$, and $g'$.  If one point lies outside
$I^*$, insert the intervening endpoint of $I^*$; both the extension
and its derivative vanish near that endpoint.  This proves the
estimate in \eqref{eq:extension-modulus} when \(0<t<|I^*|\).
When \(t\geq|I^*|\), \eqref{eq:extension-slope} gives \(\vartheta_{\widetilde g}(t) \leq2\|\widetilde g'\|_\infty \leq C_\kappa\varepsilon_0\). Combining these two ranges proves \eqref{eq:extension-modulus}.  Its
Dini integral is finite by
\eqref{eq:Dini}.
\end{proof}

\section{The Dini change of variables}\label{sec:3}

The maps used in the proof have the form
\begin{equation}\label{eq:graph-change}
 \Psi(x)=x+u(x),
 \qquad
u(x_1,x')=(0,\psi(x_1)),
 \qquad
 \psi\in C_c^{1,\mathrm{Dini}}(\R;\R^{n-1}).
\end{equation}
Thus \(u\) is bounded, although its support as a function on \(\R^n\)
need not be compact.  For every homeomorphism
\(F:\R^n\to\R^n\), we use the notation
\begin{equation}\label{eq:pullback-measure}
F^*\mu:=(F^{-1})_\#\mu,
 \qquad
 (F^*\mu)(E)=\mu(F(E)).
\end{equation}
Thus \(F^*\mu\) is the measure in the \(F^{-1}\)-coordinates.  For
every nonnegative Borel function \(h\),
\[
 \int_{\R^n}h(x)\,d(F^*\mu)(x)
 =\int_{\R^n}h(F^{-1}X)\,d\mu(X).
\]
We write \(\widetilde\mu=\Psi^*\mu\).
The conformal fractional Riesz kernel is
\begin{equation}\label{eq:conformal-kernel}
 \bfK_\Psi^{\alpha,n}(x,y)
 =c_{\alpha,n}\frac{y-x}
 {|\Psi(y)-\Psi(x)|^{n+1-\alpha}}.
\end{equation}

For \(r>0\), put
\begin{equation*}
 \mathcal M_{\alpha,r}(f\sigma)(x)
 :=
r^{\alpha-n}
 \int_{\{r/2<|x-y|\leq2r\}}|f(y)|\,d\sigma(y).
\end{equation*}

\begin{lemma}\label{lem:annular-A2}
For every \(r>0\),
\[
 \|\mathcal M_{\alpha,r}(f\sigma)\|_{L^2(\omega)}
 \leq C_{n,\alpha}
 \sqrt{\mathfrak A_2^\alpha(\sigma,\omega)}
       \|f\|_{L^2(\sigma)}.
\]
The constant is independent of \(r\).  The dual estimate also holds.

Let \(\Omega:\R^n\to\R^n\) be a fixed globally biLipschitz map.  Write \(d_\Omega(x,y):=|\Omega^{-1}x-\Omega^{-1}y|\) and define
\[
 \mathcal M_{\alpha,r}^{\Omega}(f\sigma)(x)
 :=
r^{\alpha-n}
 \int_{\{r/2<d_\Omega(x,y)\leq2r\}}
 |f(y)|\,d\sigma(y).
\]
Then
\[
 \|\mathcal M_{\alpha,r}^{\Omega}(f\sigma)\|_{L^2(\omega)}
 \leq C_{n,\alpha,\Omega}
 \sqrt{\mathfrak A_{2,\Omega}^\alpha(\sigma,\omega)}
 \|f\|_{L^2(\sigma)}.
\]
The corresponding dual estimate also holds.
The subscript \(\Omega\) means that the four Muckenhoupt constants
are computed over the \(\Omega\)-quasicubes.
\end{lemma}

\begin{proof}
Put \(s=n-\alpha\).  We first prove the Euclidean assertion.  For an
ordinary cube \(Q\), recall that
\[
 \mathcal P^\alpha(Q,\nu)
 =
 \int_{\R^n}
 \left(
 \frac{\ell(Q)}
 {(\ell(Q)+|y-c_Q|)^2}
 \right)^s d\nu(y).
\]
The forward one-tailed condition with a hole is
\[
 \mathcal A_2^\alpha(\sigma,\omega)
 =
 \sup_Q
 \frac{\omega(Q)}{\ell(Q)^s}
 \mathcal P^\alpha(Q,\one_{Q^c}\sigma).
\]

Set \(\delta=(16\sqrt n)^{-1}\), and partition \(\R^n\) into half-open
cubes \(Q_m\) of side length \(\delta r\).  Let \(Q_m^*\) be the cube
concentric with \(Q_m\) and of side length \(8r\).

Fix \(x\in Q_m\).  If \(r/2<|x-y|\leq2r\), then \(y\notin3Q_m\).
Indeed, \(x\in Q_m\) and \(y\in3Q_m\) would give \(|x-y|\leq2\sqrt n\,\delta r=\frac r8\). Also, \(|y-c_{Q_m}| \leq |y-x|+|x-c_{Q_m}| <2r+\frac r{32}<4r\), so \(y\in Q_m^*\).  Consequently, \(\{y:r/2<|x-y|\leq2r\} \subset E_m:=Q_m^*\setminus3Q_m\). In particular, \(E_m\subset Q_m^c\).

For \(y\in E_m\), we have
\(|y-c_{Q_m}|\leq4\sqrt n\,r\).  Hence
\begin{align*}
 \mathcal P^\alpha(Q_m,\one_{Q_m^c}\sigma)
 &\geq
 \mathcal P^\alpha(Q_m,\one_{E_m}\sigma)
 \geq
 \left(
 \frac{\delta r}
 {(\delta r+4\sqrt n\,r)^2}
 \right)^s
 \sigma(E_m)
 \geq c_{n,\alpha}r^{-s}\sigma(E_m).
\end{align*}
Since \(\ell(Q_m)=\delta r\), the forward one-tailed condition gives \(r^{-2s}\omega(Q_m)\sigma(E_m) \leq C_{n,\alpha}\mathcal A_2^\alpha(\sigma,\omega)\).

By Cauchy--Schwarz, for \(x\in Q_m\),
\[
 |\mathcal M_{\alpha,r}(f\sigma)(x)|^2
 \leq
r^{-2s}\sigma(E_m)\int_{E_m}|f|^2\,d\sigma.
\]
It follows that
\begin{align*}
 \int_{Q_m}|\mathcal M_{\alpha,r}(f\sigma)|^2\,d\omega
 &\leq
r^{-2s}\omega(Q_m)\sigma(E_m)
 \int_{E_m}|f|^2\,d\sigma
 \leq
C_{n,\alpha}\mathcal A_2^\alpha(\sigma,\omega)
 \int_{Q_m^*}|f|^2\,d\sigma.
\end{align*}
The cubes \(Q_m^*\) have bounded overlap.  Summing in \(m\) proves the
forward estimate.  Interchanging \(\sigma\) and \(\omega\) proves the
dual estimate, using \(\mathcal A_2^{\alpha,*}(\sigma,\omega)\).

We now prove the quasicube assertion.  Let \(L\geq1\) be a common
biLipschitz constant for \(\Omega\) and \(\Omega^{-1}\).  Use the same
partition \(\{Q_m\}\) in the preimage coordinates, and put \(H_m=\Omega Q_m,\quad H_m^*=\Omega Q_m^*,\quad F_m=H_m^*\setminus\Omega(3Q_m)\). If \(x\in H_m\) and \(r/2<d_\Omega(x,y)\leq2r\), the preceding
Euclidean calculation applied to \(\Omega^{-1}x\) and
\(\Omega^{-1}y\) gives \(y\in F_m\).  Also \(F_m\subset H_m^c\).

Write \(c_{H_m}=\Omega(c_{Q_m})\) and
\(a_{H_m}=|H_m|^{1/n}\).  BiLipschitz continuity gives \(L^{-1}\delta r\leq a_{H_m}\leq L\delta r\). For \(y\in F_m\), \(|y-c_{H_m}| \leq L|\Omega^{-1}y-c_{Q_m}| \leq4L\sqrt n\,r\). The quasicube reproducing Poisson integral is
\[
 \mathcal P^\alpha(H,\nu)
 =
 \int_{\R^n}
 \left(
 \frac{|H|^{1/n}}
 {\bigl(|H|^{1/n}+|y-c_H|\bigr)^2}
 \right)^s d\nu(y).
\]
Therefore \(\mathcal P^\alpha(H_m,\one_{H_m^c}\sigma) \geq c_{n,\alpha,L}r^{-s}\sigma(F_m)\). The forward quasicube one-tailed condition gives \(r^{-2s}\omega(H_m)\sigma(F_m) \leq C_{n,\alpha,L}\mathcal A_{2,\Omega}^\alpha(\sigma,\omega)\). Cauchy--Schwarz now gives
\[
 \int_{H_m}
 |\mathcal M_{\alpha,r}^{\Omega}(f\sigma)|^2\,d\omega
 \leq
C_{n,\alpha,L}\mathcal A_{2,\Omega}^\alpha(\sigma,\omega)
 \int_{H_m^*}|f|^2\,d\sigma.
\]
The sets \(H_m^*\) have the same overlap multiplicity as the cubes
\(Q_m^*\), because \(\Omega\) is one-to-one.  Summing in \(m\) proves
the quasicube estimate.  The dual quasicube estimate follows by
interchanging the two measures.  The half-open partition assigns every
boundary atom to one member, and \(F_m\subset H_m^c\) separates the two
measures in each one-tailed product.
\end{proof}

Fix \(\eta_0\in C^\infty([0,\infty))\) with \(0\leq\eta_0\leq1,\quad \eta_0=0\text{ on }[0,1/2],\quad \eta_0=1\text{ on }[1,\infty)\). For \(0<\varepsilon<R\), put
\begin{equation}\label{eq:smooth-truncation}
 \tau_{\varepsilon,R}(0)=0,
 \qquad
 \tau_{\varepsilon,R}(t)
 =\eta_0(t/\varepsilon)\eta_0(R/t)\quad(t>0),
 \qquad
q_{\varepsilon,R}(x,y)
 =\tau_{\varepsilon,R}(|\Psi(y)-\Psi(x)|).
\end{equation}
We also put
\begin{equation}\label{eq:source-radial-cutoff}
s_{\varepsilon,R}(x,y)
 =\tau_{\varepsilon,R}(|x-y|).
\end{equation}
Thus \(s_{\varepsilon,R}\) is radial in the present coordinates, and
\(q_{\varepsilon,R}\) is radial after applying \(\Psi\).  We have
\[
 \begin{array}{lll}
s_{\varepsilon,R}(x,y)=0
   &\text{if }|x-y|\leq\varepsilon/2
   &\text{or }|x-y|\geq2R,\\
s_{\varepsilon,R}(x,y)=1
   &\text{if }\varepsilon\leq|x-y|\leq R.&
 \end{array}
\]
The analogous statements for \(q_{\varepsilon,R}\) hold with
\(|x-y|\) replaced by \(|\Psi(y)-\Psi(x)|\).

Let \(K\) be a scalar or vector kernel, and let \(\mu,\nu\) be locally finite
positive Borel measures.  For \(a=s\) or \(a=q\), define
\begin{equation*}
T_{\varepsilon,R,\mu}^{K,a}f(x)
 :=\int a_{\varepsilon,R}(x,y)K(x,y)f(y)\,d\mu(y),
 \qquad 0<\varepsilon<R,
\end{equation*}
where the truncated kernel is defined to be zero when \(x=y\).  This
convention removes the diagonal contribution of every common point mass.
Set
\begin{equation}\label{eq:uniform-radial-norms}
 \mathfrak N_K^{a}(\mu,\nu)
 :=\sup_{0<\varepsilon<R}
 \bigl\|T_{\varepsilon,R,\mu}^{K,a}
 \bigr\|_{L^2(\mu)\to L^2(\nu)}.
\end{equation}
For a family \(\mathcal F\) of bounded Borel sets, put
\begin{align*}
 \mathfrak T_{K,\mathcal F}^{a}(\mu,\nu)
 &:=\sup_{\substack{E\in\mathcal F,\ \mu(E)>0\\0<\varepsilon<R}}
 \frac{
 \bigl\|\one_E T_{\varepsilon,R,\mu}^{K,a}\one_E
 \bigr\|_{L^2(\nu)}}{\mu(E)^{1/2}},\\
 \mathfrak T_{K,\mathcal F}^{a,\mathrm{dual}}(\mu,\nu)
 &:=\sup_{\substack{E\in\mathcal F,\ \nu(E)>0\\0<\varepsilon<R}}
 \frac{
 \bigl\|\one_E T_{\varepsilon,R,\nu}^{K^*,a}\one_E
 \bigr\|_{L^2(\mu)}}{\nu(E)^{1/2}}.
\end{align*}
Here \(K^*(x,y)=\overline{K(y,x)}\), component by component.  Both
\(s_{\varepsilon,R}\) and \(q_{\varepsilon,R}\) are symmetric in
\(x,y\).  Thus, when \(K\) is vector-valued, the last expression is
formed from the formal adjoints of its scalar components.
Here and below the norm of a vector-valued function is its Euclidean
norm.  The superscript \(s\) always refers to the source-radial cutoff in
\eqref{eq:source-radial-cutoff}.  The superscript \(q\) always refers to
the pulled-back target-radial cutoff in \eqref{eq:smooth-truncation}.

\begin{lemma}\label{lem:cutoff-stability}
Let \(\Psi\) be globally biLipschitz.  Let
\(\chi\in C^\infty([0,\infty))\) satisfy
\begin{equation}\label{eq:fixed-near-cutoff}
 0\leq\chi\leq1,\qquad
 \chi=1\text{ on }[0,r_0],\qquad
 \chi=0\text{ on }[2r_0,\infty),\qquad
 |\chi'|\leq C/r_0.
\end{equation}
Fix a measure pair \((\mu,\nu)\) with
\(\mathfrak A_2^\alpha(\mu,\nu)<\infty\).  Suppose that
\(K\) satisfies \(|K(x,y)|\leq C_K|x-y|^{\alpha-n}\). Then the operators with kernels \(q_{\varepsilon,R}(x,y)\chi(|x-y|)K(x,y),\quad q_{\varepsilon,R}(x,y)(1-\chi(|x-y|))K(x,y)\) have norm at most
\begin{equation}\label{eq:cutoff-stability-bound}
C\left(\mathfrak N_K^s(\mu,\nu)
 +\sqrt{\mathfrak A_2^\alpha(\mu,\nu)}\right)
\end{equation}
from \(L^2(\mu)\) to \(L^2(\nu)\), uniformly in
\(0<\varepsilon<R\).  Here \(C\) depends on
\(n,\alpha,C_K\) and the biLipschitz constant of \(\Psi\), but not on
\(\varepsilon,R\).  For any fixed family \(\mathcal F\), the
corresponding forward and dual testing bounds hold with
\(\mathfrak N_K^s\) in \eqref{eq:cutoff-stability-bound} replaced by
\(\mathfrak T_{K,\mathcal F}^s\) and
\(\mathfrak T_{K,\mathcal F}^{s,\mathrm{dual}}\), respectively.
For a fixed globally biLipschitz map \(\Omega\), the same assertions hold
with \(\mathfrak A_2^\alpha\) replaced by the combined
\(\Omega\)-quasicube constant.  The testing family may be any
fixed family of bounded Borel sets; in the application it is a family of
\(\Omega\)-quasicubes.  The constant may depend on \(\Omega\).
\end{lemma}

\begin{proof}
Let \(L\) be a common biLipschitz constant for \(\Psi,\Psi^{-1}\).  Then \(L^{-1}|x-y|\leq|\Psi(y)-\Psi(x)|\leq L|x-y|\). The two cutoff families satisfy
\begin{align*}
 \{q_{\varepsilon,R}\neq s_{\varepsilon,R}\}
 &\subset
 \left\{\frac{\varepsilon}{2L}<|x-y|<L\varepsilon\right\}\cup
 \left\{\frac{R}{L}<|x-y|<2LR\right\}.
\end{align*}
Each set on the right is a union of at most \(C_L\) dyadic annuli.
The size estimate for \(K\) and \cref{lem:annular-A2} therefore give
\begin{equation}\label{eq:q-s-general-kernel}
 \sup_{0<\varepsilon<R}
 \bigl\|T_{\varepsilon,R,\mu}^{K,q}
       -T_{\varepsilon,R,\mu}^{K,s}\bigr\|_{L^2(\mu)\to L^2(\nu)}
 \leq C_L C_K\sqrt{\mathfrak A_2^\alpha(\mu,\nu)}.
\end{equation}

It remains to insert \(\chi\).  Put \(\rho=\min\{R,r_0\}\).  If
\(\varepsilon<\rho\), then \(s_{\varepsilon,R}(x,y)\chi(|x-y|) -s_{\varepsilon,\rho}(x,y)\) is supported in
\(\{r_0\leq|x-y|\leq2r_0\}\).  If
\(\varepsilon\geq\rho\), then necessarily \(\rho=r_0\), and
\(s_{\varepsilon,R}\chi\) is supported in
\(\{r_0/2<|x-y|<2r_0\}\).  In both cases the error is either zero or is
supported where \(|x-y|\approx r_0\).  Its kernel is bounded by a fixed
number of kernels of the form \(C\rho_0^{\alpha-n} \one_{\{\rho_0/2<|x-y|\leq2\rho_0\}}, \quad \rho_0\approx r_0\). Hence
\begin{align*}
 &\sup_{0<\varepsilon<R}
 \left\|\int q_{\varepsilon,R}(x,y)\chi(|x-y|)
K(x,y)f(y)\,d\mu(y)\right\|_{L^2(\nu)}\leq C\left(\mathfrak N_K^s(\mu,\nu)
 +\sqrt{\mathfrak A_2^\alpha(\mu,\nu)}\right)
 \|f\|_{L^2(\mu)}.
\end{align*}
The full \(q\)-truncation is controlled by
\eqref{eq:q-s-general-kernel}; more precisely,
\[
 \sup_{0<\varepsilon<R}
 \|T_{\varepsilon,R,\mu}^{K,q}\|_{L^2(\mu)\to L^2(\nu)}
 \leq \mathfrak N_K^s(\mu,\nu)
 +C_LC_K\sqrt{\mathfrak A_2^\alpha(\mu,\nu)}.
\]
Subtracting the part with \(\chi\)
gives the same estimate for the part with \(1-\chi\).

For forward testing, apply these identities to \(f=\one_E\), restrict
the result to \(E\), and divide by \(\mu(E)^{1/2}\).  The source-radial
term is bounded by \(\mathfrak T_{K,\mathcal F}^s\).  Every annular error
is bounded by its operator norm and hence by
\(C\sqrt{\mathfrak A_2^\alpha}\).  Interchanging the variables proves
the dual testing estimate.

For the \(\Omega\)-quasicube assertion, use the preimage distance
\(|\Omega^{-1}x-\Omega^{-1}y|\).  By biLipschitz continuity, each
Euclidean annulus above is contained in a fixed number of annuli for
this distance, with comparable radii.  Apply the quasicube form of
\cref{lem:annular-A2} to these annuli.  The preceding proof is otherwise
unchanged.
\end{proof}

We now define the tangent-line truncations used in this paper.  Put
\(d=n-\alpha>0\).  For \(0<\varepsilon<R\), let \(S_R=(1+d^{-1})R\) and define
\[
 \rho_{\varepsilon,R}^{\alpha}(r)
 =
 \begin{cases}
  \varepsilon^{-d}(1+d-d\,r/\varepsilon),
     &0<r<\varepsilon,\\
r^{-d},&\varepsilon\leq r\leq R,\\
R^{-d}(1+d-d\,r/R),
     &R<r<S_R,\\
  0,&r\geq S_R.
 \end{cases}
\]
The values and first derivatives of the adjacent formulas agree at
\(r=\varepsilon\) and \(r=R\), and
\(\rho_{\varepsilon,R}^{\alpha}(S_R)=0\).  This is the profile used in
\cite[Section~2]{CurvePaper}.  For \(x\neq y\), put \(t_{\varepsilon,R}(x,y) :=|x-y|^d \rho_{\varepsilon,R}^{\alpha}(|x-y|)\). The tangent-line truncated kernel is
\[
K_{\varepsilon,R}^{\mathrm{tl}}(x,y)
 :=
 \begin{cases}
t_{\varepsilon,R}(x,y)K(x,y),&x\neq y,\\
 0,&x=y.
 \end{cases}
\]
The corresponding operator is
\[
T_{\varepsilon,R,\mu}^{K,\mathrm{tl}}f(x)
 :=\int K_{\varepsilon,R}^{\mathrm{tl}}(x,y)f(y)\,d\mu(y).
\]
For \(K=\bfK_\Psi^{\alpha,n}\), we also write this operator as
\(\bfR_{\Psi,\varepsilon,R}^{\alpha,n,\mathrm{tl}}\).
For \(K=\bfK^{\alpha,n}\), we write it as
\(\bfR_{\varepsilon,R}^{\alpha,n,\mathrm{tl}}\).
The quantities
\(\mathfrak N_K^{\mathrm{tl}}\),
\(\mathfrak T_{K,\mathcal F}^{\mathrm{tl}}\), and
\(\mathfrak T_{K,\mathcal F}^{\mathrm{tl},\mathrm{dual}}\)
are defined as in \eqref{eq:uniform-radial-norms} and the two displays
following it.  Notice that \(t_{\varepsilon,R}(x,y)=1 \quad\text{when}\quad \varepsilon\leq |x-y|\leq R\). Thus, if \(A,B\subset\R^n\) satisfy
\[
 \varepsilon<\inf_{x\in A,\,y\in B}|x-y|
 \leq\sup_{x\in A,\,y\in B}|x-y|<R,
\]
then the tangent-line truncated kernel equals the full kernel on
\(A\times B\).

\begin{lemma}
\label{lem:tangent-smooth-comparison}
Let \(K\) be the fractional Riesz vector \(\bfK^{\alpha,n}\), one of
its scalar components \(K_j^{\alpha,n}(x,y) :=e_j\cdot\bfK^{\alpha,n}(x,y), \quad 1\leq j\leq n\), the conformal fractional Riesz vector \(\bfK_\Psi^{\alpha,n}\), or one
of its scalar components \(K_{\Psi,j}^{\alpha,n}(x,y) :=e_j\cdot\bfK_\Psi^{\alpha,n}(x,y), \quad 1\leq j\leq n\), where \(\Psi\) is globally biLipschitz.
Let \(\mu,\nu\) be locally finite positive Borel measures, possibly
with common point masses.  Let \(\Omega\) be a fixed globally
biLipschitz map, and compute the combined Muckenhoupt constant with
respect to the \(\Omega\)-quasicubes.  Then
\[
 \sup_{0<\varepsilon<R}
 \left\|
T_{\varepsilon,R,\mu}^{K,\mathrm{tl}}
 -T_{\varepsilon,R,\mu}^{K,s}
 \right\|_{L^2(\mu)\to L^2(\nu)}
 \leq
C\sqrt{\mathfrak A_{2,\Omega}^\alpha(\mu,\nu)}.
\]
The same estimate holds for the adjoint after interchanging
\(\mu\) and \(\nu\); for a vector kernel, the adjoint is taken
componentwise.  Consequently,
\begin{equation}\label{eq:smooth-tangent-comparison}
 \begin{aligned}
 \mathfrak N_K^s
 &\leq C\left(
 \mathfrak N_K^{\mathrm{tl}}
 +\sqrt{\mathfrak A_{2,\Omega}^\alpha}\right),
 &\qquad
 \mathfrak N_K^{\mathrm{tl}}
 &\leq C\left(
 \mathfrak N_K^s
 +\sqrt{\mathfrak A_{2,\Omega}^\alpha}\right).
 \end{aligned}
\end{equation}
For every family \(\mathcal F\) of bounded Borel sets, the same two
inequalities hold with
\((\mathfrak N_K^s,\mathfrak N_K^{\mathrm{tl}})\) replaced by either
\[
 (\mathfrak T_{K,\mathcal F}^s,
   \mathfrak T_{K,\mathcal F}^{\mathrm{tl}})
 \quad\text{or}\quad
 (\mathfrak T_{K,\mathcal F}^{s,\mathrm{dual}},
   \mathfrak T_{K,\mathcal F}^{\mathrm{tl},\mathrm{dual}}).
\]
The constant depends only on \(n,\alpha\), the fixed function
\(\eta_0\), and the biLipschitz constants of \(\Omega\) and, in the
conformal case, \(\Psi\).  It is independent of
\(\varepsilon,R,\mu,\nu\), and \(\mathcal F\).
\end{lemma}

\begin{proof}
Both kernels satisfy \(|K(x,y)|\leq C|x-y|^{-d}, \quad x\neq y\). For the conformal kernel this follows from \(|\Psi(y)-\Psi(x)|\geq \Lip(\Psi^{-1})^{-1}|x-y|\). Set \(r=|x-y|\).  On
\(\varepsilon\leq r\leq R\), both truncation multipliers are one.
If \(0<r<\varepsilon\), then \(\left| \bigl(t_{\varepsilon,R}(x,y) -s_{\varepsilon,R}(x,y)\bigr)K(x,y) \right| \leq C\varepsilon^{-d}\). Indeed, the tangent-line term is bounded by
\(C\rho_{\varepsilon,R}^{\alpha}(r)\leq C\varepsilon^{-d}\).
The smooth term vanishes for \(r\leq\varepsilon/2\), while for
\(\varepsilon/2<r<\varepsilon\) it is bounded by
\(Cr^{-d}\leq C\varepsilon^{-d}\).

Let \(A_d=\max\{2,1+d^{-1}\}\). For \(r>R\), both truncated kernels vanish when \(r\geq A_dR\), and
on \(R<r<A_dR\) their difference is bounded by \(CR^{-d}\).
It follows that
\[
 \left|
 \bigl(t_{\varepsilon,R}-s_{\varepsilon,R}\bigr)(x,y)K(x,y)
 \right|
 \leq C\varepsilon^{-d}\one_{\{0<|x-y|<\varepsilon\}}
 +CR^{-d}\one_{\{R<|x-y|<A_dR\}}.
\]

For the first term, write
\[
 \{0<|x-y|\leq\varepsilon\}
 =
 \bigcup_{j=0}^{\infty}
 \left\{
 2^{-j-1}\varepsilon<|x-y|
 \leq2^{-j}\varepsilon
 \right\}.
\]
On the \(j\)-th annulus, \(\varepsilon^{-d} = 2^{-jd}(2^{-j}\varepsilon)^{-d}\). The quasicube form of \cref{lem:annular-A2} therefore gives a bound
\[
C\sum_{j=0}^{\infty}2^{-jd}
 \sqrt{\mathfrak A_{2,\Omega}^\alpha(\mu,\nu)}
 \leq
C\sqrt{\mathfrak A_{2,\Omega}^\alpha(\mu,\nu)}.
\]
The outer region \(R<|x-y|<A_dR\) is the union of a number of dyadic
annuli depending only on \(d\).  Another application of
\cref{lem:annular-A2} gives the same bound for the outer term.

To use the quasicube version of that lemma, take the annuli in the
preimage distance \(|\Omega^{-1}x-\Omega^{-1}y|\). The biLipschitz inequalities for \(\Omega\) place each Euclidean
annulus above in a fixed number of preimage annuli with comparable
radii.  The factors \(2^{-jd}\) in the inner sum change only by a
fixed constant.

For a vector kernel, take the Euclidean norm before applying the
positive annular majorants.  This proves the vector-valued operator
estimate.  For every \(j\),
\[
 |K_j^{\alpha,n}(x,y)|\leq|\bfK^{\alpha,n}(x,y)|,
 \qquad
 |K_{\Psi,j}^{\alpha,n}(x,y)|
 \leq|\bfK_\Psi^{\alpha,n}(x,y)|.
\]
The same annular majorants therefore prove the scalar estimates, with
constants independent of \(j\).  Since the multipliers are symmetric
in \(x,y\), the same proof applies to the vector or scalar adjoint after
the measures are interchanged.

Let \(E\in\mathcal F\).  Apply the operator estimate to
\(\one_E\), restrict the result to \(E\), and divide by
\(\mu(E)^{1/2}\).  This proves the forward testing comparison.
Interchanging \(\mu,\nu\) proves the dual comparison.  This argument
applies to every family of bounded Borel sets, including
\(\Omega\)-quasicubes.

\end{proof}

We omit the superscript \(\mathrm{tl}\) from the tangent-line operator
and testing constants in the rest of the paper.

\begin{proposition}
\label{prop:change-variable}
Let \(\Psi\) be as in \eqref{eq:graph-change}.  Assume that \(\Psi\) is
globally biLipschitz.  Let \(\sigma\) and \(\omega\) be locally finite
positive Borel measures with arbitrary supports, possibly with common
point masses.  Let \(\mathcal F\) be a family of bounded
Borel sets.  Then
\begin{align*}
 &\mathfrak N_{\bfR^{\alpha,n}}(\sigma,\omega)
   +\sqrt{\mathfrak A_2^\alpha(\sigma,\omega)}
\approx
 \mathfrak N_{\bfR_\Psi^{\alpha,n}}
   (\widetilde\sigma,\widetilde\omega)
   +\sqrt{\mathfrak A_2^\alpha(\sigma,\omega)}.
\end{align*}
The forward and dual testing constants over
\(\mathcal F\) and \(\Psi^{-1}\mathcal F\) satisfy the corresponding
comparisons.  All constants are uniform in the truncation parameters.
For any \(r_0>0\) with \(2r_0\leq1\), put
\begin{equation*}
 \Lambda_{\psi,r_0}
 :=\vartheta_\psi(2r_0)
 +\int_0^{2r_0}\vartheta_\psi(t)\,\frac{dt}{t}
 +\frac{\|\psi\|_\infty}{r_0}.
\end{equation*}
The Taylor and far-field kernel errors are bounded by
\(C\Lambda_{\psi,r_0}\sqrt{\mathfrak A_2^\alpha}\).  The cutoff
comparisons add \(C\sqrt{\mathfrak A_2^\alpha}\).  Thus the full
additive error is bounded by
\(C(1+\Lambda_{\psi,r_0})\sqrt{\mathfrak A_2^\alpha}\).
In these additive estimates, \(C\) depends only on \(n,\alpha\),
the biLipschitz constant of \(\Psi\), and \(\|D\psi\|_\infty\).
The constants in the norm and testing comparisons may also depend on
\(\Lambda_{\psi,r_0}\).
\end{proposition}

\begin{proof}
Put \(p=n+1-\alpha\).  The pullback of the Riesz kernel is \(\bfK^{\mathrm{pb}}(x,y) := c_{\alpha,n} \frac{\Psi(y)-\Psi(x)} {|\Psi(y)-\Psi(x)|^p}\). Fix \(r_0>0\) with \(2r_0\leq1\), and choose \(\chi\) as in
\eqref{eq:fixed-near-cutoff}.
Write \(\chi(x,y)=\chi(|x-y|)\).
All kernel identities below are for \(x\neq y\).  As above, every
truncated kernel is defined to be zero on the diagonal.

Since \(\Psi=I+u\), we have
\begin{align}
 \bfK^{\mathrm{pb}}(x,y)
 &=\chi(x,y)D\Psi(x)\bfK_\Psi^{\alpha,n}(x,y)
 +(1-\chi(x,y))\bfK_\Psi^{\alpha,n}(x,y)
 +\bfE_{x,\mathrm{near}}(x,y)+\bfE_{\mathrm{far}}(x,y),
 \label{eq:near-decomposition}
\end{align}
where
\begin{align*}
 \bfE_{x,\mathrm{near}}(x,y)
 &=c_{\alpha,n}\chi(x,y)
 \frac{u(y)-u(x)-Du(x)(y-x)}
 {|\Psi(y)-\Psi(x)|^p},\\
 \bfE_{\mathrm{far}}(x,y)
 &=c_{\alpha,n}(1-\chi(x,y))
 \frac{u(y)-u(x)}{|\Psi(y)-\Psi(x)|^p}.
\end{align*}
Expansion at \(y\) gives the analogous identity with \(D\Psi(y)\) and
\[
 \bfE_{y,\mathrm{near}}(x,y)
 =c_{\alpha,n}\chi(x,y)
 \frac{u(y)-u(x)-Du(y)(y-x)}{|\Psi(y)-\Psi(x)|^p}.
\]
Since \(\Psi\) is biLipschitz, \cref{lem:Taylor-Dini} gives, for
\(z=x,y\),
\begin{align*}
 |\bfE_{z,\mathrm{near}}(x,y)|
 &\leq C|x-y|^{\alpha-n}\vartheta_\psi(|x-y|)
 \one_{\{|x-y|\leq2r_0\}},\\
 |\bfE_{\mathrm{far}}(x,y)|
 &\leq C\|u\|_\infty |x-y|^{\alpha-n-1}
 \one_{\{|x-y|\geq r_0\}}.
\end{align*}

For \(j\in\mathbb Z\) with \(j\leq1\), the coefficient of
\(|x-y|^{\alpha-n}\) in the near error on the annulus \(2^{j-1}r_0<|x-y|\leq2^jr_0\) is at most \(C\vartheta_\psi(2^jr_0)\).  For \(j\in\mathbb Z\) with
\(j\geq0\), the corresponding coefficient in the far error is at most
\(C\|u\|_\infty/(2^jr_0)\).  Hence
\cref{lem:dyadic-Dini,lem:annular-A2} gives, for \(z=x,y\),
\begin{align}
 &\left\|\int q_{\varepsilon,R}(x,y)
 \bigl(\bfE_{z,\mathrm{near}}(x,y)+\bfE_{\mathrm{far}}(x,y)\bigr)
f(y)\,d\widetilde\sigma(y)\right\|_{L^2(\widetilde\omega)}
\leq
C\Lambda_{\psi,r_0}\sqrt{\mathfrak A_2^\alpha}
 \|f\|_{L^2(\widetilde\sigma)}.
 \label{eq:uniform-error-operator}
\end{align}
The constant is independent of the supports of the measures and of
\(\varepsilon,R\).
Here we use the quasicube form of \cref{lem:annular-A2} with
\(\Omega=\Psi^{-1}\).  A Euclidean annulus is contained in a fixed
number of \(\Psi^{-1}\)-quasiannuli at comparable radii.  The combined
Muckenhoupt constant is then replaced by the original one using
\eqref{eq:A2-pullback}.

More generally, if \(F\) and \(\Omega\) are globally biLipschitz, then
the combined pullback estimate is
\begin{equation}\label{eq:A2-pullback}
 \mathfrak A_{2,F^{-1}\circ\Omega}^\alpha
   (F^*\sigma,F^*\omega)
 \approx
 \mathfrak A_{2,\Omega}^\alpha(\sigma,\omega).
\end{equation}
For $H=\Omega Q$ and $\widetilde H=F^{-1}H$, the masses of
corresponding sets agree, and $\widetilde H^c=F^{-1}(H^c)$.
Moreover, $c_{\widetilde H}=F^{-1}(c_H)$ and
$|\widetilde H|^{1/n}\approx|H|^{1/n}$.  BiLipschitz continuity
therefore compares the reproducing Poisson kernels after the change
of variables.  This proves the forward and dual one-tailed comparisons
for these corresponding families.  The common atoms and their masses
also correspond, so the two punctured constants compare as well.
Taking the suprema gives \eqref{eq:A2-pullback}.

By \eqref{eq:uniform-radial-norms}, the cutoff in
\(\mathfrak N_{\bfK^{\mathrm{pb}}}^q\) is radial in the original
variables \(X=\Psi(x)\) and \(Y=\Psi(y)\).  The cutoff in
\(\mathfrak N_{\bfK^{\mathrm{pb}}}^s\) uses the distance \(|x-y|\).
Let \(L\geq1\) be a common biLipschitz constant for
\(\Psi\) and \(\Psi^{-1}\).  Then
\begin{align}
 \{q_{\varepsilon,R}\neq s_{\varepsilon,R}\}
 &\subset
 \left\{\frac{\varepsilon}{2L}<|x-y|<L\varepsilon\right\}
 \cup
 \left\{\frac{R}{L}<|x-y|<2LR\right\}.
 \label{eq:target-source-cutoff-support}
\end{align}
Indeed, both inner cutoff factors vanish below
\(\varepsilon/(2L)\) and equal one above \(L\varepsilon\).  Both outer
cutoff factors equal one below \(R/L\) and vanish above \(2LR\).

Since \(|\bfK^{\mathrm{pb}}(x,y)|\leq C|x-y|^{\alpha-n}\), each set on the right side of
\eqref{eq:target-source-cutoff-support} is the union of at most
\(C_L\) dyadic annuli.  Hence
\cref{lem:annular-A2} and \eqref{eq:A2-pullback} give
\begin{align}
 &\sup_{0<\varepsilon<R}
 \left\|\int
 (q_{\varepsilon,R}-s_{\varepsilon,R})(x,y)
 \bfK^{\mathrm{pb}}(x,y)f(y)\,
d\widetilde\sigma(y)\right\|_{L^2(\widetilde\omega)}
 \leq
C_L\sqrt{\mathfrak A_2^\alpha(\sigma,\omega)}
 \|f\|_{L^2(\widetilde\sigma)}.
 \label{eq:target-source-truncation}
\end{align}
The same estimate holds for the adjoint and for both normalized testing
constants.

We also need a fixed upper cutoff.  Put \(\rho=\min\{R,r_0\}\).  If
\(\varepsilon<\rho\), then \(s_{\varepsilon,R}(x,y)\chi(|x-y|) -s_{\varepsilon,\rho}(x,y)\) is supported in \(\{r_0\leq|x-y|\leq2r_0\}\).  If
\(\varepsilon\geq\rho\), then \(\rho=r_0\), and
\(s_{\varepsilon,R}\chi\) is supported in
\(\{\rho/2<|x-y|<2\rho\}\).  Another application of
\cref{lem:annular-A2} gives the required bound for the source-radial
cutoff.  The difference between that cutoff and
\(q_{\varepsilon,R}\) is controlled by
\eqref{eq:target-source-truncation}.  More explicitly, if \(T_a\)
denotes the operator with kernel
\(a(x,y)\bfK^{\mathrm{pb}}(x,y)\), then, uniformly in
\(0<\varepsilon<R\),
\begin{align*}
 \|T_{q_{\varepsilon,R}\chi}\|
 &\leq \|T_{s_{\varepsilon,R}\chi}\|
     +C\sqrt{\mathfrak A_2^\alpha}
 \leq \mathfrak N_{\bfK^{\mathrm{pb}}}^s
     +C\sqrt{\mathfrak A_2^\alpha}
 \leq \mathfrak N_{\bfK^{\mathrm{pb}}}^q
     +C\sqrt{\mathfrak A_2^\alpha}.
\end{align*}
The first and last inequalities use
\eqref{eq:target-source-truncation}.  The middle inequality uses the
two cases above and \cref{lem:annular-A2}.  Consequently,
\begin{align}
 &\sup_{0<\varepsilon<R}
 \left\|\int
q_{\varepsilon,R}(x,y)\chi(|x-y|)
 \bfK^{\mathrm{pb}}(x,y)f(y)\,
d\widetilde\sigma(y)\right\|_{L^2(\widetilde\omega)}
 \notag\\
 &\qquad\leq C\left(
 \mathfrak N_{\bfK^{\mathrm{pb}}}^q
 +\sqrt{\mathfrak A_2^\alpha(\sigma,\omega)}\right)
 \|f\|_{L^2(\widetilde\sigma)}.
 \label{eq:fixed-upper-pullback}
\end{align}
This estimate also holds for the adjoint.  Repeating the preceding three
inequalities with each operator norm replaced by its normalized testing
expression proves the forward and dual testing estimates.

Multiplying \eqref{eq:near-decomposition} by
\(q_{\varepsilon,R}\), we obtain
\begin{align*}
q_{\varepsilon,R}\bfK^{\mathrm{pb}}
 &=D\Psi(x)q_{\varepsilon,R}\chi\bfK_\Psi^{\alpha,n}
 +q_{\varepsilon,R}(1-\chi)\bfK_\Psi^{\alpha,n}+q_{\varepsilon,R}
 (\bfE_{x,\mathrm{near}}+\bfE_{\mathrm{far}}).
\end{align*}
The matrix \(D\Psi(x)\) depends only on the left variable.  Thus
we apply the quasicube form of \cref{lem:cutoff-stability} with
\(\Omega=\Psi^{-1}\).  Equation \eqref{eq:A2-pullback} bounds its
combined quasicube constant by
\(C\mathfrak A_2^\alpha(\sigma,\omega)\).  The lemma and
\eqref{eq:uniform-error-operator} give
\begin{equation}\label{eq:one-direction-change}
 \mathfrak N_{\bfK^{\mathrm{pb}}}^q
 \leq C\left(
 \mathfrak N_{\bfK_\Psi^{\alpha,n}}^s
 +(1+\Lambda_{\psi,r_0})
   \sqrt{\mathfrak A_2^\alpha}\right).
\end{equation}
By \eqref{eq:smooth-tangent-comparison}, the source-radial norm on the
right is bounded by the tangent-line norm plus
\(C\sqrt{\mathfrak A_2^\alpha}\).  Repeating the preceding inequalities
with each operator norm replaced by its normalized testing expression
proves forward testing.  Expansion at \(y\) proves dual testing because
\(D\Psi(y)\) then depends only on the left variable for the adjoint.
The error in both estimates is
\(C(1+\Lambda_{\psi,r_0})\sqrt{\mathfrak A_2^\alpha}\).  In both
testing estimates, the left-hand side uses the \(q\)-cutoffs, while the
first term on the right-hand side uses the \(s\)-cutoffs.  Equation
\eqref{eq:smooth-tangent-comparison} gives the testing
constants.

For the reverse inequality, we use the following identity for a graph
shear: \(Du(x)^2=0\). Consequently, \((I-\chi(x,y)Du(x))(I+\chi(x,y)Du(x))=I\). Equation \eqref{eq:near-decomposition} is \(\bfK^{\mathrm{pb}} =(I+\chi Du(x))\bfK_\Psi^{\alpha,n} +\bfE_{x,\mathrm{near}}+\bfE_{\mathrm{far}}\). It follows that
\begin{align}
 \bfK_\Psi^{\alpha,n}
 &=\bfK^{\mathrm{pb}}-\chi Du(x)\bfK^{\mathrm{pb}}
 -(I-\chi Du(x))
 (\bfE_{x,\mathrm{near}}+\bfE_{\mathrm{far}}).
 \label{eq:reverse-kernel-identity}
\end{align}
Multiply \eqref{eq:reverse-kernel-identity} by
\(q_{\varepsilon,R}\).  The first term is one of the truncations used in
\(\mathfrak N_{\bfK^{\mathrm{pb}}}^q\).  For the second term,
\[
q_{\varepsilon,R}(x,y)\chi(|x-y|)
Du(x)\bfK^{\mathrm{pb}}(x,y)
 =Du(x)\bigl[q_{\varepsilon,R}(x,y)\chi(|x-y|)
 \bfK^{\mathrm{pb}}(x,y)\bigr].
\]
Equations \eqref{eq:target-source-truncation} and
\eqref{eq:fixed-upper-pullback} therefore give
\[
 \|T_{q_{\varepsilon,R}\chi Du\,\bfK^{\mathrm{pb}}}\|
 \leq C\|Du\|_\infty\left(
 \mathfrak N_{\bfK^{\mathrm{pb}}}^q
 +\sqrt{\mathfrak A_2^\alpha}\right).
\]
For the last term, multiplication by \(\chi\) only decreases the
positive annular majorants used in
\eqref{eq:uniform-error-operator}.  Hence
\[
 \|T_{q_{\varepsilon,R}(I-\chi Du)
 (\bfE_{x,\mathrm{near}}+\bfE_{\mathrm{far}})}\|
 \leq C(1+\|Du\|_\infty)
 \Lambda_{\psi,r_0}\sqrt{\mathfrak A_2^\alpha}.
\]
The norms in the last two displays are from
\(L^2(\widetilde\sigma)\) to \(L^2(\widetilde\omega)\).
It follows that
\begin{equation}\label{eq:reverse-q-change}
 \mathfrak N_{\bfK_\Psi^{\alpha,n}}^q
 \leq C\left(
 \mathfrak N_{\bfK^{\mathrm{pb}}}^q
 +(1+\Lambda_{\psi,r_0})
   \sqrt{\mathfrak A_2^\alpha}\right).
\end{equation}
The same estimates, with the same factor
\(1+\Lambda_{\psi,r_0}\), hold for the normalized \(q\)-testing
constants.
Expansion at \(y\) gives the reverse dual \(q\)-testing estimate.

These estimates use the target-radial cutoffs
\(q_{\varepsilon,R}\).  Since \(|\bfK_\Psi^{\alpha,n}(x,y)| \leq C|x-y|^{\alpha-n}\), the support comparison
\eqref{eq:target-source-cutoff-support} and
\cref{lem:annular-A2} also give
\begin{align*}
 &\sup_{0<\varepsilon<R}
 \left\|\int
 (q_{\varepsilon,R}-s_{\varepsilon,R})(x,y)
 \bfK_\Psi^{\alpha,n}(x,y)f(y)\,
d\widetilde\sigma(y)\right\|_{L^2(\widetilde\omega)}\leq C_L\sqrt{\mathfrak A_2^\alpha}
 \|f\|_{L^2(\widetilde\sigma)}.
\end{align*}
The same estimate holds for the adjoint and for both normalized testing
constants.  Combining it with \eqref{eq:reverse-q-change} gives
\begin{equation}\label{eq:reverse-source-change}
 \mathfrak N_{\bfK_\Psi^{\alpha,n}}^s
 \leq C\left(
 \mathfrak N_{\bfK^{\mathrm{pb}}}^q
 +(1+\Lambda_{\psi,r_0})
   \sqrt{\mathfrak A_2^\alpha}\right),
\end{equation}
and the analogous forward and dual \(s\)-testing estimates, with the
same factor \(1+\Lambda_{\psi,r_0}\).  Equation
\eqref{eq:smooth-tangent-comparison} converts the left-hand sides to the
tangent-line constants.

Finally, under \(X=\Psi(x)\), \(Y=\Psi(y)\), the target-radial kernel \(\tau_{\varepsilon,R}(|X-Y|)\bfK^{\alpha,n}(X,Y)\) pulls back to \(q_{\varepsilon,R}(x,y)\bfK^{\mathrm{pb}}(x,y)\). The measures become \(\widetilde\sigma\) and
\(\widetilde\omega\), the source function becomes \(f\circ\Psi\), and
an original testing set \(Q\) becomes \(\Psi^{-1}Q\).  In particular, \(\|f\circ\Psi\|_{L^2(\widetilde\sigma)} =\|f\|_{L^2(\sigma)}\). Therefore
\begin{align*}
 \mathfrak N_{\bfK^{\alpha,n}}^s(\sigma,\omega)
 &=\mathfrak N_{\bfK^{\mathrm{pb}}}^q
   (\widetilde\sigma,\widetilde\omega),\\
 \mathfrak T_{\bfK^{\alpha,n},\mathcal F}^s(\sigma,\omega)
 &=\mathfrak T_{\bfK^{\mathrm{pb}},\Psi^{-1}\mathcal F}^q
   (\widetilde\sigma,\widetilde\omega),\\
 \mathfrak T_{\bfK^{\alpha,n},\mathcal F}^{s,\mathrm{dual}}
   (\sigma,\omega)
 &=\mathfrak T_{\bfK^{\mathrm{pb}},\Psi^{-1}\mathcal F}^{q,\mathrm{dual}}
   (\widetilde\sigma,\widetilde\omega).
\end{align*}
Apply \eqref{eq:smooth-tangent-comparison} to the fractional Riesz
kernel in the original coordinates and to the conformal kernel in the
pullback coordinates.  Together with
\eqref{eq:one-direction-change} and
\eqref{eq:reverse-source-change}, these identities prove the norm and
testing comparisons in the proposition.
\end{proof}

\begin{remark}\label{rem:Dini-exact}
The small-distance error is summable by the Dini integral.  The
large-distance error is summable because \(u\) is bounded and gains one
power of \(|x-y|^{-1}\).  Both estimates hold for measures with
unbounded support.  For
the extension in \cref{lem:graph-extension}, choose
\(r_0\approx|I^*|\) after a fixed dilation.  Then
\(\|\psi\|_\infty/r_0\leq C_\kappa\varepsilon_0\).
\end{remark}

\section{Energy and biLipschitz maps}\label{sec:4}

We prove that strong energy is preserved by a globally biLipschitz change
of variables.  We will then apply the energy-side-condition theorem to the
original smooth Riesz kernel.

We first recall the strong energy constant used in
\cite[Definition~5]{EnergyCommon}.  Fix the goodness parameters
\(r\in\mathbb N\) and \(0<\varepsilon<1\).  For a dyadic quasigrid
\(\Omega\mathcal D\), define \(\qdist_\Omega(E,F) :=\dist\bigl(\Omega^{-1}E,\Omega^{-1}F\bigr)\). For \(J\in\Omega\mathcal D\) and a quasicube \(K\), write
\(J\Subset_{r,\varepsilon}K\) if
\begin{equation}\label{eq:deep-embedding-definition}
J\subset K,\qquad
 \ell(J)\leq2^{-r}\ell(K),\qquad
 \qdist_\Omega(J,\partial K)
 \geq\frac12\ell(J)^\varepsilon\ell(K)^{1-\varepsilon}.
\end{equation}
The family \(M_{r,\varepsilon\text{-deep}}(K)\) consists of the
maximal \(J\in\Omega\mathcal D\) satisfying
\eqref{eq:deep-embedding-definition}.

An alternate quasicube \(K\in A\Omega\mathcal D\) is a union of
\(2^n\) quasicubes \(K'\in\Omega\mathcal D\), each with side length
\(\ell(K)/2\).  Denote these constituents by
\(\mathcal C_{\Omega\mathcal D}(K)\), and denote by \(\pi^\ell K'\)
the \(\ell\)-th dyadic ancestor of \(K'\).  For integers \(\ell\geq0\), define
\begin{align*}
 &M_{r,\varepsilon\text{-deep},\Omega\mathcal D}^{\ell}(K)
:=\left\{
J\in M_{r,\varepsilon\text{-deep}}(\pi^\ell K'):
 \begin{array}{l}
K'\in\mathcal C_{\Omega\mathcal D}(K),\\[2pt]
J\subset L\text{ for some }
L\in M_{r,\varepsilon\text{-deep}}(K)
 \end{array}
 \right\}.
\end{align*}
These are the families in
\cite[Section~2.6 and Definition~5]{EnergyCommon}.  The map \(\Omega\)
is fixed in the following definition; each supremum over
\(\Omega\mathcal D\) means that the ordinary dyadic grid
\(\mathcal D\) varies over all dyadic grids.  Then
\begin{align}
 \bigl(\cE_{\alpha,\Omega}^{\mathrm{strong}}
       (\sigma,\omega)\bigr)^2
 &:=
 \sup_{\Omega\mathcal D}
 \sup_{\substack{I\in\Omega\mathcal D,\ \sigma(I)>0,\\
I=\dot\bigcup_s I_s,\ I_s\in\Omega\mathcal D}}
 \frac1{\sigma(I)}
 \sum_s\sum_{J\in
M_{r,\varepsilon\text{-deep}}(I_s)}
 \left(
  \frac{P^\alpha(J,\one_I\sigma)}{|J|^{1/n}}
 \right)^2
 \|P_J^\omega\mathbf x\|_{L^2(\omega)}^2
 \notag\\
 &\quad+
 \sup_{\Omega\mathcal D}
 \sup_{\substack{I\in A\Omega\mathcal D\\\sigma(I)>0}}
 \sup_{\ell\geq0}
 \frac1{\sigma(I)}
 \sum_{J\in
M_{r,\varepsilon\text{-deep},\Omega\mathcal D}^{\ell}(I)}
 \left(
  \frac{P^\alpha(J,\one_I\sigma)}{|J|^{1/n}}
 \right)^2
 \|P_J^\omega\mathbf x\|_{L^2(\omega)}^2.
 \label{eq:strong-energy-definition}
\end{align}
In the first term, the subpartition may be finite or countable.  In the
second term,
\(I\in A\Omega\mathcal D\) is an alternate quasicube and
\(\ell\geq0\) is the refinement index.  These are the two expressions in
\cite[Definition~5]{EnergyCommon}.
The dual constant is obtained by interchanging \(\sigma\) and
\(\omega\).
We choose \(\varepsilon\) sufficiently small and \(r\) sufficiently
large, depending only on \(n\) and \(\alpha\), as required in
\cite[Theorem~2]{GoodLambda}.  In Section~5 we may increase \(r\) once
more so that \eqref{eq:gamma-goodness} holds.

Let \(F\) and \(\Omega\) be globally biLipschitz maps.  If
\(Q=\Omega Q_0\) is an \(\Omega\)-quasicube, put
\[
 \widetilde\Omega=F^{-1}\circ\Omega,
 \qquad
 \widetilde Q=F^{-1}Q=\widetilde\Omega Q_0.
\]
The inherited side length and center are
defined by \(\ell(\widetilde Q)=\ell(Q)=\ell(Q_0), \quad c_{\widetilde Q}=F^{-1}(c_Q)\). The dyadic children and descendants of $\widetilde Q$ are the preimages of
the corresponding dyadic cubes.
If \(G\) is a fixed globally biLipschitz map and \(Q=GC\), then \(c_G\ell(Q) \leq |Q|^{1/n} \leq C_G\ell(Q), \quad \ell(Q):=\ell(C)\). Thus the denominator \(|J|^{1/n}\) in
\eqref{eq:strong-energy-definition} is the Lebesgue-volume normalization in
\cite[Definition~5]{EnergyCommon}.  It may be
replaced by the inherited quasicube side length at the cost of constants
depending only on \(G\).

Write $\mathbb E_J^\mu f=\mu(J)^{-1}\int_J f\,d\mu$ when
$\mu(J)>0$, and zero otherwise.

\begin{lemma}\label{lem:variance}
Let $\mu$ be a positive Borel measure and let $J$ be a dyadic cube or
quasicube.  Then
\begin{equation}\label{eq:variance}
 \|P_J^\mu\mathbf x\|_{L^2(\mu)}^2
 =\int_J|x-m_J^\mu|^2\,d\mu(x)
 =\frac1{2\mu(J)}
   \iint_{J\times J}|x-y|^2\,d\mu(x)d\mu(y),
\end{equation}
with the usual value $0$ when $\mu(J)=0$.
\end{lemma}

\begin{proof}
The projection \(P_J^\mu\) is the sum of the weighted Haar projections
over all dyadic subquasicubes \(Q\subseteq J\), including \(J\) itself.
Hence \(P_J^\mu\mathbf x=\one_J(\mathbf x-m_J^\mu)\). The second equality in \eqref{eq:variance} follows by expanding
$|x-y|^2$ and using
$m_J^\mu=\mu(J)^{-1}\int_Jx\,d\mu(x)$.
\end{proof}

\begin{proposition}
\label{prop:energy-invariance}
Let \(F,\Omega:\R^n\to\R^n\) be globally biLipschitz.  Put \(\widetilde\Omega=F^{-1}\circ\Omega,\quad \widetilde\sigma=F^*\sigma,\quad \widetilde\omega=F^*\omega\). Then
\begin{equation}\label{eq:energy-equivalence}
 \cE_{\alpha,\widetilde\Omega}^{\mathrm{strong}}
   (\widetilde\sigma,\widetilde\omega)
 \approx
 \cE_{\alpha,\Omega}^{\mathrm{strong}}
   (\sigma,\omega),
\end{equation}
and the same comparison holds for the dual strong energy constants.  The
constants depend only on $n$, $\alpha$, and the biLipschitz constants of
$F$.
\end{proposition}

\begin{proof}
For every ordinary dyadic grid \(\mathcal D\), \(F^{-1}(\Omega\mathcal D)=\widetilde\Omega\mathcal D\). We compare the two lines of
\eqref{eq:strong-energy-definition}.  Deep embedding is preserved because \(\qdist_{\widetilde\Omega} (F^{-1}E,F^{-1}G) =\qdist_{\Omega}(E,G)\). The families
$M_{r,\varepsilon\text{-deep}}(I)$ and
$M_{r,\varepsilon\text{-deep}}^\ell(I)$ therefore correspond under \(F\).
Alternate cubes, the auxiliary cubes in the refined energy condition, and
dyadic subpartitions also correspond under \(F\), by their definitions.
More precisely,
\begin{align*}
F^{-1}(A\Omega\mathcal D)
 &=A\widetilde\Omega\mathcal D,\\
 \{F^{-1}J:J\in M_{r,\varepsilon\text{-deep}}(I)\}
 &=M_{r,\varepsilon\text{-deep}}(F^{-1}I),\\
 \{F^{-1}J:J\in
M_{r,\varepsilon\text{-deep},\Omega\mathcal D}^{\ell}(I)\}
 &=M_{r,\varepsilon\text{-deep},
 \widetilde\Omega\mathcal D}^{\ell}(F^{-1}I).
\end{align*}
Hence the indices in both lines of
\eqref{eq:strong-energy-definition} correspond under $F$.

Let $J\subset I$ correspond to
$\widetilde J\subset\widetilde I$.  For a quasicube \(Q\), put
\(a_Q=|Q|^{1/n}\).  Recall that
\begin{align*}
P^\alpha(Q,\nu)
 &=\int_{\R^n}
 \frac{a_Q}{(a_Q+|x-c_Q|)^{n+1-\alpha}}\,d\nu(x),\\
 \cP^\alpha(Q,\nu)
 &=\int_{\R^n}
 \left(\frac{a_Q}{(a_Q+|x-c_Q|)^2}\right)^{n-\alpha}
d\nu(x).
\end{align*}
The strong energy condition uses \(P^\alpha\).

Let \(L\) be a common biLipschitz constant for \(F\) and \(F^{-1}\).
Then \(c_{\widetilde J}=F^{-1}(c_J),\quad L^{-1}|J|^{1/n} \leq|\widetilde J|^{1/n} \leq L|J|^{1/n}\), and \(L^{-1}|u-c_J| \leq|F^{-1}(u)-F^{-1}(c_J)| \leq L|u-c_J|\). Using \eqref{eq:pullback-measure}, we obtain
\begin{align*}
P^\alpha(\widetilde J,
        \one_{\widetilde I}\widetilde\sigma)
 &=
 \int_I
 \frac{|\widetilde J|^{1/n}}
 {\left(|\widetilde J|^{1/n}
 +|F^{-1}(u)-F^{-1}(c_J)|\right)^{n+1-\alpha}}
 \,d\sigma(u)\\
 &\approx
 \int_I
 \frac{|J|^{1/n}}
 {\left(|J|^{1/n}+|u-c_J|\right)^{n+1-\alpha}}
 \,d\sigma(u)\\
 &=P^\alpha(J,\one_I\sigma).
\end{align*}
Consequently,
\begin{equation}\label{eq:Poisson-comparison}
 \frac{P^\alpha(\widetilde J,\one_{\widetilde I}\widetilde\sigma)}
      {|\widetilde J|^{1/n}}
 \approx
 \frac{P^\alpha(J,\one_I\sigma)}{|J|^{1/n}}.
\end{equation}
Replacing \(P^\alpha\) by \(\cP^\alpha\) in the preceding calculation
gives \(\cP^\alpha(\widetilde J, \one_{\widetilde I}\widetilde\sigma) \approx \cP^\alpha(J,\one_I\sigma)\).
The masses of corresponding sets are equal: \(\widetilde\sigma(\widetilde I)=\sigma(I), \quad \widetilde\omega(\widetilde J)=\omega(J)\). By
\cref{lem:variance},
\begin{align}
 \|P_{\widetilde J}^{\widetilde\omega}\mathbf x\|_{L^2(\widetilde\omega)}^2
 &=\frac1{2\omega(J)}
   \iint_{J\times J}
   |F^{-1}(u)-F^{-1}(v)|^2
   \,d\omega(u)d\omega(v)
 \notag\\
 &\approx
 \frac1{2\omega(J)}
   \iint_{J\times J}|u-v|^2\,d\omega(u)d\omega(v)
 \notag\\
 &=\|P_J^\omega\mathbf x\|_{L^2(\omega)}^2.
 \label{eq:variance-comparison}
\end{align}
Equations \eqref{eq:Poisson-comparison} and
\eqref{eq:variance-comparison}, followed by summation over the corresponding
deep cubes, compare the first line of
\eqref{eq:strong-energy-definition}.  Apply the same two comparisons to
each \(I\in A\Omega\mathcal D\) and each \(\ell\geq0\).  This compares
the second line.  Taking the suprema over the grids, subpartitions, alternate tops, and
\(\ell\geq0\) proves one inequality in \eqref{eq:energy-equivalence}.  For the opposite inequality, apply the comparison just
proved with \(F\) replaced by \(F^{-1}\), \(\Omega\) replaced by
\(F^{-1}\circ\Omega\), and the pulled-back pair in place of
\((\sigma,\omega)\).
Interchanging the measures gives the dual statement.
\end{proof}

\begin{remark}\label{rem:no-Dini-energy}
\Cref{prop:energy-invariance} requires only global biLipschitz
continuity.  It allows us to return the energy estimates to the original
coordinates, where the fractional Riesz kernel is smooth.
\end{remark}

\section{Energy reversal after straightening}\label{sec:5}

We prove the line-supported energy estimates under a \(C^1\)
small-slope assumption.  Let
\begin{equation}\label{eq:graph-map}
 \Psi(x_1,x')=(x_1,x'+\psi(x_1)),
 \qquad
 \|D\psi\|_\infty
 <\varepsilon_{n,\alpha}
 :=\frac{n-\alpha}{8n^2}.
\end{equation}

Write \(\Gamma\) for the curved support and \(L_0=\R e_1\) for the
supporting line in the straightened coordinates.  Let
\(J=\Omega Q\) be a quasicube for a fixed
globally biLipschitz map \(\Omega\).  Constants in this section may depend
on the biLipschitz constant of \(\Omega\).  In this section, testing and
Muckenhoupt constants are computed over the \(\Omega\)-quasicubes,
and we omit the subscript \(\Omega\).  Choose
\(A_0=A_0(n,\Omega)\) large enough for
\cref{lem:end-side-geometry}.  Next choose \(\gamma\) large enough for
\cref{lem:end-side-geometry,lem:end-reversal,lem:C1-energy,lem:line-geometry,lem:backward-reversal}.
The strict margin in \eqref{eq:graph-map}, together with
\(C_\Omega/\gamma\to0\), ensures that one such choice of \(\gamma\)
works uniformly for all slopes allowed in \eqref{eq:graph-map}.
Finally, take the goodness parameter \(r\) so large that
\begin{equation}\label{eq:gamma-goodness}
 \gamma\leq c_n2^{(1-\varepsilon)r}.
\end{equation}
If \(J=\Omega Q\), its dilation means \(\gamma J=\Omega(\gamma Q)\), where \(\gamma Q\) has the same center as \(Q\) and side length
\(\gamma\ell(Q)\).
A family \(\mathcal J\) is called admissible in \(I\) if
\begin{equation}\label{eq:admissible-family}
 \gamma J\subset I\quad(J\in\mathcal J),
 \qquad
 \sum_{J\in\mathcal J}\one_{\gamma J}\leq\beta\one_I.
\end{equation}
The constant \(\beta\) is fixed.

\begin{lemma}\label{lem:admissible-families}
Each family occurring in the two lines of
\eqref{eq:strong-energy-definition} is admissible.  More precisely, this
holds for
\[
 \mathcal J=\bigcup_s
M_{r,\varepsilon\text{-deep}}(I_s),
 \qquad I=\dot\bigcup_s I_s,
\]
and for
\[
 \mathcal J=
M_{r,\varepsilon\text{-deep},\Omega\mathcal D}^{\ell}(I),
 \qquad I\in A\Omega\mathcal D,\quad \ell\geq0.
\]
The overlap constant is independent of the grid and of \(\ell\).
\end{lemma}

\begin{proof}
By \eqref{eq:gamma-goodness}, \cite[(3.2) and (3.4)]{EnergyCommon}
gives the required inclusions and overlap estimates.  For the first
family, these estimates give
\[
 \gamma J\subset I_s,\qquad
 \sum_{J\in M_{r,\varepsilon\text{-deep}}(I_s)}
 \one_{\gamma J}\leq\beta\one_{I_s}.
\]
The \(I_s\) are disjoint, so summing in \(s\) gives
\eqref{eq:admissible-family}.  For an alternate top,
\cite[(3.4)]{EnergyCommon} gives
\[
 \gamma J\subset I,\qquad
 \sum_{J\in
M_{r,\varepsilon\text{-deep},\Omega\mathcal D}^{\ell}(I)}
 \one_{\gamma J}\leq\beta\one_I
\]
uniformly in the grid and in \(\ell\).
\end{proof}

\begin{lemma}\label{lem:C1-energy}
Let \(\varphi_J\) be an \(\R^n\)-valued function supported on \(J\),
and suppose that \(\int_J\varphi_J\,d\omega=0\) componentwise.  If a
finite positive Borel measure \(\nu\) is supported outside
\(\gamma J\), then
\begin{equation*}
 \left|\left\langle
 \bfR_\Psi^{\alpha,n}\nu,\varphi_J
 \right\rangle_\omega\right|
 \leq C\|\varphi_J\|_{L^2(\omega)}
 \frac{P^\alpha(J,\nu)}{\ell(J)}
 \|P_J^\omega\mathbf x\|_{L^2(\omega)}.
\end{equation*}
The constant depends only on \(n,\alpha,\Omega,\gamma\), and
\(\|D\psi\|_\infty\); it does not depend on a modulus of \(D\psi\).
\end{lemma}

\begin{proof}
Put \(p=n+1-\alpha\).  Direct differentiation gives
\begin{align*}
D_y\bfK_\Psi^{\alpha,n}(y,x)
 &=c_{\alpha,n}\left[
 -\frac{I}{|\Psi(y)-\Psi(x)|^p}
 +p\frac{(y-x)\otimes
D\Psi(y)^T(\Psi(y)-\Psi(x))}
 {|\Psi(y)-\Psi(x)|^{p+2}}\right].
\end{align*}
Here \(D_y\bfK_\Psi^{\alpha,n}(y,x)\) means the derivative of the map
\(y\mapsto\bfK_\Psi^{\alpha,n}(y,x)\).  By antisymmetry, it has the
opposite overall sign from \(D_y\bfK_\Psi^{\alpha,n}(x,y)\), where
\(y\) is the second kernel variable.
Thus \(|D_y\bfK_\Psi^{\alpha,n}(y,x)| \leq C|y-x|^{-p}\). If \(\omega(J)=0\), there is nothing to prove.  Otherwise put
\[
m_J^\omega=\frac1{\omega(J)}\int_Jy\,d\omega(y).
\]
Then \(m_J^\omega\in\overline{\operatorname{conv}}(J)\).  Hence, for
\(y\in J\) and \(\zeta\in[y,m_J^\omega]\), \(\zeta\in\overline{\operatorname{conv}}(J), \quad |\zeta-c_J| \leq\diam J \leq C_\Omega\ell(J)\). If \(x\notin\gamma J=\Omega(\gamma Q)\), then
\(\Omega^{-1}(x)\notin\gamma Q\), and therefore \(|x-c_J|\geq c_\Omega\gamma\ell(J)\). By the choice of \(\gamma\), these estimates give \(|\zeta-x|\approx_\Omega \ell(J)+|x-c_J|\). The fundamental theorem of
calculus and the mean-zero condition give
\begin{align*}
 \left|\left\langle
 \bfR_\Psi^{\alpha,n}\nu,\varphi_J
 \right\rangle_\omega\right|
 &\leq C\frac{P^\alpha(J,\nu)}{\ell(J)}
 \int_J|\varphi_J(y)|\,|y-m_J^\omega|\,d\omega(y)
 \leq C\frac{P^\alpha(J,\nu)}{\ell(J)}
 \|\varphi_J\|_{L^2(\omega)}
 \|P_J^\omega\mathbf x\|_{L^2(\omega)}.
\end{align*}
No difference \(D\psi(s)-D\psi(t)\) is used.
\end{proof}

\subsection{Forward energy}

Let \(\omega\) be supported on \(L_0\).  Recall the fixed choices of
\(A_0\) and \(\gamma\) above.  For \(J^*=\gamma J\), define
\begin{align*}
 \mathrm E(J^*)
 &=(I\setminus J^*)\cap
 \left\{y:|y'-c_J'|
 \leq\frac{A_0}{\gamma}|y^1-c_J^1|\right\},\\
 \mathrm S(J^*)&=(I\setminus J^*)\setminus\mathrm E(J^*).
\end{align*}
The replacement of the fixed number \(10\) in
\cite[proof of Theorem~13]{CurvePaper} by \(A_0(\Omega)\) is needed
for an arbitrary
biLipschitz quasicube map.

\begin{lemma}\label{lem:end-side-geometry}
There are \(A_0=A_0(n,\Omega)\) and
\(\gamma_0=\gamma_0(n,\Omega,A_0)\) such that, for
\(\gamma\geq\gamma_0\), the following holds whenever
\(J\cap L_0\neq\varnothing\).

If \(y\in\mathrm E(J^*)\), \(x,z\in J\cap L_0\), and \(u\) lies between
\(x^1-y^1\) and \(z^1-y^1\), then
\begin{equation}\label{eq:end-geometry}
 |u|\approx\ell(J)+|y-c_J|,\qquad
 |y'|\leq\frac{CA_0}{\gamma}|u|,\qquad
 |x^1-z^1|\leq\frac C\gamma|u|.
\end{equation}
If \(y\in\mathrm S(J^*)\), then, with \(d(y)=|y'|\),
\begin{align}
d(y)&>0,\qquad
 \ell(J)\leq Cd(y),\qquad
 |y-c_J|\geq c\,d(y),
 \label{eq:side-geometry}\\
 \mathcal H^1(J^*\cap L_0)&\geq c\gamma\ell(J),\qquad
J^*\cap L_0\subset
 \{(t,0'):|t-y^1|\leq C\gamma d(y)\}.
 \label{eq:side-shadow}
\end{align}
\end{lemma}

\begin{proof}
Let \(\Lambda\) be a common biLipschitz constant for
\(\Omega,\Omega^{-1}\).  There are \(a,b>0\), depending only on
\(n,\Lambda\), such that \(J\subset B(c_J,b\ell(J)),\quad B(x_0,a(\gamma-1)\ell(J))\subset\gamma J \quad(x_0\in J)\). Choose \(A_0\geq8b/a\), and then
\(\gamma\geq\max\{2A_0,8b/a\}\).

Suppose \(y\in\mathrm E(J^*)\).  Since \(y\notin\gamma J\),
\(|y-c_J|\geq c\gamma\ell(J)\).  Put \(A=|y'-c_J'|,\quad B=|y^1-c_J^1|\). The end condition gives \(A\leq A_0B/\gamma\), and hence
\(B\approx|y-c_J|\).  Points of \(J\cap L_0\) have first coordinate within
\(b\ell(J)\) of \(c_J^1\).  This proves
\eqref{eq:end-geometry}.

Now let \(y\in\mathrm S(J^*)\), and fix \(x_0\in J\cap L_0\).  The side
condition gives \(B<\frac{\gamma}{A_0}A \leq\frac{\gamma}{A_0}\{d(y)+b\ell(J)\}\). Since \(y\notin\gamma J\), \(a(\gamma-1)\ell(J) \leq|y-x_0|\leq B+b\ell(J)+d(y)\). Our choices allow the terms containing \(\ell(J)\) on the right to be
absorbed on the left.  Thus \(d(y)\geq cA_0\ell(J)\), and
\(|y-c_J|\geq c\,d(y)\).  We also obtain
\(|y^1-x_0^1|\leq C\gamma d(y)\).  Intersecting
\(B(x_0,a(\gamma-1)\ell(J))\) with \(L_0\) gives the lower bound for
\(\mathcal H^1(J^*\cap L_0)\).  Finally,
\(\diam J^*\leq C\gamma\ell(J)\leq C\gamma d(y)\), which gives
\eqref{eq:side-shadow}.
\end{proof}

\begin{lemma}\label{lem:end-reversal}
For \(x,z\in J\cap L_0\) with \(x\neq z\), or equivalently
\(x^1\neq z^1\),
\begin{align}
 &\left|
 \frac{R_{\Psi,1}^{\alpha,n}
  (\one_{\mathrm E(J^*)}\sigma)(x)
 -R_{\Psi,1}^{\alpha,n}
  (\one_{\mathrm E(J^*)}\sigma)(z)}
 {x^1-z^1}\right|
 \approx
 \frac{P^\alpha(J,\one_{\mathrm E(J^*)}\sigma)}
 {\ell(J)}.
 \label{eq:end-reversal}
\end{align}
\end{lemma}

\begin{proof}
Fix \(y=(y^1,y')\in\mathrm E(J^*)\).  For
\(u\) between \(x^1-y^1\) and \(z^1-y^1\), put \(a(u)=\psi(u+y^1)-\psi(y^1)-y', \quad \varphi(u)=u(u^2+|a(u)|^2)^{-p/2}\), where \(p=n+1-\alpha\).  Direct differentiation gives
\begin{align}
 \varphi'(u)
 &=(u^2+|a(u)|^2)^{-p/2-1}
 \times\left\{
 |a(u)|^2-pu\,a(u)\cdot D\psi(u+y^1)
 -(p-1)u^2\right\}.
 \label{eq:end-exact-derivative}
\end{align}
By \cref{lem:end-side-geometry}, \(|u|\approx\ell(J)+|y-c_J|, \quad |a(u)|\leq \left(\|D\psi\|_\infty+\frac{CA_0}{\gamma}\right)|u|\). The bound in \eqref{eq:graph-map}, followed by the choice of \(\gamma\)
above, ensures that \(q_0^2+pq_0\|D\psi\|_\infty\leq\frac{p-1}{2}, \quad q_0=\|D\psi\|_\infty+\frac{CA_0}{\gamma}\). It follows from \eqref{eq:end-exact-derivative} that \(c|u|^{-p}\leq-\varphi'(u)\leq C|u|^{-p}\). The sign is independent of \(u\) and \(y\).  The fundamental theorem of
calculus, followed by integration in \(y\), proves
\eqref{eq:end-reversal}.
\end{proof}

\begin{lemma}\label{lem:forward-hole}
If \(\mathcal J\) is admissible, then
\begin{align}
 &\sum_{J\in\mathcal J}
 \left(
 \frac{P^\alpha(J,\one_{I\setminus\gamma J}\sigma)}
 {\ell(J)}\right)^2
 \|P_J^\omega\mathbf x\|_{L^2(\omega)}^2
 \leq C\beta
 \left[
 \bigl(\mathfrak T_{\bfR_\Psi^{\alpha,n}}\bigr)^2
 +\mathcal A_2^\alpha
 \right]\sigma(I).
 \label{eq:forward-hole}
\end{align}
\end{lemma}

\begin{proof}
Terms with \(\omega(J)=0\) vanish.  Since \(\omega\) is supported on
\(L_0\), the variance of \(\mathbf x\) on \(J\) is the variance of the
first coordinate.  Put
\[
a_J=\frac{P^\alpha(J,\one_{\mathrm E(J^*)}\sigma)}
 {\ell(J)},\qquad
b_J=\frac{P^\alpha(J,\one_{\mathrm S(J^*)}\sigma)}
 {\ell(J)}.
\]
Since \(I\setminus J^*=\mathrm E(J^*)\dot\cup\mathrm S(J^*)\),
linearity of \(P^\alpha\) and \((a_J+b_J)^2\leq2a_J^2+2b_J^2\)
reduce the proof to the two corresponding energy sums.
By \cref{lem:end-reversal} and the variance identity,
\[
a_J^2\|P_J^\omega\mathbf x\|_2^2
 \leq C\int_J
 \left|R_{\Psi,1}^{\alpha,n}
  (\one_{\mathrm E(J^*)}\sigma)
 -\mathbb E_J^\omega R_{\Psi,1}^{\alpha,n}
 (\one_{\mathrm E(J^*)}\sigma)\right|^2d\omega.
\]
First take a finite subfamily
\(\mathcal J_0\subset\mathcal J\).  Terms for which
\(I\setminus J^*=\varnothing\) have zero Poisson factor and may be
discarded.  If no term remains, there is nothing to prove.  Otherwise
put
\[
d_0=\min_{J\in\mathcal J_0}\dist(J,I\setminus J^*),
 \qquad
D_0=\max_{J\in\mathcal J_0}
 \sup_{\substack{x\in J\\y\in I\setminus J^*}}|x-y|.
\]
Since \(J^*=\gamma J\), quasicube geometry gives \(d_0\geq c_\Omega(\gamma-1) \min_{J\in\mathcal J_0}\ell(J)>0, \quad D_0\leq\diam I\). Choose one tangent-line truncation with \(0<\delta<d_0,\quad R>D_0\). Then \(t_{\delta,R}=1\) on
\(J\times(I\setminus J^*)\) for every \(J\in\mathcal J_0\).
The half-open convention gives the partition \(I=J^*\dot\cup\mathrm E(J^*)\dot\cup\mathrm S(J^*)\). Consequently, for \(x\in J\),
\begin{align*}
 \bfR_\Psi^{\alpha,n}
   (\one_{\mathrm E(J^*)}\sigma)(x)
 &=
 \bfR_{\Psi,\delta,R}^{\alpha,n,\mathrm{tl}}
   (\one_I\sigma)(x)
 -\bfR_{\Psi,\delta,R}^{\alpha,n,\mathrm{tl}}
   (\one_{J^*}\sigma)(x)
 -\bfR_{\Psi,\delta,R}^{\alpha,n,\mathrm{tl}}
   (\one_{\mathrm S(J^*)}\sigma)(x).
\end{align*}
In the next three testing estimates,
each transform denotes this fixed tangent-line truncation.
After subtracting the \(\omega\)-average on \(J\), the triangle
inequality gives three oscillation terms.  Testing on \(I\) gives
\begin{align*}
 \sum_{J\in\mathcal J_0}\int_J
 |R_{\Psi,1}^{\alpha,n}(\one_I\sigma)|^2\,d\omega
 &\leq\beta\int_I
 |R_{\Psi,1}^{\alpha,n}(\one_I\sigma)|^2\,d\omega
 \leq\beta\mathfrak T_{\bfR_\Psi^{\alpha,n}}^2\sigma(I).
\end{align*}
Testing on each \(J^*\), followed by
\eqref{eq:admissible-family}, gives
\begin{align*}
 \sum_{J\in\mathcal J_0}\int_J
 |R_{\Psi,1}^{\alpha,n}(\one_{J^*}\sigma)|^2\,d\omega
 &\leq\mathfrak T_{\bfR_\Psi^{\alpha,n}}^2
       \sum_{J\in\mathcal J_0}\sigma(J^*)
 \leq\beta\mathfrak T_{\bfR_\Psi^{\alpha,n}}^2\sigma(I).
\end{align*}
For the side term, apply \cref{lem:C1-energy} to a vector test
function whose only nonzero component is the first one.  By duality,
\[
 \int_J\left|R_{\Psi,1}^{\alpha,n}
 (\one_{\mathrm S(J^*)}\sigma)
 -\mathbb E_J^\omega R_{\Psi,1}^{\alpha,n}
 (\one_{\mathrm S(J^*)}\sigma)\right|^2d\omega
 \leq Cb_J^2\|P_J^\omega\mathbf x\|_2^2.
\]
All bounds are uniform in the truncation.  Letting the finite
subfamilies \(\mathcal J_0\) increase to \(\mathcal J\), monotone
convergence applies to the nonnegative energy sums.  It remains to
estimate
\[
B=\sum_{J\in\mathcal J}
 \left(
 \frac{P^\alpha(J,\one_{\mathrm S(J^*)}\sigma)}
 {\ell(J)}\right)^2
 \|P_J^\omega\mathbf x\|_2^2.
\]
Write
\[
 \cP^\alpha(J,\nu)
 \approx_\Omega\int_{\R^n}
 \left(\frac{\ell(J)}
 {(\ell(J)+|y-c_J|)^2}\right)^{n-\alpha}d\nu(y).
\]
Suppose first that \(n-1\leq\alpha<n\).  Since
\(\mathrm S(J^*)\subset J^c\), \(P^\alpha(J,\one_{\mathrm S(J^*)}\sigma) \leq C\cP^\alpha(J,\one_{\mathrm S(J^*)}\sigma)\). Using
\(\|P_J^\omega\mathbf x\|_2^2\leq C\ell(J)^2\omega(J)\) and the
one-tailed condition gives
\begin{align*}
 &\left(
 \frac{P^\alpha(J,\one_{\mathrm S(J^*)}\sigma)}
 {\ell(J)}\right)^2
 \|P_J^\omega\mathbf x\|_2^2\leq C\mathcal A_2^\alpha
 \int_{\mathrm S(J^*)}
 \left(\frac{\ell(J)}
 {\ell(J)+|y-c_J|}\right)^{n+1-\alpha}d\sigma(y).
\end{align*}
If \(0\leq\alpha<n-1\), Cauchy--Schwarz gives
\begin{align*}
P^\alpha(J,\one_{\mathrm S(J^*)}\sigma)^2
 &\leq C_\Omega
 \cP^\alpha(J,\one_{\mathrm S(J^*)}\sigma)
 \int_{\mathrm S(J^*)}
 \frac{\ell(J)^{\alpha+2-n}}
 {(\ell(J)+|y-c_J|)^2}\,d\sigma(y).
\end{align*}
Using
\(\|P_J^\omega\mathbf x\|_2^2\leq C\ell(J)^2\omega(J)\) and \(\omega(J)\cP^\alpha(J,\one_{\mathrm S(J^*)}\sigma) \leq C\mathcal A_2^\alpha\ell(J)^{n-\alpha}\), we obtain
\begin{align*}
 &\left(
 \frac{P^\alpha(J,\one_{\mathrm S(J^*)}\sigma)}
 {\ell(J)}\right)^2
 \|P_J^\omega\mathbf x\|_2^2\leq C\mathcal A_2^\alpha
 \int_{\mathrm S(J^*)}
 \left(\frac{\ell(J)}
 {\ell(J)+|y-c_J|}\right)^2d\sigma(y).
\end{align*}
Thus, with \(q_\alpha=\min\{2,n+1-\alpha\}>1\),
\begin{equation*}
B\leq C\mathcal A_2^\alpha
 \int_I\sum_{\substack{J\in\mathcal J\\y\in\mathrm S(J^*)}}
 \left(\frac{\ell(J)}
 {\ell(J)+|y-c_J|}\right)^{q_\alpha}d\sigma(y).
\end{equation*}
By \cref{lem:end-side-geometry}, if \(d(y)=\dist(y,L_0)\) and
\(y\in\mathrm S(J^*)\), then \eqref{eq:side-geometry} and
\eqref{eq:side-shadow} hold.
Hence, by bounded overlap,
\[
 \sum_{\substack{J\in\mathcal J\\y\in\mathrm S(J^*)}}\ell(J)
 \leq\frac C\gamma
 \int_{|t-y^1|\leq C\gamma d(y)}
 \sum_{J\in\mathcal J}\one_{J^*}(t,0)\,dt
 \leq C\beta d(y).
\]
Moreover, \eqref{eq:side-geometry} gives
\(\ell(J)\leq Cd(y)\) and
\(|y-c_J|\geq c\,d(y)\).  Since \(1<q_\alpha\leq2\),
the number \(t_J=\ell(J)/d(y)\) satisfies \(t_J\leq C\), and hence
\(t_J^{q_\alpha}\leq C^{q_\alpha-1}t_J\).  Therefore
\begin{align*}
 &\sum_{\substack{J\in\mathcal J\\y\in\mathrm S(J^*)}}
 \left(\frac{\ell(J)}
 {\ell(J)+|y-c_J|}\right)^{q_\alpha}\leq
C\sum_{\substack{J\in\mathcal J\\y\in\mathrm S(J^*)}}
 \left(\frac{\ell(J)}{d(y)}\right)^{q_\alpha}
 \leq
 \frac C{d(y)}
 \sum_{\substack{J\in\mathcal J\\y\in\mathrm S(J^*)}}\ell(J)
 \leq C\beta.
\end{align*}
Thus
\(B\leq C\beta\mathcal A_2^\alpha\sigma(I)\), which proves
\eqref{eq:forward-hole}.
\end{proof}

\begin{lemma}\label{lem:fill-hole}
For every admissible family,
\begin{align}
 &\sum_{J\in\mathcal J}
 \left(
 \frac{P^\alpha(J,\one_{\gamma J}\sigma)}
 {\ell(J)}\right)^2
 \|P_J^\omega\mathbf x\|_2^2
 \leq
C\beta A_2^{\alpha,*,\mathrm{punct}}\sigma(I).
 \label{eq:fill-hole}
\end{align}
After interchanging the measures, the dual estimate is
\begin{align*}
 &\sum_{J\in\mathcal J}
 \left(
 \frac{P^\alpha(J,\one_{\gamma J}\omega)}
 {\ell(J)}\right)^2
 \|P_J^\sigma\mathbf x\|_2^2
 \leq C\beta A_2^{\alpha,\mathrm{punct}}\omega(I).
\end{align*}
\end{lemma}

\begin{proof}
Let \(\mathfrak P\) be the set of common point masses, and write
\[
 \omega(Q,\mathfrak P)
 =\omega(Q)-\sup_{p\in Q\cap\mathfrak P}\omega(\{p\}),
\]
where the supremum of the empty set is \(0\).  If
\(J^*\cap\mathfrak P\neq\varnothing\), choose a largest common
\(\omega\)-atom \(p_J\).  Such an atom exists because
\(\omega(J^*)<\infty\): if the positive supremum of the common atomic
masses were not attained, infinitely many of them would exceed half of
that supremum.
If \(p_J\in J\), the minimizing property of variance gives
\[
 \|P_J^\omega\mathbf x\|_2^2
 \leq\int_J|x-p_J|^2\,d\omega(x)
 \leq C\ell(J)^2\omega(J^*,\mathfrak P).
\]
If \(p_J\notin J\), then
the half-open sets \(J\) and \(\{p_J\}\) are disjoint.  Since \(p_J\)
is a largest common \(\omega\)-atom in \(J^*\), \(\omega(J) \leq \omega\bigl(J^*\setminus\{p_J\}\bigr) =\omega(J^*,\mathfrak P)\). The minimizing property of the barycenter, followed by the quasicube
diameter bound, now gives
\begin{align*}
 \|P_J^\omega\mathbf x\|_2^2
 &\leq\int_J|x-c_J|^2\,d\omega(x)\leq C\ell(J)^2\omega(J)
 \leq C\ell(J)^2\omega(J^*,\mathfrak P).
\end{align*}
When \(J^*\cap\mathfrak P\) is empty,
\(\omega(J^*,\mathfrak P)=\omega(J^*)\).  Centering at \(c_J\) gives
\[
 \|P_J^\omega\mathbf x\|_2^2
 \leq C\ell(J)^2\omega(J)
 \leq C\ell(J)^2\omega(J^*)
 =C\ell(J)^2\omega(J^*,\mathfrak P).
\]
Hence \(\|P_J^\omega\mathbf x\|_2^2 \leq C\ell(J)^2\omega(J^*,\mathfrak P)\). Also, \(P^\alpha(J,\one_{J^*}\sigma) \leq C\frac{\sigma(\gamma J)}{\ell(J)^{n-\alpha}}\). Thus the left side of \eqref{eq:fill-hole} is at most
\[
C\sum_{J\in\mathcal J}
 \frac{\sigma(J^*)^2\omega(J^*,\mathfrak P)}
 {\ell(J)^{2(n-\alpha)}}.
\]
Since the puncture is in the \(\omega\)-factor, we use the starred
punctured condition.  Since
\(\ell(J^*)=\gamma\ell(J)\), it gives
\[
 \frac{\sigma(J^*)\omega(J^*,\mathfrak P)}
 {\ell(J)^{2(n-\alpha)}}
 \leq C_\gamma A_2^{\alpha,*,\mathrm{punct}}.
\]
Bounded overlap now gives
\[
C_\gamma A_2^{\alpha,*,\mathrm{punct}}
 \sum_{J\in\mathcal J}\sigma(J^*)
 \leq C_\gamma\beta
A_2^{\alpha,*,\mathrm{punct}}\sigma(I).
\]
Interchanging \(\sigma\) and \(\omega\) punctures the
\(\sigma\)-factor.  This gives the unstarred constant in the dual
estimate stated above.
\end{proof}

\subsection{Backward energy}

We now estimate the dual energy.

\begin{lemma}
\label{lem:line-geometry}
There are constants \(c_\Omega,C_\Omega>0\) with the following
properties.
\begin{enumerate}
\item If \(3J\cap L_0=\varnothing\), \(y\in J\), and \(x\in L_0\), then \(|y'|\geq c_\Omega\ell(J),\quad |y-x|\approx_\Omega\ell(J)+|x-c_J|\). \item Suppose \(3J\cap L_0\neq\varnothing\).  If \(\gamma\) is large,
\(x\in L_0\setminus\gamma J\), \(y,z\in J\), and
\(w_t=z+t(y-z)\), then
\begin{align*}
 |w_t'|&\leq C_\Omega\ell(J),\\
 |w_t^1-x^1|&\geq c_\Omega\gamma\ell(J),\\
 |w_t'|&\leq \frac{C_\Omega}{\gamma}|w_t^1-x^1|,\\
 |w_t^1-x^1|&\approx_\Omega\ell(J)+|x-c_J|.
\end{align*}
\end{enumerate}
\end{lemma}

\begin{proof}
Write \(J=\Omega Q\).  The biLipschitz bounds imply \(\diam J\leq C_\Omega\ell(J),\quad \dist(J,(\gamma J)^c)\geq c_\Omega(\gamma-1)\ell(J)\), where the second inequality remains valid for half-open quasicubes after
decreasing \(c_\Omega\).  If \(3J\cap L_0=\varnothing\), then
\(L_0\subset(3J)^c\).  The preceding distance bound with \(\gamma=3\)
therefore gives \(\dist(J,L_0)\geq c_\Omega\ell(J)\).  This proves the
first assertion.

For the second, choose \(\zeta\in3J\cap L_0\).  Since
\(w_t\in\operatorname{conv}(J)\), we have
\(|w_t-c_J|\leq C_\Omega\ell(J)\), and hence
\(|w_t-\zeta|\leq C_\Omega\ell(J)\).  This gives the first bound.
Since \(x\notin\gamma J\),
\(|x-c_J|\geq c_\Omega\gamma\ell(J)\).  Therefore, after increasing
\(\gamma\), \(|x-w_t|\geq |x-c_J|-|w_t-c_J| \geq c_\Omega\gamma\ell(J)\). Subtracting the transverse bound gives the second and third bounds.
The estimates for \(|x-c_J|\) and \(|w_t-c_J|\) give the last
comparison.
\end{proof}

\begin{lemma}\label{lem:backward-reversal}
Assume that \(3J\cap L_0\neq\varnothing\), and let \(\nu\) be a finite
positive Borel measure supported on \(L_0\setminus\gamma J\).  Under
\eqref{eq:graph-map}, for the above choice of \(\gamma\) and every
\(y,z\in J\),
\begin{equation}\label{eq:backward-reversal}
 |y^1-z^1|\frac{P^\alpha(J,\nu)}{\ell(J)}
 \leq C|\bfR_\Psi^{\alpha,n}\nu(y)
       -\bfR_\Psi^{\alpha,n}\nu(z)|.
\end{equation}
\end{lemma}

\begin{proof}
Put \(p=n+1-\alpha\), \(q=p-1=n-\alpha\), and write \(h=y-z=(h_1,h'),\quad w_t=z+th\). For \(x=(x^1,0)\in L_0\setminus\gamma J\), set \(a=w_t^1-x^1,\quad u=w_t',\quad b=u+\psi(w_t^1)-\psi(x^1),\quad r=(a^2+|b|^2)^{1/2}\). BiLipschitz quasicube geometry gives \(|u|\leq\eta|a|,\quad |b|\leq\delta|a|,\quad r\approx|a|\approx\ell(J)+|x-c_J|\), where \(\eta=\frac{C_\Omega}{\gamma}, \quad \delta=\eta+\|D\psi\|_\infty\). The bound in \eqref{eq:graph-map}, followed by the choice of \(\gamma\)
above, ensures that
\begin{align}
 \delta^2+p\delta(1+\|D\psi\|_\infty)&\leq q/2,
 \label{eq:backward-choice-one}\\
p\eta\{1+\delta(1+\|D\psi\|_\infty)\}&\leq1/2.\notag
\end{align}

Suppose first that \(|h'|\leq|h_1|\).  If \(h_1=0\), then \(y=z\),
so the desired estimate is trivial.  We may therefore assume
\(h_1\neq0\).  For \(\phi_x(t)=a r^{-p}\), direct differentiation gives
\begin{align*}
 \phi_x'(t)
 &=\frac{h_1}{r^{p+2}}N_1(t,x),\\
N_1(t,x)
 &=-qa^2+|b|^2
 -pa\,b\cdot\frac{h'}{h_1}
 -pa\,b\cdot D\psi(w_t^1).
\end{align*}
By \eqref{eq:backward-choice-one}, \(N_1(t,x)\leq-\frac q2a^2\). The sign is fixed for every \(x,t\).  Integration first in \(t\), then
in \(x\), yields
\[
 \left|
 (R_{\Psi,1}^{\alpha,n}\nu)(y)
 -(R_{\Psi,1}^{\alpha,n}\nu)(z)
 \right|
 \geq c|h_1|\int\!\!\int_0^1\frac{a^2}{r^{p+2}}\,dt\,d\nu(x)
 \geq c|h_1|\frac{P^\alpha(J,\nu)}{\ell(J)}.
\]

Suppose next that \(|h'|>|h_1|\).  Differentiate the transverse vector:
\begin{align}
 \frac d{dt}(u r^{-p})
 &=h'r^{-p}+\mathbf e_x(t),\notag\\
 \mathbf e_x(t)
 &=-p\,u r^{-p-2}
 \left\{ah_1+b\cdot h'
 +b\cdot D\psi(w_t^1)h_1\right\},\notag\\
 |\mathbf e_x(t)|
 &\leq p\eta\{1+\delta(1+\|D\psi\|_\infty)\}
 |h'|r^{-p}
 \leq\frac12|h'|r^{-p}.
 \label{eq:backward-case-two}
\end{align}
The reverse triangle inequality gives
\begin{align*}
 \left|\int_{L_0}\int_0^1
 \bigl(h'r^{-p}+\mathbf e_x(t)\bigr)\,dt\,d\nu(x)\right|
 &\geq |h'|\int_{L_0}\int_0^1r^{-p}\,dt\,d\nu(x)
 -\int_{L_0}\int_0^1|\mathbf e_x(t)|\,dt\,d\nu(x)\\
 &\geq\frac{|h'|}{2}
 \int_{L_0}\int_0^1r^{-p}\,dt\,d\nu(x).
\end{align*}
Consequently,
\[
 \left|
 (\bfR_\Psi^{\alpha,n}\nu(y))'
 -(\bfR_\Psi^{\alpha,n}\nu(z))'
 \right|
 \geq c|h'|\frac{P^\alpha(J,\nu)}{\ell(J)}
 \geq c|h_1|\frac{P^\alpha(J,\nu)}{\ell(J)}.
\]
This proves \eqref{eq:backward-reversal}.  The calculation uses
\(|\psi(s)-\psi(t)|\leq\|D\psi\|_\infty|s-t|\) and the pointwise
bound for \(D\psi\), but no modulus of \(D\psi\).
\end{proof}

\begin{lemma}\label{lem:weak-backward}
Let \(\nu\) be a finite positive Borel measure supported on \(L_0\).
\begin{enumerate}
\item If \(3J\cap L_0=\varnothing\), then \(|\bfR_\Psi^{\alpha,n}\nu(y)|\geq cP^\alpha(J,\nu), \quad y\in J\). \item If \(3J\cap L_0\neq\varnothing\), \(\nu\) is supported outside
\(\gamma J\), and \(2\leq k\leq n\), then \(|(R_{\Psi,k}^{\alpha,n}\nu)(y)| \approx |y^k|\frac{P^\alpha(J,\nu)}{\ell(J)}\). \end{enumerate}
\end{lemma}

\begin{proof}
In the first case, \(y'\neq0\).  Use the point-dependent unit vector \(\xi(y)=\left(0,\frac{y'}{|y'|}\right)\). For \(x\in L_0\), \(\xi(y)\cdot(x-y)=-|y'|\), and
\(|y'|\geq c_\Omega\ell(J)\).  Moreover, \(|\Psi(y)-\Psi(x)|\approx|y-x| \approx_\Omega\ell(J)+|x-c_J|\). Taking the scalar product with \(\xi(y)\) and integrating in \(x\)
gives the first estimate.

For the second estimate, use \eqref{eq:conformal-kernel} and the fact
that \(x^k=0\) on \(L_0\).  Then
\[
 (R_{\Psi,k}^{\alpha,n}\nu)(y)
 =-c_{\alpha,n}y^k
 \int_{L_0\setminus\gamma J}
 \frac{d\nu(x)}{|\Psi(y)-\Psi(x)|^{n+1-\alpha}}.
\]
The integral is positive and comparable to
\(P^\alpha(J,\nu)/\ell(J)\).
\end{proof}

\begin{lemma}\label{lem:backward-hole}
For every admissible family,
\begin{align}
 &\sum_{J\in\mathcal J}
 \left(
 \frac{P^\alpha(J,\one_{I\setminus\gamma J}\omega)}
 {\ell(J)}\right)^2
 \|P_J^\sigma\mathbf x\|_2^2
\leq C\beta
 \bigl(\mathfrak T_{\bfR_\Psi^{\alpha,n}}^{\mathrm{dual}}\bigr)^2
 \omega(I).
 \label{eq:backward-hole}
\end{align}
\end{lemma}

\begin{proof}
Fix a finite subfamily \(\mathcal J_0\subset\mathcal J\).  Terms for
which \(I\setminus\gamma J=\varnothing\) have zero Poisson factor and
may be discarded.  If no term remains, there is nothing to prove.
Otherwise put
\[
d_0=\min_{J\in\mathcal J_0}
 \dist(J,I\setminus\gamma J),
 \qquad
D_0=\max_{J\in\mathcal J_0}
 \sup_{\substack{x\in J\\y\in I\setminus\gamma J}}|x-y|.
\]
Quasicube geometry gives \(d_0\geq c_\Omega(\gamma-1) \min_{J\in\mathcal J_0}\ell(J)>0, \quad D_0\leq\diam I\). Choose \(0<\delta<d_0,\quad R>D_0\). Then \(t_{\delta,R}=1\) on every
\(J\times(I\setminus\gamma J)\), \(J\in\mathcal J_0\).  In the sums
below, \(J\) ranges over \(\mathcal J_0\).  The estimates are uniform
in \(\mathcal J_0,\delta\), and \(R\).
Let \(\mathcal J_{\rm off}=\{J\in\mathcal J_0:3J\cap L_0=\varnothing\}\). For \(J\in\mathcal J_{\rm off}\), put \(\nu_J=\one_{I\setminus\gamma J}\omega\). The first part of \cref{lem:weak-backward} and
\(\|P_J^\sigma\mathbf x\|_2^2\leq C\ell(J)^2\sigma(J)\) give
\begin{align*}
 &\left(
 \frac{P^\alpha(J,\nu_J)}{\ell(J)}\right)^2
 \|P_J^\sigma\mathbf x\|_2^2
 \leq C\int_J|\bfR_\Psi^{\alpha,n}\nu_J|^2\,d\sigma.
\end{align*}
For \(x\in J\), the common truncation gives the pointwise identity
\[
 \bfR_\Psi^{\alpha,n}\nu_J(x)
 =
 \bfR_{\Psi,\delta,R}^{\alpha,n,\mathrm{tl}}
   (\one_I\omega)(x)
 -
 \bfR_{\Psi,\delta,R}^{\alpha,n,\mathrm{tl}}
   (\one_{\gamma J}\omega)(x).
\]
The testing transforms below use this fixed
tangent-line truncation.
Since \(J\subset\gamma J\) and
\(\sum_J\one_{\gamma J}\leq\beta\one_I\), summation gives
\begin{align*}
 \sum_{J\in\mathcal J_{\rm off}}
 \int_J|\bfR_\Psi^{\alpha,n}\nu_J|^2\,d\sigma
 &\leq2\beta\int_I
 |\bfR_\Psi^{\alpha,n}(\one_I\omega)|^2\,d\sigma+2\sum_{J\in\mathcal J_{\rm off}}\int_{\gamma J}
 |\bfR_\Psi^{\alpha,n}(\one_{\gamma J}\omega)|^2\,d\sigma
 \leq C\beta
 \bigl(\mathfrak T_{\bfR_\Psi^{\alpha,n}}^{\mathrm{dual}}\bigr)^2
 \omega(I).
\end{align*}
Here the last line uses dual testing on \(I\) and on every \(\gamma J\),
together with
\(\sum_{J\in\mathcal J_{\rm off}}\omega(\gamma J)
\leq\beta\omega(I)\).

For \(3J\cap L_0\neq\varnothing\), put \(\nu_J=\one_{I\setminus\gamma J}\omega, \quad a_J=\frac{P^\alpha(J,\nu_J)}{\ell(J)}\). If \(\sigma(J)=0\), the corresponding energy term is zero.  Suppose
that \(\sigma(J)>0\).  For the first coordinate, the variance identity
and \cref{lem:backward-reversal} give
\begin{align*}
a_J^2\int_J
   |y^1-(m_J^\sigma)^1|^2\,d\sigma(y)
 &=
 \frac{a_J^2}{2\sigma(J)}
 \int_J\int_J|y^1-z^1|^2\,d\sigma(y)\,d\sigma(z)\\
 &\leq
 \frac C{2\sigma(J)}
 \int_J\int_J
 |\bfR_\Psi^{\alpha,n}\nu_J(y)
  -\bfR_\Psi^{\alpha,n}\nu_J(z)|^2
 \,d\sigma(y)\,d\sigma(z)\\
 &=
C\int_J
 |\bfR_\Psi^{\alpha,n}\nu_J
  -\mathbb E_J^\sigma
   \bfR_\Psi^{\alpha,n}\nu_J|^2\,d\sigma\\
 &\leq
C\int_J|\bfR_\Psi^{\alpha,n}\nu_J|^2\,d\sigma.
\end{align*}
For \(2\leq k\leq n\), the minimizing property of the mean and the
second part of \cref{lem:weak-backward} give
\begin{align*}
a_J^2\int_J|y^k-(m_J^\sigma)^k|^2\,d\sigma(y)
 &\leq a_J^2\int_J|y^k|^2\,d\sigma(y)
 \leq C\int_J
 |R_{\Psi,k}^{\alpha,n}\nu_J(y)|^2\,d\sigma(y).
\end{align*}
Adding the coordinate estimates gives
\[
a_J^2\|P_J^\sigma\mathbf x\|_2^2
 \leq C\int_J|\bfR_\Psi^{\alpha,n}\nu_J|^2\,d\sigma.
\]
The same pointwise identity gives
\[
 \bfR_\Psi^{\alpha,n}\nu_J(x)
 =
 \bfR_{\Psi,\delta,R}^{\alpha,n,\mathrm{tl}}
   (\one_I\omega)(x)
 -
 \bfR_{\Psi,\delta,R}^{\alpha,n,\mathrm{tl}}
   (\one_{\gamma J}\omega)(x),
 \qquad x\in J.
\]
The first term is summed using bounded overlap and dual testing on \(I\).
For the second, use dual testing on \(\gamma J\) and
\[
 \sum_J\omega(\gamma J)\leq\beta\omega(I).
\]
The estimates are uniform in \(\mathcal J_0\) and in the common
truncation.  Letting \(\mathcal J_0\) increase to \(\mathcal J\) and
using monotone convergence for the nonnegative energy sums gives
\eqref{eq:backward-hole}.
\end{proof}

\begin{proposition}
\label{prop:C1-energy-reversal}
Let \(\psi\in C^1(\R;\R^{n-1})\) satisfy
\eqref{eq:graph-map}.  Let \(\omega\) be supported on \(L_0\).  For every
globally biLipschitz quasicube map \(\Omega\),
\begin{align}
 \cE_{\alpha,\Omega}^{\mathrm{strong}}(\sigma,\omega)
 &\leq C\left(
 \mathfrak T_{\bfR_\Psi^{\alpha,n},\Omega}
 +\sqrt{\mathfrak A_{2,\Omega}^\alpha}\right),
 \label{eq:forward-strong-line}\\
 \cE_{\alpha,\Omega}^{\mathrm{strong,dual}}(\sigma,\omega)
 &\leq C\left(
 \mathfrak T_{\bfR_\Psi^{\alpha,n},\Omega}^{\mathrm{dual}}
 +\sqrt{\mathfrak A_{2,\Omega}^\alpha}\right).
 \label{eq:backward-strong-line}
\end{align}
Up to constants depending on the fixed quasicube map, the two quantities
on the left are the strong energy constants in
\cite[Definition~5]{EnergyCommon}, since
\(|J|^{1/n}\approx_\Omega\ell(J)\).  The estimates include both lines of
\eqref{eq:strong-energy-definition}, and are uniform in the quasigrid and
in the refinement index \(\ell\).  No modulus of \(D\psi\) occurs.
\end{proposition}

\begin{proof}
By \cref{lem:admissible-families}, it is enough to estimate every
admissible family.  Split
\[
P^\alpha(J,\one_I\sigma)
 \leq P^\alpha(J,\one_{I\setminus\gamma J}\sigma)
 +P^\alpha(J,\one_{\gamma J}\sigma).
\]
\Cref{lem:forward-hole,lem:fill-hole} give
\eqref{eq:forward-strong-line}.  The same split with the measures
interchanged, followed by
\cref{lem:backward-hole} and the dual form of
\cref{lem:fill-hole}, gives \eqref{eq:backward-strong-line}.  This
proves the first deep-subpartition expression and, separately, the
alternate/refined expression for every \(\ell\).
\end{proof}

\section{Localization to one arc}\label{sec:6}

We continue to write \(L_0=\R e_1\) for the straight coordinate line.
The original curved support is denoted by \(\Gamma\).
All cubes in this section are half-open.  Let \(\mathcal Q^n\) be the
family of all axis-parallel half-open cubes.  For a biLipschitz
map \(\Omega\), put \(\Omega\mathcal Q^n=\{\Omega Q:Q\in\mathcal Q^n\}\). Testing is taken over the full family \(\Omega\mathcal Q^n\).
If \(P=\Omega C\), all dilations in this section are quasicube
dilations: \(3P:=\Omega(3C)\).
\subsection{A transverse family with connected intersections}

Put \(v=n^{-1/2}(1,\ldots,1)\), and choose a rotation \(R\) with
\(R^{-1}e_1=v\).  For the small graph maps used below, the map \(\Omega=\Psi^{-1}\circ R\) satisfies
\begin{equation}\label{eq:coordinate-monotonicity}
c_n\leq
 \frac d{dt}\bigl[\Omega^{-1}(te_1)\bigr]_k
 \leq C_n,
 \qquad 1\leq k\leq n.
\end{equation}
Indeed,
\(\Omega^{-1}(te_1)=R^{-1}\Psi(te_1)\), whose derivative is uniformly
close to \(R^{-1}e_1=v\).

\begin{lemma}\label{lem:axis-sections}
Assume \(n\geq2\) and \eqref{eq:coordinate-monotonicity}.
\begin{enumerate}
\item For every \(Q\in\Omega\mathcal Q^n\), \(Q\cap L_0\) is empty or a
half-open interval.
\item If \(S=[a,b)e_1\subset L_0\), there is
\(Q_S\in\Omega\mathcal Q^n\) such that \(Q_S\cap L_0=S\), and
\begin{equation}\label{eq:auxiliary-size}
 \ell(Q_S)\leq C_\Omega(b-a),\qquad
Q_S\cup3Q_S\subset
 \{x:\dist(x,S)\leq C_\Omega(b-a)\}.
\end{equation}
\end{enumerate}
\end{lemma}

\begin{proof}
Write \(\zeta(t)=\Omega^{-1}(te_1)\).  Every coordinate
\(\zeta_k\) is strictly increasing.  If
\(C=\prod_k[d_k,d_k+\ell)\), then
\[
 \{t:\zeta(t)\in C\}
 =\bigcap_k\{t:d_k\leq\zeta_k(t)<d_k+\ell\},
\]
an intersection of half-open intervals.  This proves the first part.

For the second, put \(h=b-a\), \(\ell=8C_nh\), and define
\[
C_S=\prod_{k=1}^n[d_k,d_k+\ell)
\]
by
\[
d_1=\zeta_1(a),\qquad
d_2=\zeta_2(b)-\ell,\qquad
d_k=\frac{\zeta_k(a)+\zeta_k(b)-\ell}{2}
 \quad(3\leq k\leq n).
\]
For \(a\leq t<b\),
\[
 0\leq\zeta_1(t)-d_1\leq C_nh<\ell,
 \qquad
 0<d_2+\ell-\zeta_2(t)\leq\ell,
\]
and \(\zeta_2(t)>d_2\).  For \(3\leq k\leq n\), \(\dist\bigl(\zeta_k([a,b]), \{d_k,d_k+\ell\}\bigr) \geq\frac{\ell-C_nh}{2} \geq\frac{7\ell}{16}\). Thus the path \(\zeta([a,b])\) meets only the first lower face at
\(\zeta(a)\) and the second upper face at \(\zeta(b)\).
Moreover, \(t<a\) is excluded by the first coordinate, while
\(t\geq b\) is excluded by the second coordinate.  Hence
\(\zeta(t)\in C_S\) exactly when \(a\leq t<b\).  Take
\(Q_S=\Omega C_S\).
The derivative bounds and the biLipschitz bounds for \(\Omega\) give
\eqref{eq:auxiliary-size}.
\end{proof}

\subsection{Nested arcs}

\begin{lemma}\label{lem:nested-arcs}
Let \(\Gamma\) be a compact regular \(C^1\) curve.  Assume that it is
simple: the parametrization is one-to-one on the interval for an open
curve and on the parameter circle for a closed curve.  There is
a finite family of graph charts and global graph extensions with the
following property.  For every \(A>1\), there is a finite disjoint Borel
decomposition
\[
 \Gamma=\dot\bigcup_{i=1}^N\Gamma_i
\]
and nested graph arcs \(\Gamma_i\Subset\widehat{\Gamma}_i \Subset\Gamma_i^*\) with the following properties.  Let \(\mathcal U_i\) be the rigid motion
for the \(i\)-th arc.  In every formula concerning this fixed chart, we
relabel the rigid images
\(\mathcal U_i\Gamma\),
\(\mathcal U_i\Gamma_i\),
\(\mathcal U_i\widehat{\Gamma}_i\), and
\(\mathcal U_i\Gamma_i^*\) as
\(\Gamma\), \(\Gamma_i\),
\(\widehat{\Gamma}_i\), and \(\Gamma_i^*\), respectively.
In these rigid coordinates, one of the graph maps in the fixed family has
the form \(\Psi_i(x_1,x')=(x_1,x'+\psi_i(x_1))\) and a half-open interval \(I_i\subset L_0\) such that \(\Psi_i^{-1}(\Gamma_i)\subset I_i\). Except for the last arc of an open curve, equality holds.  For that last
arc, \(I_i\) includes an additional terminal segment on which every
pulled-back measure supported on \(\Gamma\) has zero mass.

If \(\widetilde{\Gamma}^{(i)}=\Psi_i^{-1}(\Gamma)\) then
\begin{equation}\label{eq:tube-localization}
 \{x:\dist(x,I_i)\leq A|I_i|\}
 \cap\widetilde{\Gamma}^{(i)}
 \subset\Psi_i^{-1}(\widehat{\Gamma}_i)\subset L_0.
\end{equation}
The fixed graph slopes can be made smaller than any prescribed positive
constant.  In the \(C^{1,\mathrm{Dini}}\) case, the fixed maps can be
chosen as in \cref{lem:graph-extension}.  Thus their biLipschitz
constants and the constants in \cref{lem:axis-sections} are fixed before
\(A\).  Only the final subdivision into the arcs \(\Gamma_i\) depends
on \(A\).
\end{lemma}

\begin{proof}
Let \(M=[a,b]\) for an open curve.  For a closed curve, let \(M=\R/(b-a)\mathbb Z\) and regard \(\Phi\) as a one-to-one embedding of the parameter circle
\(M\).  We use the parameter distance on \(M\).

For every \(t_0\in M\), choose a rigid motion \(\mathcal U_{t_0}\) that
sends \(\Phi(t_0)\) to the origin and \(\Phi'(t_0)\) to \(e_1\).  After
shrinking the parameter
neighborhood, write \(\mathcal U_{t_0}\Phi(t) =\bigl(u_{t_0}(t),g_{t_0}(u_{t_0}(t))\bigr)\). The neighborhoods can be chosen so that, with one constant \(c_0>0\),
\begin{equation}\label{eq:parameter-coordinate-comparison}
c_0\leq u_{t_0}'(t)\leq1.
\end{equation}
The slopes of the functions \(g_{t_0}\) can be made as small as desired.
For two parameters \(s,t\) in the same neighborhood,
\begin{equation}\label{eq:parameter-coordinate-distance}
c_0d_M(s,t)
 \leq |u_{t_0}(s)-u_{t_0}(t)|
 \leq d_M(s,t).
\end{equation}

Choose finitely many triples \(V_\nu\Subset\widehat V_\nu\Subset V_\nu^*, \quad 1\leq\nu\leq N_0\), such that the \(V_\nu\) cover \(M\) and every \(V_\nu^*\) lies in one
of the graph neighborhoods above.  The intervals may be chosen with a
fixed relative collar.  By
\eqref{eq:parameter-coordinate-comparison}, the corresponding
\(u_\nu\)-intervals have a fixed relative collar as well.  Our choice of
rigid motion gives \(g_\nu(0)=0\), after relabeling the graph coordinate
so that \(u_\nu(t_0)=0\).  At an endpoint of an open curve, extend the
local graph by its tangent affine function to a slightly larger
\(u_\nu\)-interval.  This is a \(C^1\) extension.  If the original
derivative has a Dini modulus, then the derivative of this one-sided
extension has the same Dini integral, up to an absolute constant.

Extend every \(g_\nu\) by \cref{lem:graph-extension}, and denote the
resulting graph shear by \(\Psi_\nu\).  Put \(K=\max_{1\leq\nu\leq N_0} \max\{\Lip(\Psi_\nu),\Lip(\Psi_\nu^{-1})\}\). For each \(\nu\), let \(R_\nu\) be the diagonal rotation used in
\eqref{eq:coordinate-monotonicity}, and put \(\Omega_\nu=\Psi_\nu^{-1}\circ R_\nu\). The rotations are isometries.  Hence the biLipschitz constants of the
finitely many maps \(\Omega_\nu\) are bounded in terms of \(K\).  Define
\begin{equation}\label{eq:axis-family-constant}
C_{\rm ax}=\max_{1\leq\nu\leq N_0}C_{\Omega_\nu},
\end{equation}
where \(C_{\Omega_\nu}\) is the constant in
\eqref{eq:auxiliary-size}.  Thus \(K\) and \(C_{\rm ax}\) have been fixed
without using \(A\).

The rigid motions preserve distance.  Simplicity and compactness give \(d_\nu:=\dist\bigl(\Phi(\overline V_\nu), \Phi(M\setminus\widehat V_\nu)\bigr)>0\). Thus \(d_0:=\min_{1\leq\nu\leq N_0}d_\nu>0\).
Let \(\lambda>0\) be a Lebesgue number for the cover
\(\{V_\nu\}_{\nu=1}^{N_0}\).  The numbers \(d_0\) and \(\lambda\) also
depend only on the finite chart family.  Fix the prescribed number
\(A>1\), and choose \(0<h<\min\left\{\lambda,\frac{d_0}{4AK}\right\}\). Partition \(M\) into consecutive arcs \(J_i\) of parameter length at
most \(h\).  In the closed case these are cyclic half-open arcs.  In the
open case use half-open arcs, except that the last arc contains \(b\).
For each \(J_i\), choose \(\nu(i)\) such that
\(\overline J_i\subset V_{\nu(i)}\), and put
\[
 \Gamma_i=\Phi(J_i),\qquad
 \widehat{\Gamma}_i=\Phi(\widehat V_{\nu(i)}),\qquad
 \Gamma_i^*=\Phi(V_{\nu(i)}^*).
\]
Work in the rigid coordinates \(\mathcal U_{\nu(i)}\), and write
\(\Psi_i=\Psi_{\nu(i)}\).

If \(J_i\) is not the last arc of an open curve, define \(I_i=u_{\nu(i)}(J_i)e_1\). For the last open arc \(J_i=[t_{N-1},b]\), put
\(\beta_i=u_{\nu(i)}(b)\), choose \(0<\delta_i\leq \min\left\{h,\frac{d_0}{4K}\right\}\), and set \(I_i=[u_{\nu(i)}(t_{N-1}),\beta_i+\delta_i)e_1\). The added terminal segment is \(G_i=(\beta_i,\beta_i+\delta_i)e_1\).  No point of the local curve has first coordinate
larger than \(\beta_i\).  Every point of the remaining curve is at least
\(d_0\) from
\(\mathcal U_{\nu(i)}\Phi(\overline V_{\nu(i)})\), while every point of
\(\Psi_i(G_i)\) is within \(K\delta_i\) of
\(\mathcal U_{\nu(i)}\Phi(b)\).  Hence
\begin{equation}\label{eq:ghost-zero}
 \Psi_i(G_i)\cap\mathcal U_{\nu(i)}\Gamma=\varnothing.
\end{equation}
Moreover,
\[
 \dist\bigl(\Psi_i(G_i),
 \mathcal U_{\nu(i)}\Phi(M\setminus\widehat V_{\nu(i)})\bigr)
 \geq d_0-K\delta_i\geq\frac{d_0}{2}.
\]

For every \(i\), including the last open arc, we therefore have \(\dist\left(\Psi_i(I_i), \mathcal U_{\nu(i)}\Phi(M\setminus\widehat V_{\nu(i)})\right) \geq\frac{d_0}{2}\). Equation \eqref{eq:parameter-coordinate-distance} and the definition of
\(G_i\) give \(|I_i|\leq2h\).
Let \(x\in\widetilde{\Gamma}^{(i)} \setminus\Psi_i^{-1}(\widehat{\Gamma}_i)\). There is \(t\in M\setminus\widehat V_{\nu(i)}\) such that
\(\Psi_i(x)=\mathcal U_{\nu(i)}\Phi(t)\).  Hence
\begin{align*}
 \dist(x,I_i)
 &\geq K^{-1}
 \dist\bigl(\Psi_i(x),\Psi_i(I_i)\bigr)\geq\frac{d_0}{2K}
 >2Ah
 \geq A|I_i|.
\end{align*}
This proves the first inclusion in
\eqref{eq:tube-localization}.  On \(\widehat V_{\nu(i)}\), the extended
graph agrees with the original graph.  Therefore \(\Psi_i^{-1}(\widehat{\Gamma}_i) =u_{\nu(i)}(\widehat V_{\nu(i)})e_1\subset L_0\).
For a closed curve, each partition point belongs to the cyclic interval
that starts there and not to the interval that ends there.  For an open
curve, the first endpoint belongs to the first interval and the last
endpoint belongs to the last interval.  Equation \eqref{eq:ghost-zero}
shows that \(G_i\) is disjoint from the pulled-back support.

For each arc \(J_i\) in the final partition, set
\(R_i=R_{\nu(i)}\) and
\(\Omega_i=\Omega_{\nu(i)}=\Psi_i^{-1}\circ R_i\).  If the
prescribed number satisfies \(A>2C_{\rm ax}\), then the auxiliary quasicubes lie in the tube in
\eqref{eq:tube-localization}.  Indeed, if \(S\subset I_i\) and
\(P_S=Q_S\) is the quasicube from \cref{lem:axis-sections}, then
\begin{align*}
 3P_S
 &\subset\{x:\dist(x,S)\leq C_{\Omega_i}|S|\}\\
 &\subset\{x:\dist(x,I_i)\leq C_{\rm ax}|I_i|\}
 \subset\{x:\dist(x,I_i)\leq A|I_i|\}.
\end{align*}
Consequently, \(3P_S\cap\widetilde{\Gamma}^{(i)} \subset\Psi_i^{-1}(\widehat{\Gamma}_i)\subset L_0\). This is the form used in \cref{prop:localized-testing}.
\end{proof}

\subsection{Tripled testing and restriction}

\begin{lemma}\label{lem:Hardy}
Let \(\lambda,\tau\) be locally finite positive Borel measures on
\([0,\infty)\), and let
\(s>0\).  Assume
\begin{equation}\label{eq:Hardy-no-common-zero}
 \lambda(\{0\})\tau(\{0\})=0.
\end{equation}
At \(r=0\), if \(\tau(\{0\})>0\), interpret
\(r^{-s}\lambda([0,r])\) as zero and replace \([r,\infty)\) by
\((0,\infty)\) in the second integrand.  Then
\begin{align*}
 \int_0^\infty
 [r^{-s}\lambda([0,r])]^2\,d\tau(r)
 &\leq4B_1\lambda([0,\infty)),\\
 \int_0^\infty
 \left[\int_{[r,\infty)}t^{-s}\,d\lambda(t)\right]^2d\tau(r)
 &\leq4B_2\lambda([0,\infty)),
\end{align*}
where
\begin{align*}
B_1&=\sup_{\rho>0}\lambda([0,\rho])
 \int_{[\rho,\infty)}r^{-2s}\,d\tau(r), \ \
B_2=\sup_{\rho>0}\tau([0,\rho])
 \int_{[\rho,\infty)}t^{-2s}\,d\lambda(t).
\end{align*}
\end{lemma}

\begin{proof}
These are the \(p=2\) direct and dual two-weight Hardy inequalities for
positive Borel measures; see \cite{Hardy}.  We include the
Stieltjes-measure proof.  Consider
\[
Hf(r)=r^{-s}\int_{[0,r]}f(t)\,d\lambda(t).
\]
Put \(F(r)=\lambda([0,r]),\quad d\nu(r)=r^{-2s}d\tau(r),\quad G(r)=\nu([r,\infty))\). Here \(\nu\) is defined on \((0,\infty)\); a possible atom at zero is
treated at the end of the proof.
Thus \(F(r)G(r)\leq B_1\).  The elementary Stieltjes estimates
\begin{align}
 \int_{[0,r]}F(t)^{-1/2}\,d\lambda(t)
 &\leq2F(r)^{1/2},
 \label{eq:Stieltjes-increasing}\\
 \int_{[t,\infty)}G(r)^{-1/2}\,d\nu(r)
 &\leq2G(t)^{1/2}
 \label{eq:Stieltjes-decreasing}
\end{align}
follow from the distribution functions.  Indeed, for $0<A=F(r)<\infty$,
\[
 \begin{split}
 \int_{[0,r]}F(t)^{-1/2}\,d\lambda(t)
 &=A^{1/2}+\frac12\int_0^A u^{-3/2}
   \lambda\{t\in[0,r]:0<F(t)\leq u\}\,du\\
 &\leq A^{1/2}+\frac12\int_0^A u^{-1/2}\,du=2A^{1/2}.
 \end{split}
\]
Here $\lambda\{t\in[0,r]:0<F(t)\leq u\}\leq u$.
For the second estimate, use
$\nu\{r\geq t:0<G(r)\leq u\}\leq u$ and the same calculation
with $A=G(t)$.  If $A=0$, both sides vanish; if $A=\infty$, the estimate is
immediate.  Sets on which the denominator is zero have measure zero
and are omitted.

Cauchy--Schwarz and \eqref{eq:Stieltjes-increasing} give
\[
 \left|\int_{[0,r]}f\,d\lambda\right|^2
 \leq2F(r)^{1/2}
 \int_{[0,r]}|f(t)|^2F(t)^{1/2}\,d\lambda(t).
\]
Moreover, \(F(r)^{1/2}\leq B_1^{1/2}G(r)^{-1/2}\), and hence
\begin{align*}
 \int_{[t,\infty)}F(r)^{1/2}\,d\nu(r)
 &\leq2B_1^{1/2}G(t)^{1/2}
 \leq2B_1F(t)^{-1/2}.
\end{align*}
Fubini therefore yields
\begin{align*}
 \|Hf\|_{L^2(\tau)}^2
 &\leq2\int_{[0,\infty)}|f(t)|^2F(t)^{1/2}
 \left[\int_{[t,\infty)}F(r)^{1/2}\,d\nu(r)\right]
d\lambda(t)
 \leq4B_1\|f\|_{L^2(\lambda)}^2.
\end{align*}
Taking \(f=1\) proves the first inequality.

For the second inequality, put
\[
Tf(r)=\int_{[r,\infty)}t^{-s}f(t)\,d\lambda(t).
\]
The adjoint is
\[
T^*g(t)=t^{-s}\int_{[0,t]}g(r)\,d\tau(r).
\]
Apply the first Hardy inequality to \(T^*\), with the roles of the
measures interchanged.  The resulting Muckenhoupt quantity is \(B_2\),
so \(\|T^*\|^2\leq4B_2\).  Since \(\|T\|=\|T^*\|\), taking \(f=1\)
gives the second estimate.  The intervals in
\eqref{eq:Stieltjes-increasing}--\eqref{eq:Stieltjes-decreasing} include
their endpoints as written.  Thus boundary atoms are assigned
consistently with the two kernels above.

The condition \eqref{eq:Hardy-no-common-zero} removes the only singular
common endpoint.  We only need to discuss a possible \(\tau\)-atom at
zero.
If \(\tau(\{0\})>0\), the contribution at \(r=0\) to the second
inequality also follows from
\[
 \left(\int_{(0,\infty)}t^{-s}\,d\lambda(t)\right)^2
 \leq\lambda([0,\infty))
 \int_{(0,\infty)}t^{-2s}\,d\lambda(t)
\]
and the limit as \(\rho\downarrow0\) in the definition of \(B_2\).
\end{proof}

\begin{lemma}\label{lem:local-tripled}
Let \(P=Q_S\) be the auxiliary quasicube from
\cref{lem:axis-sections}, and assume \(3P\cap\supp\widetilde\omega\subset L_0\). Then
\begin{align}
 &\int_{3P}\left|
 \bfR_\Psi^{\alpha,n}(\one_P\widetilde\omega)
 \right|^2d\widetilde\sigma
 \leq C\left[
 \bigl(\mathfrak T_{\mathrm{full}}^{\mathrm{dual}}\bigr)^2
 +\mathcal A_{2,\Omega}^\alpha
   (\widetilde\sigma,\widetilde\omega)
 +\mathcal A_{2,\Omega}^{\alpha,*}
   (\widetilde\sigma,\widetilde\omega)
 \right]\widetilde\omega(P).
 \label{eq:local-tripled}
\end{align}
Here
\[
 \mathfrak T_{\mathrm{full}}^{\mathrm{dual}}
 =\mathfrak T_{\bfR_\Psi^{\alpha,n},\Omega}^{\mathrm{dual}}
   (\widetilde\sigma,\widetilde\omega)
\]
is the testing constant for the unrestricted pair over
\(\Omega\mathcal Q^n\).
\end{lemma}

\begin{proof}
Put \(s=n-\alpha\).  Fix one of the tangent-line truncations used in
\(\mathfrak T_{\mathrm{full}}^{\mathrm{dual}}\).  Its radial multiplier
is symmetric, and its truncated kernel satisfies the same size bound,
with a uniform constant.  Moreover, \((\bfK_\Psi^{\alpha,n})^*(x,y) =\bfK_\Psi^{\alpha,n}(y,x) =-\bfK_\Psi^{\alpha,n}(x,y)\). All constants below are independent of the inner and outer truncation
radii.

The quasicube \(3P\) is the disjoint union of the images of \(3^n\)
half-open cubes of side \(\ell(P)\).  On \(P\), unrestricted dual testing
gives
\[
 \int_P\left|\bfR_\Psi^{\alpha,n}
   (\one_P\widetilde\omega)\right|^2d\widetilde\sigma
 \leq
 \bigl(\mathfrak T_{\mathrm{full}}^{\mathrm{dual}}\bigr)^2
 \widetilde\omega(P).
\]

Let \(E\) be one of the other members of this partition.  Write
\(P=\Omega C\), \(E=\Omega D\), and
\(P\cap L_0=[a,b)e_1\).  Put
\(p_-=ae_1\), \(p_+=be_1\), and
\(\zeta(t)=\Omega^{-1}(te_1)\).  By the construction in
\cref{lem:axis-sections}, \(\ell(C)=8C_n(b-a)\). The path \(\zeta([a,b])\) meets only the first lower face, at
\(\zeta(a)\), and the second upper face, at \(\zeta(b)\).  Its
distance from every other face is at least \(c\ell(P)\).

There is a unique \(m=(m_1,\ldots,m_n)\in\{-1,0,1\}^n\setminus\{0\}\) such that \(D=C+\ell(C)m\).  If \(m_1=-1\), assign \(E\) to
\(p_-\).  For
\(x=te_1\in P\cap L_0\), \(y\in E\), and
\(\eta=\Omega^{-1}y\),
\begin{align*}
c_n|x-p_-|
 &\leq \zeta_1(t)-\zeta_1(a)
 \leq \zeta_1(t)-\eta_1
 \leq |\zeta(t)-\eta|
 \leq \Lip(\Omega^{-1})|x-y|.
\end{align*}
The triangle inequality then gives
\(|y-p_-|\leq C|x-y|\).  Suppose next that \(m_1\neq-1\) and
\(m_2=1\).  Assign \(E\) to \(p_+\).  The second coordinate gives
\begin{align*}
c_n|x-p_+|
 &\leq \zeta_2(b)-\zeta_2(t)
 \leq \eta_2-\zeta_2(t)
 \leq |\zeta(t)-\eta|
 \leq \Lip(\Omega^{-1})|x-y|.
\end{align*}
Thus \(|y-p_+|\leq C|x-y|\).

It remains to consider \(m_1\in\{0,1\},\quad m_2\in\{-1,0\},\quad m\neq0\). If \(m_1=1\), the first upper face separates \(D\) from the path.  If
\(m_2=-1\), the second lower face does so.  If
\(m_k\neq0\) for some \(3\leq k\leq n\), one of the two \(k\)-th
faces does so.  The face margins in \cref{lem:axis-sections} are at
least \(7\ell(C)/16\).  Hence every quasicube in this last class is
separated from \(P\cap L_0\) by \(c_\Omega\ell(P)\).  Its integral is
bounded by
\cref{lem:annular-A2}, applied to the finitely many annuli with radii
comparable to \(\ell(P)\).
In this separated application, the proof of \cref{lem:annular-A2} uses
only the one-tailed condition with the corresponding order of the two
measures; the punctured conditions are not used.

Fix one endpoint \(p\), and denote the corresponding Borel set by
\(E_p\).  The two preceding coordinate estimates and the triangle
inequality give
\begin{equation}\label{eq:endpoint-cone-proof}
 |x-y|\geq c_0\bigl(|x-p|+|y-p|\bigr),
 \qquad x\in P\cap L_0,\quad y\in E_p,
\end{equation}
where \(c_0>0\) depends only on \(\Omega\).  Define measures on
\([0,\infty)\) by
\begin{align*}
 \lambda(B)
 &=\widetilde\omega
   \bigl(\{x\in P\cap L_0:|x-p|\in B\}\bigr),\\
 \tau(B)
 &=\widetilde\sigma
   \bigl(\{y\in E_p:|y-p|\in B\}\bigr).
\end{align*}
The half-open sets \(P\) and \(E_p\) are disjoint.  Thus \(p\) belongs
to at most one of them.  Consequently,
\begin{equation}\label{eq:no-common-radial-origin}
 \lambda(\{0\})\tau(\{0\})=0.
\end{equation}
Thus the hypothesis \eqref{eq:Hardy-no-common-zero} of
\cref{lem:Hardy} holds.
For \(r=|y-p|\), the kernel estimate and
\eqref{eq:endpoint-cone-proof} give
\begin{align}
 \left|\bfR_\Psi^{\alpha,n}
   (\one_P\widetilde\omega)(y)\right|
 &\leq C\int_{[0,\infty)}(r+t)^{-s}\,d\lambda(t)
 \leq C\left[
r^{-s}\lambda([0,r])
 +\int_{[r,\infty)}t^{-s}\,d\lambda(t)
 \right].
 \label{eq:hardy-kernel-split}
\end{align}
If \(\tau(\{0\})>0\), then
\eqref{eq:no-common-radial-origin} gives
\(\lambda(\{0\})=0\).  At \(r=0\), we therefore read the first term as
zero and the second as \(\int_{(0,\infty)}t^{-s}\,d\lambda(t)\).

It remains to bound \(B_1\) and \(B_2\) in \cref{lem:Hardy}.  Recall that
\begin{equation*}
 \mathcal P^\alpha(H,\nu)
 =\int_{\R^n}
 \left(
  \frac{|H|^{1/n}}
  {\bigl(|H|^{1/n}+|z-c_H|\bigr)^2}
 \right)^s d\nu(z).
\end{equation*}
Thus
\begin{align}
 \mathcal A_{2,\Omega}^\alpha
 &=\sup_{H\in\Omega\mathcal Q^n}
 \frac{\widetilde\omega(H)}{|H|^{s/n}}
 \mathcal P^\alpha
   (H,\one_{H^c}\widetilde\sigma),
 \label{eq:one-tail-forward-local}\\
 \mathcal A_{2,\Omega}^{\alpha,*}
 &=\sup_{H\in\Omega\mathcal Q^n}
 \frac{\widetilde\sigma(H)}{|H|^{s/n}}
 \mathcal P^\alpha
   (H,\one_{H^c}\widetilde\omega).
 \label{eq:one-tail-dual-local}
\end{align}
The estimates below use the reproducing Poisson integral
\(\mathcal P^\alpha\).

Both \(P\cap L_0\) and \(E_p\) are contained in a ball
\(B(p,C_1\ell(P))\).  Hence the expressions defining the Hardy
constants vanish when \(\rho>C_1\ell(P)\).  We may therefore fix
\(0<\rho\leq C_1\ell(P)\).

We first estimate
\[
B_1(\rho)
 :=\lambda([0,\rho])
   \int_{[\rho,\infty)}r^{-2s}\,d\tau(r).
\]
Let \(S_{\rho,p}\subset L_0\) be the half-open interval of length
\(2\rho\) which starts at \(p\) when \(p\in P\), and which ends at
\(p\) when \(p\notin P\).  Then \(\{x\in P\cap L_0:|x-p|\leq\rho\}\subset S_{\rho,p}\). Let \(H_\rho=Q_{S_{\rho,p}}\) be the quasicube supplied by
\cref{lem:axis-sections}.  The construction in that lemma gives \(\ell(H_\rho)\approx_\Omega\rho, \quad H_\rho\subset B(p,C_2\rho), \quad |c_{H_\rho}-p|\leq C_2\rho\). Choose \(C_0>4C_2\).  If \(y\in E_p\) and
\(|y-p|\geq C_0\rho\), then \(y\in H_\rho^c\),
\(|y-c_{H_\rho}|\approx|y-p|\), and
\begin{equation*}
 |y-p|^{-2s}
 \leq C\rho^{-s}
 \left(
  \frac{|H_\rho|^{1/n}}
  {\bigl(|H_\rho|^{1/n}+|y-c_{H_\rho}|\bigr)^2}
 \right)^s.
\end{equation*}
The factor \(\lambda([0,\rho])\) is bounded by
\(\widetilde\omega(H_\rho)\), while the part of
\(\widetilde\sigma\) in the integral is supported in \(H_\rho^c\).
Therefore
\begin{align}
 &\lambda([0,\rho])
 \int_{[C_0\rho,\infty)}r^{-2s}\,d\tau(r)\leq
C\frac{\widetilde\omega(H_\rho)}{|H_\rho|^{s/n}}
 \mathcal P^\alpha
   (H_\rho,\one_{H_\rho^c}\widetilde\sigma)
 \leq C\mathcal A_{2,\Omega}^\alpha.
 \label{eq:B1-far}
\end{align}

It remains to treat \([\rho,C_0\rho)\).  Put
\begin{align*}
X_\rho^0&=\{x\in P\cap L_0:|x-p|\leq\rho\},\qquad
Y_\rho^1=\{y\in E_p:\rho\leq|y-p|<C_0\rho\}.
\end{align*}
By \eqref{eq:endpoint-cone-proof},
\(\dist(X_\rho^0,Y_\rho^1)\geq c_0\rho\).  Choose
\(0<\delta<c_0/(10\sqrt n\Lip(\Omega))\).  Partition the
preimage space into half-open cubes of side \(\delta\rho\).  Let
\(\mathcal U_\rho\) be the images under \(\Omega\) of the cubes which
meet \(X_\rho^0\), and let \(\mathcal V_\rho\) be those which meet
\(Y_\rho^1\).  Each family has at most \(N=N(n,\Omega,C_0)\) members.
No member of one family is a member of the other.  Thus every
\(U\in\mathcal U_\rho\) and \(V\in\mathcal V_\rho\) is a pair of
disjoint quasicubes, and \(\ell(U)=\ell(V)=\delta\rho, \quad |z-c_U|\leq C\rho\quad(z\in V)\). Consequently, \(\mathcal P^\alpha(U,\one_{U^c}\widetilde\sigma) \geq c\frac{|U|^{s/n}}{\rho^{2s}} \widetilde\sigma(V)\). By \eqref{eq:one-tail-forward-local}, \(\rho^{-2s}\widetilde\omega(U)\widetilde\sigma(V) \leq C\mathcal A_{2,\Omega}^\alpha\). Summing over at most \(N^2\) pairs gives
\begin{align}
 &\lambda([0,\rho])
 \int_{[\rho,C_0\rho)}r^{-2s}\,d\tau(r)\leq
 \rho^{-2s}
 \widetilde\omega(X_\rho^0)
 \widetilde\sigma(Y_\rho^1)
 \leq C\mathcal A_{2,\Omega}^\alpha.
 \label{eq:B1-middle}
\end{align}
Equations \eqref{eq:B1-far} and \eqref{eq:B1-middle} show that
\begin{equation}\label{eq:B1-bound}
B_1\leq C\mathcal A_{2,\Omega}^\alpha.
\end{equation}

We next estimate
\[
B_2(\rho)
 :=\tau([0,\rho])
   \int_{[\rho,\infty)}t^{-2s}\,d\lambda(t).
\]
Partition the preimage space into half-open cubes of side \(\rho\), and
let \(\mathcal H_\rho\) be the images under \(\Omega\) of the cubes
which meet \(Y_\rho^0=\{y\in E_p:|y-p|\leq\rho\}\). There are at most \(N_0=N_0(n,\Omega)\) such quasicubes.  Each
\(H\in\mathcal H_\rho\) satisfies \(\ell(H)=\rho,\quad H\subset B(p,C_3\rho),\quad |c_H-p|\leq C_3\rho\). Choose \(C_4>4C_3\).  Then the source
set \(\{x\in P\cap L_0:|x-p|\geq C_4\rho\}\) lies in \(H^c\) for every
\(H\in\mathcal H_\rho\), and \(|x-p|^{-2s} \leq C\rho^{-s} \left( \frac{|H|^{1/n}} {\bigl(|H|^{1/n}+|x-c_H|\bigr)^2} \right)^s\). The set \(Y_\rho^0\) is covered by \(\mathcal H_\rho\), while
\(\{x\in P\cap L_0:|x-p|\geq C_4\rho\}\subset H^c\) for every
\(H\in\mathcal H_\rho\).  Hence
\begin{align}
 &\tau([0,\rho])
 \int_{[C_4\rho,\infty)}t^{-2s}\,d\lambda(t)\leq
C\sum_{H\in\mathcal H_\rho}
 \frac{\widetilde\sigma(H)}{|H|^{s/n}}
 \mathcal P^\alpha
   (H,\one_{H^c}\widetilde\omega)
 \leq C\mathcal A_{2,\Omega}^{\alpha,*}.
 \label{eq:B2-far}
\end{align}

For the intermediate part, put \(X_\rho^1 =\{x\in P\cap L_0:\rho\leq|x-p|<C_4\rho\}\). The sets \(X_\rho^1\) and \(Y_\rho^0\) are separated by
\(c_0\rho\).  Repeat the half-open grid construction above, with
quasicubes of side \(\delta\rho\).  The two resulting families have at
most \(N_1=N_1(n,\Omega,C_4)\) members.  If \(U\) meets \(X_\rho^1\) and
\(V\) meets \(Y_\rho^0\), then \(U\cap V=\varnothing\) and \(\mathcal P^\alpha(V,\one_{V^c}\widetilde\omega) \geq c\frac{|V|^{s/n}}{\rho^{2s}} \widetilde\omega(U)\). The dual one-tailed condition gives \(\rho^{-2s}\widetilde\sigma(V)\widetilde\omega(U) \leq C\mathcal A_{2,\Omega}^{\alpha,*}\). After summing over the finitely many pairs,
\begin{align*}
 &\tau([0,\rho])
 \int_{[\rho,C_4\rho)}t^{-2s}\,d\lambda(t)\leq
 \rho^{-2s}
 \widetilde\sigma(Y_\rho^0)
 \widetilde\omega(X_\rho^1)
 \leq C\mathcal A_{2,\Omega}^{\alpha,*}.
\end{align*}
Together with \eqref{eq:B2-far}, this proves
\begin{equation}\label{eq:B2-bound}
B_2\leq C\mathcal A_{2,\Omega}^{\alpha,*}.
\end{equation}
The constant in \eqref{eq:B1-bound} and \eqref{eq:B2-bound} depends only
on \(n\), \(\alpha\), the biLipschitz constants of \(\Omega\), and the
constants in \eqref{eq:coordinate-monotonicity}.  It is independent of
\(P\), \(p\), \(\rho\), and the truncation parameters.

It remains to check atoms on the half-open boundaries.  By
\eqref{eq:no-common-radial-origin}, no common atom occurs on both sides
of the singular Hardy kernel.  The quasicubes used above are also
half-open.  An atom on a boundary belongs to one member of the
partition, and the complementary measure in
\eqref{eq:one-tail-forward-local} or
\eqref{eq:one-tail-dual-local} lies in the corresponding hole.  Thus
\eqref{eq:B1-bound} and \eqref{eq:B2-bound} include endpoint atoms.

Applying \cref{lem:Hardy} to \eqref{eq:hardy-kernel-split} yields
\[
 \int_{E_p}\left|
 \bfR_\Psi^{\alpha,n}
   (\one_P\widetilde\omega)
 \right|^2d\widetilde\sigma
 \leq C\left(
 \mathcal A_{2,\Omega}^\alpha
 +\mathcal A_{2,\Omega}^{\alpha,*}
 \right)\widetilde\omega(P).
\]
Apply this estimate to the two endpoint unions.  Add the unrestricted
testing estimate on \(P\) and the estimates for the finitely many
separated members of the partition of \(3P\).  This proves
\eqref{eq:local-tripled}.
\end{proof}

\begin{remark}
The estimate in \eqref{eq:local-tripled} depends only on testing and
the two one-tailed conditions for the unrestricted pair.
\end{remark}

\begin{proposition}\label{prop:localized-testing}
Choose the decomposition in \cref{lem:nested-arcs} with
\(A>2C_{\rm ax}\), where \(C_{\rm ax}\) is fixed by
\eqref{eq:axis-family-constant}.  Fix \(i\), and put \(\Omega=\Omega_i=\Psi_i^{-1}\circ R_i\), where \(R_i\) is the rotation used in
\eqref{eq:coordinate-monotonicity}.  Thus \(\Omega\) satisfies the
hypothesis of \cref{lem:axis-sections}.  Work in the rigid coordinates
of the $i$-th arc, and denote $(\mathcal U_i)_\#\sigma$ and
$(\mathcal U_i)_\#\omega$ again by $\sigma$ and $\omega$.  Put
\[
 \widetilde\sigma_i=\Psi_i^*\sigma,\qquad
 \widetilde\omega^{(i)}=\Psi_i^*\omega,\qquad
 \widetilde\omega_i=\Psi_i^*(\one_{\Gamma_i}\omega).
\]
In this proposition,
\(\mathfrak T_{\mathrm{full}}\),
\(\mathfrak T_{\mathrm{full}}^{\mathrm{dual}}\), and the
Muckenhoupt constants refer to the unrestricted pair
\((\widetilde\sigma_i,\widetilde\omega^{(i)})\) over
\(\Omega\mathcal Q^n\).  Then, for every
\(Q\in\Omega\mathcal Q^n\),
\begin{align}
 \int_Q\left|
 \bfR_{\Psi_i}^{\alpha,n}(\one_Q\widetilde\omega_i)
 \right|^2d\widetilde\sigma_i
 &\leq C\left[
 \bigl(\mathfrak T_{\mathrm{full}}^{\mathrm{dual}}\bigr)^2
 +\mathcal A_{2,\Omega}^\alpha
 +\mathcal A_{2,\Omega}^{\alpha,*}
 \right]\widetilde\omega_i(Q),
 \label{eq:localized-dual-testing}\\
 \int_Q\left|
 \bfR_{\Psi_i}^{\alpha,n}(\one_Q\widetilde\sigma_i)
 \right|^2d\widetilde\omega_i
 &\leq
 \bigl(\mathfrak T_{\mathrm{full}}\bigr)^2
 \widetilde\sigma_i(Q).\notag
\end{align}
There is no restriction on \(\ell(Q)\).
\end{proposition}

\begin{proof}
The forward estimate follows from
\(\widetilde\omega_i\leq\widetilde\omega^{(i)}\).
By \cref{lem:axis-sections}, \(Q\cap L_0\) is a half-open interval.  Put \(S_Q=Q\cap I_i\). If \(\widetilde\omega_i(Q)=0\), the dual estimate is trivial.  Otherwise
\(S_Q\) is a nonempty half-open interval of positive length.  Let
\(P_Q=Q_{S_Q}\) be given by \cref{lem:axis-sections}.  By
\eqref{eq:auxiliary-size},
\begin{align*}
 3P_Q
 &\subset\{x:\dist(x,S_Q)\leq C_{\Omega}|S_Q|\}
 \subset\{x:\dist(x,I_i)\leq C_{\Omega}|I_i|\}.
\end{align*}
Since \(C_\Omega\leq C_{\rm ax}\), our choice of \(A\) gives
\(A>2C_\Omega\).  Hence
\eqref{eq:tube-localization} gives
\begin{equation}\label{eq:tripled-support-on-axis}
 3P_Q\cap\supp\widetilde\omega^{(i)}
 \subset3P_Q\cap\widetilde{\Gamma}^{(i)}
 \subset\Psi_i^{-1}(\widehat{\Gamma}_i)\subset L_0.
\end{equation}

The half-open construction of \(\Gamma_i\) gives
\begin{equation}\label{eq:restriction-on-axis-interval}
 \one_{I_i}\widetilde\omega^{(i)}=\widetilde\omega_i.
\end{equation}
For the last arc of an open curve, \eqref{eq:ghost-zero} shows that
\(\widetilde\omega^{(i)}\) has zero mass on the added segment.  Since
\(P_Q\cap L_0=S_Q\), equations
\eqref{eq:tripled-support-on-axis} and
\eqref{eq:restriction-on-axis-interval} give
\begin{align}
 \one_{P_Q}\widetilde\omega^{(i)}
 &=\one_{P_Q\cap L_0}\widetilde\omega^{(i)}
 =\one_{S_Q}\widetilde\omega^{(i)}
 =\one_{S_Q}\widetilde\omega_i
 =\one_Q\widetilde\omega_i.
 \label{eq:exact-measure-localization}
\end{align}
Every partition point belongs to one half-open arc, so the identity also
holds at atoms on the partition boundaries.
Split the dual testing integral over \(Q\setminus3P_Q\) and
\(Q\cap3P_Q\).  Put \(a_{P_Q}=|P_Q|^{1/n}\).  If
\(x\notin3P_Q\) and \(y\in P_Q\), then \(|x-y|\approx_\Omega a_{P_Q}+|x-c_{P_Q}|\). Consequently,
\begin{align*}
 &\int_{Q\setminus3P_Q}
 |\bfR_{\Psi_i}^{\alpha,n}
   (\one_{P_Q}\widetilde\omega^{(i)})(x)|^2
d\widetilde\sigma_i(x)\\
 &\quad\leq C\widetilde\omega^{(i)}(P_Q)^2
 \int_{P_Q^c}
 \frac{d\widetilde\sigma_i(x)}
 {(a_{P_Q}+|x-c_{P_Q}|)^{2(n-\alpha)}}\\
 &\quad=C\widetilde\omega^{(i)}(P_Q)
 \frac{\widetilde\omega^{(i)}(P_Q)}{a_{P_Q}^{n-\alpha}}
 \mathcal P^\alpha
   (P_Q,\one_{P_Q^c}\widetilde\sigma_i)\leq C\mathcal A_{2,\Omega}^\alpha
 \widetilde\omega^{(i)}(P_Q).
\end{align*}
\Cref{lem:local-tripled} controls the second part.  Finally,
\eqref{eq:exact-measure-localization} gives \(\widetilde\omega^{(i)}(P_Q)=\widetilde\omega_i(Q)\). This proves \eqref{eq:localized-dual-testing}.  Since
\(|S_Q|\leq|I_i|\), we have
\(\ell(P_Q)\lesssim_\Omega|I_i|\), independently of \(\ell(Q)\).
Thus the estimate holds for testing cubes of every size.
\end{proof}

\begin{lemma}
\label{lem:A2-restriction}
For every fixed globally biLipschitz quasicube map \(\Omega\) and every
Borel set \(E\),
\begin{equation}\label{eq:A2-restriction}
 \mathfrak A_{2,\Omega}^\alpha(\sigma,\one_E\omega)
 \leq C_\Omega\mathfrak A_{2,\Omega}^\alpha(\sigma,\omega).
\end{equation}
\end{lemma}

\begin{proof}
Put \(\omega_E=\one_E\omega\), let \(\mathfrak P\) be the common atoms
of \((\sigma,\omega)\), and let
\(\mathfrak P_E=\mathfrak P\cap E\).  The two one-tailed constants are
monotone:
\[
 \mathcal A_{2,\Omega}^\alpha(\sigma,\omega_E)
 \leq\mathcal A_{2,\Omega}^\alpha(\sigma,\omega),
 \qquad
 \mathcal A_{2,\Omega}^{\alpha,*}(\sigma,\omega_E)
 \leq\mathcal A_{2,\Omega}^{\alpha,*}(\sigma,\omega).
\]
It remains to consider the two punctured products
\[
 \frac{\sigma(Q,\mathfrak P_E)\omega_E(Q)}
 {|Q|^{2(1-\alpha/n)}}
 \quad\text{and}\quad
 \frac{\sigma(Q)\omega_E(Q,\mathfrak P_E)}
 {|Q|^{2(1-\alpha/n)}}.
\]
Fix an \(\Omega\)-quasicube \(Q\).  If
\(Q\cap\mathfrak P=\varnothing\), both products decrease under the
restriction.  We therefore assume that
\(Q\cap\mathfrak P\neq\varnothing\).

We first record an elementary fact.  If \(\mu(Q)<\infty\), the
\(\mu\)-masses of any nonempty collection of atoms in \(Q\) have a
maximum.  Indeed, if the supremum \(M>0\) were not attained, infinitely
many distinct atoms would have mass greater than \(M/2\), contradicting
\(\mu(Q)<\infty\).

Choose \(p\in Q\cap\mathfrak P\) such that \(\sigma(\{p\}) =\max_{q\in Q\cap\mathfrak P}\sigma(\{q\})\). If \(p\in E\), then \(p\) is also a largest common
\(\sigma\)-atom for the restricted pair.  Hence
\(\sigma(Q,\mathfrak P_E)=\sigma(Q,\mathfrak P)\), and the first
punctured product is bounded by
\(A_{2,\Omega}^{\alpha,\mathrm{punct}}(\sigma,\omega)\).

Suppose that \(p\notin E\).  Then \(\sigma(Q,\mathfrak P_E) \leq\sigma(Q,\mathfrak P)+\sigma(\{p\})\). After division by \(|Q|^{2(1-\alpha/n)}\), the term
\(\sigma(Q,\mathfrak P)\omega_E(Q)\) is bounded by
\(A_{2,\Omega}^{\alpha,\mathrm{punct}}(\sigma,\omega)\).
Choose \(r\in Q\cap\mathfrak P\) such that \(\omega(\{r\}) =\max_{q\in Q\cap\mathfrak P}\omega(\{q\})\). We write
\[
 \sigma(\{p\})\omega_E(Q)
 =
 \sigma(\{p\})
 \bigl[\omega_E(Q)-\one_E(r)\omega(\{r\})\bigr]
 +\one_E(r)\sigma(\{p\})\omega(\{r\}).
\]
The expression in brackets is at most
\(\omega(Q)-\omega(\{r\})=\omega(Q,\mathfrak P)\).  Since
\(\sigma(\{p\})\leq\sigma(Q)\), division by
\(|Q|^{2(1-\alpha/n)}\) shows that the first term is bounded by the
original starred punctured characteristic
\(A_{2,\Omega}^{\alpha,*,\mathrm{punct}}(\sigma,\omega)\).

If \(r=p\), the last term vanishes because \(p\notin E\).  It remains
to consider \(r\in E\) and \(r\neq p\).  Put \(a=\Omega^{-1}(p),\quad b=\Omega^{-1}(r),\quad d=|a-b|_\infty\), and let
\[
C=\prod_{k=1}^n[a_k-d/4,a_k+d/4),
 \qquad H=\Omega C.
\]
Then \(p\in H\), \(r\notin H\), and \(\ell(H)=d/2,\quad |H|^{1/n}\approx_\Omega d,\quad |p-r|\approx_\Omega d\). Therefore the starred one-tailed condition with a hole gives
\begin{align*}
 \mathcal A_{2,\Omega}^{\alpha,*}(\sigma,\omega)
 &\geq
 \frac{\sigma(H)}{|H|^{1-\alpha/n}}
 \mathcal P^\alpha(H,\one_{H^c}\omega)\gtrsim_\Omega
 \frac{\sigma(\{p\})\omega(\{r\})}
 {d^{2(n-\alpha)}}.
\end{align*}
Since \(p,r\in Q\), their preimages belong to the ordinary cube
\(\Omega^{-1}Q\).  Thus \(d\leq\ell(Q)\), while
\(|Q|^{1/n}\approx_\Omega\ell(Q)\).  It follows that
\[
 \frac{\sigma(\{p\})\omega(\{r\})}
 {|Q|^{2(1-\alpha/n)}}
 \lesssim_\Omega
 \mathcal A_{2,\Omega}^{\alpha,*}(\sigma,\omega).
\]
This proves the required estimate for the unstarred punctured
characteristic of the restricted pair.

The starred punctured characteristic is termwise monotone.  Indeed,
let \(r\) be a largest common \(\omega\)-atom of the full pair in
\(Q\).  If \(r\in E\), it remains a largest common \(\omega\)-atom
after restriction, and \(\omega_E(Q,\mathfrak P_E) =\omega_E(Q)-\omega(\{r\}) \leq\omega(Q)-\omega(\{r\}) =\omega(Q,\mathfrak P)\). If \(r\notin E\), then \(\omega_E(Q,\mathfrak P_E) \leq\omega_E(Q) \leq\omega(Q)-\omega(\{r\}) =\omega(Q,\mathfrak P)\). Multiplication by
\(\sigma(Q)/|Q|^{2(1-\alpha/n)}\) gives
\[
A_{2,\Omega}^{\alpha,*,\mathrm{punct}}(\sigma,\omega_E)
 \leq
A_{2,\Omega}^{\alpha,*,\mathrm{punct}}(\sigma,\omega).
\]
This proves \eqref{eq:A2-restriction}.
\end{proof}

\section{Proof of the main theorem}\label{sec:7}

\begin{proof}[Proof of \cref{thm:main}]
For \(n=1\), a compact regular curve is an interval, so
\cite[Theorem~9]{CurvePaper} applies directly.  We now assume
\(n\geq2\).
By duality, suppose that $\omega$ is supported on the regular
$C^{1,\mathrm{Dini}}$ curve $\Gamma=\Phi([a,b])$.

\medskip
\noindent\emph{Step 1: finite graph pieces.}
The goodness exponent \(\varepsilon\) fixed after
\eqref{eq:strong-energy-definition} and the slope bound for the graph
charts are independent parameters.  Put \(\varepsilon_{\rm sl}:=\frac12\varepsilon_{n,\alpha}\). We follow the order of choices in \cref{lem:nested-arcs}.  First choose
its finite graph-chart family, before choosing \(A\).  In the rigid
coordinates of a chart indexed by \(\nu\),
write the two graph-coordinate intervals used for the extension as \(\widehat J_\nu^0\Subset J_\nu^*, \ \dist(\widehat J_\nu^0,\R\setminus J_\nu^*) \geq\kappa|J_\nu^*|\). Let \(\psi_\nu^*\) denote the local graph function on \(J_\nu^*\).
Translate in the transverse direction so that
\(\psi_\nu^*(t_\nu)=0\) at some \(t_\nu\in J_\nu^*\).  Uniform
continuity of
\(\Phi'\) makes the chart small enough that \(\|D\psi_\nu^*\|_\infty<\varepsilon_{\rm sl}/C_\kappa\), where \(C_\kappa\) is the constant in
\cref{lem:graph-extension}.
\Cref{lem:graph-modulus} gives \(\vartheta_{\psi_\nu^*}(r) \leq C\vartheta_\Phi(Cr), \quad 0<r\leq c\). The part of the Dini integral over \([c,1]\) is finite because
\(D\psi_\nu^*\) is bounded.
\Cref{lem:graph-extension} gives
\(\psi_\nu\in C_c^{1,\mathrm{Dini}}(\R;\R^{n-1})\), equal to
\(\psi_\nu^*\) on \(\widehat J_\nu^0\), such that
\begin{equation}\label{eq:small-extended-graph}
 \|D\psi_\nu\|_\infty<\varepsilon_{n,\alpha},\qquad
 \|\psi_\nu\|_\infty
 \leq C_\kappa\varepsilon_{\rm sl}|J_\nu^*|.
\end{equation}
Let \(\Psi_\nu(x_1,x')=(x_1,x'+\psi_\nu(x_1))\). These finitely many maps fix the constants \(K\) and \(C_{\rm ax}\) in
\eqref{eq:axis-family-constant}.  Now choose \(A>2C_{\rm ax}\) and form
the final arcs in \cref{lem:nested-arcs}.  Let \(\Psi_i\) be the fixed
map and let \(\mathcal U_i\) be the rigid motion assigned there to
\(\Gamma_i\).
The map \(\Psi_i\) sends the relevant interval of \(L_0\) onto the graph
arc, and \(\Psi_i,\Psi_i^{-1}\) are globally biLipschitz.

Put \(\omega_i=\one_{\Gamma_i}\omega\).  The half-open convention in
\cref{lem:nested-arcs}, including the cyclic convention for a closed
curve, gives
\begin{equation}\label{eq:measure-decomposition}
 \omega=\sum_{i=1}^N\omega_i.
\end{equation}

\medskip
\noindent\emph{Step 2: testing after straightening.}
Fix \(i\).  Pass to its rigid graph coordinates by defining
\[
 \sigma_i^{\rm rig}=(\mathcal U_i)_\#\sigma,
 \qquad
 \omega^{(i),\rm rig}=(\mathcal U_i)_\#\omega,
 \qquad
 \omega_i^{\rm rig}=(\mathcal U_i)_\#\omega_i.
\]
For Steps 2--4 with this fixed \(i\), we denote these three measures
again by \(\sigma\), \(\omega\), and \(\omega_i\).  We restore the
original coordinates before summing in \(i\).  Let \(R_i\) be the
diagonal rotation fixed for this chart in the proof of
\cref{lem:nested-arcs}, and put \(\Omega_i=\Psi_i^{-1}\circ R_i\). The fractional Riesz vector is covariant under rigid motions.  Its
Euclidean vector norm and vector testing are unchanged, while the
combined Muckenhoupt constants are comparable with dimensional
constants.  A cube \(R_iQ\) in the rigid graph coordinates
corresponds to the rotated cube \(\mathcal U_i^{-1}(R_iQ)\) in the
original coordinates; the translation in \(\mathcal U_i\) only changes
the center.  Therefore the hypothesis over all rotated cubes gives
unrestricted testing for the relabeled pair over
\(R_i\mathcal Q^n\).  Put
\[
 \widetilde\sigma_i=\Psi_i^*\sigma,
 \qquad
 \widetilde\omega^{(i)}=\Psi_i^*\omega,
 \qquad
 \widetilde\omega_i=\Psi_i^*\omega_i.
\]
Then \(\widetilde\omega_i\) is supported on \(L_0\), and \(\Psi_i^{-1}(R_i\mathcal Q^n)=\Omega_i\mathcal Q^n\). \Cref{prop:change-variable} gives unrestricted conformal testing for
\((\widetilde\sigma_i,\widetilde\omega^{(i)})\) over this family.  The
cutoff comparison is uniform in the inner and outer
truncation radii.  The combined Muckenhoupt comparison is
\eqref{eq:A2-pullback}.

\Cref{prop:localized-testing} gives testing for the restricted
pulled-back pair over all testing quasicubes.  The measure identity
\eqref{eq:exact-measure-localization} includes atoms at partition
boundaries.  Thus
\begin{equation}\label{eq:localized-testing}
 \mathfrak T_{\bfR_{\Psi_i}^{\alpha,n},\Omega_i}
   (\widetilde\sigma_i,\widetilde\omega_i)
 +\mathfrak T_{\bfR_{\Psi_i}^{\alpha,n},\Omega_i}^{\mathrm{dual}}
   (\widetilde\sigma_i,\widetilde\omega_i)
 \leq C\left(
 \sqrt{\mathfrak A_2^\alpha}
 +\mathfrak T_{\bfR^{\alpha,n}}
 +\mathfrak T_{\bfR^{\alpha,n}}^{\mathrm{dual}}
 \right).
\end{equation}
\Cref{lem:A2-restriction} and the combined pullback estimate also give
\[
 \mathfrak A_{2,\Omega_i}^\alpha
 (\widetilde\sigma_i,\widetilde\omega_i)
 \leq C\mathfrak A_2^\alpha(\sigma,\omega).
\]
Applying \cref{prop:change-variable} in the reverse direction gives
testing for \((\sigma,\omega_i)\) over \(R_i\mathcal Q^n\).  Write
\(\mathfrak T_i\) and \(\mathfrak T_i^{\mathrm{dual}}\) for these two
restricted testing constants.  Then
\begin{equation}\label{eq:restricted-original-testing}
 \mathfrak T_i+\mathfrak T_i^{\mathrm{dual}}
 \leq C\left(
 \sqrt{\mathfrak A_2^\alpha}
 +\mathfrak T_{\bfR^{\alpha,n}}
 +\mathfrak T_{\bfR^{\alpha,n}}^{\mathrm{dual}}
 \right).
\end{equation}

\medskip
\noindent\emph{Step 3: energy on the line and transfer back.}
Apply \cref{prop:C1-energy-reversal}, using
\eqref{eq:small-extended-graph} and
\eqref{eq:localized-testing}, with the quasicube map
\(\Omega_i=\Psi_i^{-1}\circ R_i\).  The pulled-back pair satisfies the
forward and dual strong-energy conditions of
\cite[Definition~5]{EnergyCommon}:
\begin{align*}
 &\cE_{\alpha,\Omega_i}^{\mathrm{strong}}
   (\widetilde\sigma_i,\widetilde\omega_i)
 +\cE_{\alpha,\Omega_i}^{\mathrm{strong,dual}}
   (\widetilde\sigma_i,\widetilde\omega_i)
 \leq C\left(
 \sqrt{\mathfrak A_2^\alpha}
 +\mathfrak T_{\bfR^{\alpha,n}}
 +\mathfrak T_{\bfR^{\alpha,n}}^{\mathrm{dual}}
 \right).
\end{align*}
This includes the supremum over every dyadic grid and subpartition, the
alternate-top expression, and every refinement index \(\ell\).
\Cref{prop:energy-invariance}, with \(F=\Psi_i,\quad \widetilde\Omega_i=F^{-1}\circ R_i=\Omega_i\), transfers these estimates to
\[
 \cE_{\alpha,R_i}^{\mathrm{strong}}(\sigma,\omega_i)
 +\cE_{\alpha,R_i}^{\mathrm{strong,dual}}(\sigma,\omega_i)
 \leq C\left(
 \sqrt{\mathfrak A_2^\alpha}
 +\mathfrak T_{\bfR^{\alpha,n}}
 +\mathfrak T_{\bfR^{\alpha,n}}^{\mathrm{dual}}\right).
\]
We can therefore apply the \(T1\) theorem in the original coordinates.

\medskip
\noindent\emph{Step 4: apply the $T1$ theorem to the smooth kernel.}
Fix \(i\) and put \(\Omega=R_i\).  For the pair
\((\sigma,\omega_i)\), the preceding steps give
\begin{align*}
 &\mathcal A_{2,R_i}^\alpha
 +\mathcal A_{2,R_i}^{\alpha,*}
 +A_{2,R_i}^{\alpha,\mathrm{punct}}
 +A_{2,R_i}^{\alpha,*,\mathrm{punct}}<\infty,\\
 &\cE_{\alpha,R_i}^{\mathrm{strong}}
 +\cE_{\alpha,R_i}^{\mathrm{strong,dual}}<\infty.
\end{align*}
The energy constants are uniform over every dyadic quasigrid
\(R_i\mathcal D\).  They include the subpartition, alternate-top, and
refined families in \cite[Definition~5]{EnergyCommon}.

For every \(Q\in R_i\mathcal Q^n\) and \(1\leq j\leq n\), vector
testing gives
\begin{align*}
 \int_Q|R_{j,\sigma}^{\alpha,n}(\one_Q)|^2\,d\omega_i
 &\leq\mathfrak T_i^2\sigma(Q),\\
 \int_Q|R_{j,\omega_i}^{\alpha,n,*}(\one_Q)|^2\,d\sigma
 &\leq(\mathfrak T_i^{\mathrm{dual}})^2\omega_i(Q).
\end{align*}
Thus both scalar quasicube testing conditions hold, and
\eqref{eq:restricted-original-testing} gives their required bounds.
The constants in these two displays use the tangent-line truncations.
Denote the corresponding scalar smooth testing constants by
\(\mathfrak T_{i,j}^{s}\) and
\(\mathfrak T_{i,j}^{s,\mathrm{dual}}\).  By
\eqref{eq:smooth-tangent-comparison}, they satisfy
\[
 \mathfrak T_{i,j}^{s}
 +\mathfrak T_{i,j}^{s,\mathrm{dual}}
 \leq C\left(
 \mathfrak T_i+\mathfrak T_i^{\mathrm{dual}}
 +\sqrt{\mathfrak A_2^\alpha}\right).
\]
Here we apply \cref{lem:tangent-smooth-comparison} to
\(K_j^{\alpha,n}\), with a constant independent of \(j\).

The scalar Riesz kernel in the current rigid coordinates is \(K_j^{\alpha,n}(x,y) =c_{\alpha,n}\frac{y_j-x_j}{|y-x|^{n+1-\alpha}}\). Direct differentiation gives
\begin{align*}
 |K_j^{\alpha,n}(x,y)|
 &\leq C|x-y|^{\alpha-n},\\
 |\nabla_xK_j^{\alpha,n}(x,y)|
 +|\nabla_yK_j^{\alpha,n}(x,y)|
 &\leq C|x-y|^{\alpha-n-1}.
\end{align*}
If \(|x-x'|\leq|x-y|/2\), the mean value theorem applied to the second
derivatives gives
\[
 |\nabla_xK_j^{\alpha,n}(x,y)-\nabla_xK_j^{\alpha,n}(x',y)|
 \leq C|x-x'|\,|x-y|^{\alpha-n-2}.
\]
The analogous estimate holds in the \(y\)-variable and for the adjoint
kernel.  Thus this is a standard \(\alpha\)-fractional
Calder\'on--Zygmund kernel of smoothness exponent \(\delta=1\).  Its
constants depend only on \(n\) and \(\alpha\).  Therefore
\cite[Theorem~2]{GoodLambda}, with
\(T^\alpha=R_j^{\alpha,n}\) and \(\Omega=R_i\), applies to
\((\sigma,\omega_i)\).  Its hypotheses are the four Muckenhoupt
conditions, the two quasitesting conditions, and the two strong
quasienergy conditions verified above.  It gives
\[
 \mathfrak N_{R_j^{\alpha,n}}^{s}(\sigma,\omega_i)
 \leq C\left(
 \sqrt{\mathfrak A_2^\alpha}
 +\mathfrak T_{\bfR^{\alpha,n}}
 +\mathfrak T_{\bfR^{\alpha,n}}^{\mathrm{dual}}\right).
\]
By \eqref{eq:smooth-tangent-comparison}, the tangent-line norm satisfies
\[
 \mathfrak N_{R_j^{\alpha,n}}(\sigma,\omega_i)
 \leq C\left(
 \sqrt{\mathfrak A_2^\alpha}
 +\mathfrak T_{\bfR^{\alpha,n}}
 +\mathfrak T_{\bfR^{\alpha,n}}^{\mathrm{dual}}\right).
\]
Since
\[
 \|\bfR_\sigma^{\alpha,n}f\|_{L^2(\omega_i;\ell^2)}^2
 =\sum_{j=1}^n
 \|R_{j,\sigma}^{\alpha,n}f\|_{L^2(\omega_i)}^2,
\]
we obtain
\begin{equation*}
 \mathfrak N_{\bfR^{\alpha,n}}(\sigma,\omega_i)
 \leq C\left(
 \sqrt{\mathfrak A_2^\alpha}
 +\mathfrak T_{\bfR^{\alpha,n}}
 +\mathfrak T_{\bfR^{\alpha,n}}^{\mathrm{dual}}
 \right).
\end{equation*}
Returning to the original coordinates gives the same estimate for
\((\sigma,\omega_i)\).  By \eqref{eq:measure-decomposition},
\begin{align*}
 \|\bfR^{\alpha,n}_\sigma f\|_{L^2(\omega)}^2
 &=\sum_{i=1}^N
   \|\bfR^{\alpha,n}_\sigma f\|_{L^2(\omega_i)}^2
 \leq C_N\left(
 \sqrt{\mathfrak A_2^\alpha}
 +\mathfrak T_{\bfR^{\alpha,n}}
 +\mathfrak T_{\bfR^{\alpha,n}}^{\mathrm{dual}}
 \right)^2
 \|f\|_{L^2(\sigma)}^2.
\end{align*}
This proves sufficiency.

We finish with necessity.  Fix any tangent-line truncation
\(\bfR_{\varepsilon,R}^{\alpha,n,\mathrm{tl}}\) used in the definition of the
operator norm.  Its kernel is zero on the diagonal.  Hence, for every
rotated cube \(Q\),
\[
 \int_Q
 \left|\bfR_{\varepsilon,R}^{\alpha,n,\mathrm{tl}}(\one_Q\sigma)\right|^2
d\omega
 \leq
 \mathfrak N_{\bfR^{\alpha,n}}^2\sigma(Q).
\]
For each \(1\leq j\leq n\), the \(j\)-th truncated scalar component
has operator norm at most \(\mathfrak N_{\bfR^{\alpha,n}}\), and
Hilbert space duality gives the same bound for its adjoint.  Summing the
squared \(L^2(\sigma)\) norms of the \(n\) adjoint components gives
\[
 \int_Q
 \left|
 \bigl(\bfR_{\varepsilon,R}^{\alpha,n,\mathrm{tl}}\bigr)^*
 (\one_Q\omega)\right|^2d\sigma
 \leq
n\mathfrak N_{\bfR^{\alpha,n}}^2\omega(Q).
\]
These estimates are uniform in \(\varepsilon,R\).  Thus
\[
 \mathfrak T_{\bfR^{\alpha,n}}
 +\mathfrak T_{\bfR^{\alpha,n}}^{\mathrm{dual}}
 \leq(1+\sqrt n)\mathfrak N_{\bfR^{\alpha,n}}.
\]

The necessity of the two one-tailed conditions with holes and the two
punctured conditions is proved in
\cite[Theorem~5(2)]{CurvePaper}; see also
\cite[Sections~2.2 and 2.2.1]{CurvePaper}.  That proof has no regularity
assumption on the support of either measure.  It applies to the
tangent-line truncations used here.  Each separated-set comparison in the
cited proof is first made on a bounded piece.  Choose
\(\varepsilon\) smaller than the separation and \(R\) larger than the
diameter of that piece.  The tangent-line multiplier is then one on the
relevant product.  The estimates are independent of \(\varepsilon\) and
\(R\).  Exhausting an unbounded tail by bounded pieces and using monotone
convergence for the positive Poisson expressions gives the full
one-tailed conditions.

For completeness, the fractional Riesz vector has the strong
ellipticity used for the one-tailed conditions.  Cover the unit sphere by
finitely many cones.  For each cone, one can choose a unit vector \(e\)
such that \(\left| e\cdot\bfK^{\alpha,n}(x,y) \right| \geq c_{n,\alpha}|x-y|^{\alpha-n}\) whenever \(y-x\) belongs to that cone, and
\(e\cdot\bfK^{\alpha,n}(x,y)\) does not change sign there.  The
punctured argument removes a largest common
point mass in each cube.  Since every truncated kernel is zero on the
diagonal, a common atom paired with itself makes no contribution.  Thus
the common-point-mass proof in
\cite[Section~2.2.1]{CurvePaper} gives \(\mathfrak A_2^\alpha \leq C\mathfrak N_{\bfR^{\alpha,n}}^2\).
The testing estimates are uniform over all rotations.  If
\(\rho(x)=Ux+a\), where \(U\) is orthogonal, then \(\bfK^{\alpha,n}(\rho x,\rho y) =U\bfK^{\alpha,n}(x,y)\). The Euclidean norm of the vector kernel is unchanged.  Equivalently,
one may apply \cite[Theorem~5(2)]{CurvePaper} with each rigid motion
as the quasicube map; all its biLipschitz constants are one.  The
estimates are therefore uniform over all rotated cubes.  Together with
the testing estimates above, this proves necessity and the opposite
inequality in \eqref{eq:norm-comparison}.
\end{proof}

\section{Poisson equations with curve-supported sources}\label{sec:8}
\label{sec:Poisson-application}

Poisson equations with a source concentrated on a curve occur in models
of wells, thin heat sources, and mixed-dimensional flow.  Such equations
have been studied by weighted Sobolev and singularity-splitting methods;
see
\cite{GjerdeKumarNordbottenWohlmuth2019,
GjerdeKumarNordbotten2021,Ojea2024}.  We apply \cref{thm:main} to this
problem and obtain a criterion for finite observation measures.  We then
determine the range of exponents for distance power measures.

\subsection{Green potentials and the two-weight criterion}

Let \(n\geq2\), let \(\Omega\subset\R^n\) be a bounded \(C^2\)
domain, and let \(\Gamma\Subset\Omega\) be a compact regular
\(C^{1,\mathrm{Dini}}\) curve.  Let \(G_\Omega(x,y)\) be the Dirichlet
Green function for \(-\Delta\) in \(\Omega\), with \(-\Delta_xG_\Omega(x,y)=\delta_y, \quad G_\Omega(x,y)=0\quad\text{for }x\in\partial\Omega\). The fundamental solution \(E_n\) of \(-\Delta\) in \(\R^n\) is
\[
E_n(z)=
 \begin{cases}
  \displaystyle
  \frac{|z|^{2-n}}{(n-2)|\mathbb S^{n-1}|},&n\geq3,\\[6pt]
  \displaystyle-\frac{1}{2\pi}\log|z|,&n=2.
 \end{cases}
\]
Write \(G_\Omega(x,y)=E_n(x-y)-H_\Omega(x,y)\). For every \(y\in\Gamma\), the function \(H_\Omega(\cdot,y)\) is harmonic
in \(\Omega\) and has boundary values \(E_n(\cdot-y)\).  Since
\(\dist(\Gamma,\partial\Omega)>0\), these boundary data are restrictions
of functions with uniformly bounded \(C^2\) norms in a fixed
neighbourhood of \(\partial\Omega\).  Boundary estimates for harmonic
functions give
\[
 \sup_{x\in\overline\Omega,\ y\in\Gamma}
 |\nabla_xH_\Omega(x,y)|\leq C_{\Omega,\Gamma}.
\]
There is a nonzero dimensional constant \(a_n\) such that \(\nabla_xG_\Omega(x,y)=a_n\bfK^{1,n}(x,y)+B_\Omega(x,y)\),
where \(B_\Omega(x,y):=-\nabla_xH_\Omega(x,y)\). Thus the Green gradient is the \(1\)-fractional Riesz vector plus an
operator with a bounded kernel.

Let \(\sigma\) be a finite positive Borel measure supported on
\(\Gamma\), and let \(\omega\) be a finite positive Borel measure on
\(\overline\Omega\).  We extend both measures by zero to \(\R^n\).  For
\(f\in L^2(\sigma)\), put
\[
u_f(x):=\int_\Gamma G_\Omega(x,y)f(y)\,d\sigma(y).
\]
Since \(f\,\sigma\) has finite total variation, local integrability
of the fundamental solution and Fubini's theorem show that this integral
defines a locally integrable function almost everywhere in \(\Omega\).
It is smooth off \(\Gamma\), extends continuously to zero on
\(\partial\Omega\), and satisfies \(-\Delta u_f=f\,\sigma\quad\text{in }\mathcal D'(\Omega), \quad u_f|_{\partial\Omega}=0\). At every point outside \(\Gamma\), differentiation under the integral
gives
\[
 \nabla u_f(x)
 =a_n\bfR^{1,n}_\sigma f(x)+\mathbf S_{\Omega,\sigma}f(x),
 \qquad
 \mathbf S_{\Omega,\sigma}f(x)
 :=\int_\Gamma B_\Omega(x,y)f(y)\,d\sigma(y).
\]

When \(\omega\) charges \(\Gamma\), we use truncated Green gradients.
With \(\alpha=1\), let
\(\bfK_{\varepsilon,R}^{\mathrm{tl}}\) be the tangent-line truncated
kernel defined in Section~3, and put
\[
 \bfR_{\varepsilon,R,\sigma}^{1,n,\mathrm{tl}}f(x)
 :=\int_\Gamma
 \bfK_{\varepsilon,R}^{\mathrm{tl}}(x,y)f(y)\,d\sigma(y).
\]
The truncated Green gradient is
\[
 \mathbf G_{\Omega,\varepsilon,R,\sigma}f
 :=a_n\bfR_{\varepsilon,R,\sigma}^{1,n,\mathrm{tl}}f
   +\mathbf S_{\Omega,\sigma}f.
\]
The smooth term is not truncated.  If \(\omega(\Gamma)=0\), this family
converges \(\omega\)-almost everywhere to the ordinary gradient above as
\(\varepsilon\downarrow0\) and \(R\uparrow\infty\).

Let \(\mathfrak N_{\mathbf G_\Omega}\) be the least constant such that
\[
 \sup_{0<\varepsilon<R}
 \|\mathbf G_{\Omega,\varepsilon,R,\sigma}f\|_{L^2(\omega)}
 \leq
 \mathfrak N_{\mathbf G_\Omega}\|f\|_{L^2(\sigma)}.
\]
Define the forward Green testing constant by
\[
 \mathfrak T_{\mathbf G_\Omega}^2
 :=\sup_Q\sup_{0<\varepsilon<R}
 \frac{1}{\sigma(Q)}
 \int_Q
 \left|
 \mathbf G_{\Omega,\varepsilon,R,\sigma}(\one_Q)
 \right|^2d\omega,
\]
where the quotient is zero when \(\sigma(Q)=0\), and where \(Q\) runs
over all rotated cubes.  For the dual constant, let $G_{\Omega,j,\varepsilon,R}$ denote
the $j$-th component of the truncated Green kernel and put
\[
 (\mathfrak T_{\mathbf G_\Omega}^{\mathrm{dual}})^2
 :=\sup_Q\sup_{0<\varepsilon<R}\frac1{\omega(Q)}
 \sum_{j=1}^n\int_Q
 |(G_{\Omega,j,\varepsilon,R})_\omega^*(\one_Q)|^2\,d\sigma.
\]
The quotient is zero when $\omega(Q)=0$.

\begin{theorem}
\label{thm:Green-gradient}
Under the assumptions above, the uniformly truncated estimate
\[
 \sup_{0<\varepsilon<R}
 \|\mathbf G_{\Omega,\varepsilon,R,\sigma}f\|_{L^2(\omega)}
 \leq C\|f\|_{L^2(\sigma)}
\]
holds if and only if the following conditions hold:
\begin{enumerate}
\item the two one-tailed \(\mathcal A_2^1\) conditions with holes;
\item the two punctured \(A_2^1\) conditions;
\item the forward and dual Green testing conditions over all rotated
cubes.
\end{enumerate}
If \(M_\Omega(\sigma,\omega) :=\bigl(\sigma(\Gamma)\omega(\overline\Omega)\bigr)^{1/2}\), then
\[
 \mathfrak N_{\mathbf G_\Omega}+M_\Omega
 \approx
 \sqrt{\mathfrak A_2^1}
 +\mathfrak T_{\mathbf G_\Omega}
 +\mathfrak T_{\mathbf G_\Omega}^{\mathrm{dual}}
 +M_\Omega.
\]
The comparison constants may also depend on \(\Omega\), on
\(\dist(\Gamma,\partial\Omega)\), and on the curve constants in
\cref{thm:main}.
\end{theorem}

\begin{proof}
The bound for \(B_\Omega\) and Cauchy--Schwarz give \(|\mathbf S_{\Omega,\sigma}f(x)| \leq C_{\Omega,\Gamma}\sigma(\Gamma)^{1/2} \|f\|_{L^2(\sigma)}\). Consequently,
\[
 \|\mathbf S_{\Omega,\sigma}f\|_{L^2(\omega)}
 \leq C_{\Omega,\Gamma}M_\Omega(\sigma,\omega)
       \|f\|_{L^2(\sigma)}.
\]
The transposed operator satisfies the same estimate.  For every rotated
cube \(Q\), we also have
\[
 \begin{split}
 \int_Q|\mathbf S_{\Omega,\sigma}(\one_Q)|^2d\omega
 &\leq C_{\Omega,\Gamma}^2\sigma(Q)^2\omega(Q)\\
 &\leq C_{\Omega,\Gamma}^2M_\Omega(\sigma,\omega)^2\sigma(Q).
 \end{split}
\]
After interchanging \(\sigma\) and \(\omega\), the same calculation
gives the dual testing estimate.

It follows that
\[
 \begin{split}
 \mathfrak N_{\mathbf G_\Omega}
 &\leq |a_n|\mathfrak N_{\bfR^{1,n}}
       +C_{\Omega,\Gamma}M_\Omega,\\
 |a_n|\mathfrak N_{\bfR^{1,n}}
 &\leq \mathfrak N_{\mathbf G_\Omega}
       +C_{\Omega,\Gamma}M_\Omega.
 \end{split}
\]
The same two comparisons hold for the forward and dual testing
constants.  We now apply \cref{thm:main} with \(\alpha=1\).  It gives
the three stated conditions and the quantitative comparison.
\end{proof}

\subsection{Distance power measures}

We next consider a volume measure that is singular or degenerate at the
curve.  Put \(d_\Gamma(x):=\dist(x,\Gamma), \quad U_{\rho_0}:=\{x\in\Omega:d_\Gamma(x)<\rho_0\}\), where \(\rho_0>0\) is small enough that
\(\overline U_{\rho_0}\subset\Omega\).

\begin{corollary}
\label{cor:bounded-distance-power}
Let \(d\sigma=ds_\Gamma=d\mathcal H^1|_\Gamma\), let
\(\beta>-(n-1)\), and let \(d\omega_\beta(x) :=\one_{U_{\rho_0}}(x)d_\Gamma(x)^\beta\,dx\). Then \(\|\nabla u_f\|_{L^2(\omega_\beta)} \leq C_\beta\|f\|_{L^2(ds_\Gamma)}\) holds for every \(f\in L^2(ds_\Gamma)\) if and only if \(\beta>n-3\). Here \(u_f\) is the Dirichlet solution of \(-\Delta u_f=f\,ds_\Gamma\quad\text{in }\Omega, \quad u_f=0\quad\text{on }\partial\Omega\). In particular, for a curve source in \(\R^3\), the exact range is
\(\beta>0\).  At \(\beta=0\), the test density used below gives logarithmic
divergence.
\end{corollary}

\begin{proof}
We first prove sufficiency.  Partition the target tube into finitely many
smaller graph boxes subordinate to enlarged graph neighbourhoods of the
kind constructed in \cref{lem:nested-arcs}.  The part of the curve that
can approach one of these boxes lies in its enlarged graph neighbourhood;
all other pieces of the curve are separated from the box.  After a rigid
motion, the curve in one such neighbourhood has the form \(y=(t,\psi(t)),\quad t\in J\), where \(\|\psi'\|_\infty\) is small.  In a smaller neighbourhood write \(x=(s,\psi(s)+z),\quad z\in\R^{n-1}\). The Jacobian of this change of variables is one.  The nested-arc
separation in \cref{lem:nested-arcs} and the small slope give \(d_\Gamma(x)\approx|z|, \quad |x-y|\approx\bigl(|s-t|^2+|z|^2\bigr)^{1/2}\). The arc-length density
\((1+|\psi'(t)|^2)^{1/2}\) is bounded above and below.

Extend \(f\) by zero in the tangential variable.  For fixed
\(z\neq0\), the kernel size estimate and Young's inequality give
\[
 \|\bfR^{1,n}_{ds_\Gamma}f(\cdot,z)\|_{L_s^2}
 \leq C\|k_{|z|}\|_{L^1}\|f\|_{L^2(ds_\Gamma)},
 \qquad
k_\rho(t):=\one_{[-C,C]}(t)(t^2+\rho^2)^{(1-n)/2}.
\]
For \(0<\rho<\rho_0\),
\[
 \|k_\rho\|_{L^1}
 \leq
 \begin{cases}
C\rho^{2-n},&n\geq3,\\
C\log(e/\rho),&n=2.
 \end{cases}
\]
When \(n\geq3\), polar coordinates in \(\R^{n-1}\) give
\[
 \int_{|z|<\rho_0}|z|^\beta\|k_{|z|}\|_{L^1}^2\,dz
 \leq C\int_0^{\rho_0}\rho^{\beta+2-n}\,d\rho.
\]
This is finite exactly when \(\beta>n-3\).  When \(n=2\), the
corresponding integral is
\[
C\int_0^{\rho_0}\rho^\beta\log^2(e/\rho)\,d\rho,
\]
which is finite exactly when \(\beta>-1=n-3\).

Near an endpoint, the region \(s\geq0\) is covered by the preceding
estimate.  On the remaining region \(s<0\), put
\(s_-=-s\) and \(\delta=(s_-^2+|z|^2)^{1/2}\).  Cauchy--Schwarz gives
\[
 \left|
 \int_0^{|J|}
 \frac{f(t)}{((s-t)^2+|z|^2)^{(n-1)/2}}\,dt
 \right|^2
 \leq C\delta^{3-2n}\|f\|_2^2.
\]
Integration in the \(n\) variables \((s_-,z)\) again gives
\[
C\int_0^{\rho_0}\delta^{\beta+2-n}\,d\delta.
\]
Contributions from nonlocal pieces of the curve have bounded kernels.
They map \(L^2(ds_\Gamma)\) to \(L^2(\omega_\beta)\) because
\(\omega_\beta(U_{\rho_0})<\infty\).  The Green remainder has the same
property.  Summing the finitely many local estimates proves sufficiency.
The argument also bounds all tangent-line truncations uniformly, since
their multipliers are bounded by a fixed constant.

For \(n=2\), the assumed local finiteness condition is already
\(\beta>-1=n-3\).  We therefore prove necessity for \(n\geq3\).
Let \(\gamma\) be an
arc-length parametrization with \(\gamma(0)=0\), and put
\(N_0:=\gamma'(0)^\perp\).  After reducing \(r>0\), the map \((s,z)\longmapsto\gamma(s)+z, \quad (s,z)\in(-r,r)\times B_{N_0}(0,r)\), is one-to-one, its Jacobian is bounded above and below, and
\(d_\Gamma(\gamma(s)+z)\approx|z|\).  To see the first assertion,
project onto the line spanned by \(\gamma'(0)\).  The resulting scalar
function of \(s\) has derivative bounded away from zero.  The distance
comparison follows because \(\gamma'(s)\) remains close to
\(\gamma'(0)\).

Put \(\tau_s:=\gamma'(s)\) and
\(z_s:=z-(z\cdot\tau_s)\tau_s\).  Then \(|z_s|\approx|z|\).  After translating
the tangent variable by \(z\cdot\tau_s\), the normal part of the tangent
line integral is
\[
 -z_s\int_{-r+O(|z|)}^{r+O(|z|)}
 \frac{dt}{(t^2+|z_s|^2)^{n/2}}.
\]
Its absolute value is comparable to \(|z|^{2-n}\); changes at the two
finite endpoints are bounded.

Choose \(f\) equal to one on this subarc and supported on a slightly
larger arc.  Uniformly in \(s\), \(\gamma(s+t)-\gamma(s)=t\tau_s+E_s(t), \ |E_s(t)|\leq |t|\vartheta_\gamma(|t|)\). The difference between the curved kernel and the tangent-line kernel is
at most
\[
C\int_0^r
 \frac{t\vartheta_\gamma(t)}{(t^2+\rho^2)^{n/2}}\,dt,
 \qquad \rho:=|z|.
\]
A dyadic decomposition at \(t=\rho\) bounds this by
\[
C\rho^{2-n}
 \left(
  \vartheta_\gamma(\rho)
  +\sum_{\substack{k\geq0\\2^k\rho<r}}
   2^{-k(n-2)}\vartheta_\gamma(2^{k+1}\rho)
 \right).
\]
Choose \(r\) so small that the expression in parentheses is smaller
than the fixed tangent-line lower bound.  The part of the source outside the smaller arc and the Green remainder
are continuous across that arc.  The singular term tends to infinity,
so \(|\nabla u_f(\gamma(s)+z)|\geq c|z|^{2-n}\) for all small \(z\) and for \(s\) in a fixed shorter interval.  Integration in these coordinates gives
\[
 \int_{U_{\rho_0}}|\nabla u_f(x)|^2d_\Gamma(x)^\beta\,dx
 \geq c\int_0^{\rho_1}
 \rho^{4-2n}\rho^\beta\rho^{n-2}\,d\rho
 =c\int_0^{\rho_1}\rho^{\beta+2-n}\,d\rho.
\]
The last integral diverges when \(\beta\leq n-3\), with logarithmic
divergence at equality.  This proves necessity.
\end{proof}

The last proof also shows where the two parts of the \(T1\) criterion
enter.  If \(Q_r\) is a cube of side length \(r\) centered at an interior
point of \(\Gamma\), then \(ds_\Gamma(Q_r)\approx r, \quad \omega_\beta(Q_r)\approx r^{\beta+n}\). Hence the local fractional \(A_2^1\) product has size \(\frac{ds_\Gamma(Q_r)\omega_\beta(Q_r)} {|Q_r|^{2(1-1/n)}} \approx r^{\beta+3-n}\). The size condition requires \(\beta\geq n-3\).  If \(n\geq3\), the
weight is locally finite at equality and the local size product remains
comparable to $1$, but forward testing fails by the logarithmic integral in the
proof.  If \(n=2\), equality means \(\beta=-1\), and the weight is not
locally finite.  Thus every locally finite power in that dimension
already satisfies the strict lower inequality.

For the observation measure on the whole space, the decay of the field
at infinity imposes a second restriction.

\begin{corollary}
\label{cor:whole-space-distance-power}
Let \(\Gamma\subset\R^n\) be a compact regular
\(C^{1,\mathrm{Dini}}\) curve, let \(d\sigma=ds_\Gamma\), let
\(\beta>-(n-1)\), and put \(d\omega_\beta(x)=d_\Gamma(x)^\beta\,dx\). The whole-space estimate \(\|\bfR^{1,n}_{ds_\Gamma}f\|_{L^2(\omega_\beta)} \leq C_\beta\|f\|_{L^2(ds_\Gamma)}\) holds for every \(f\in L^2(ds_\Gamma)\) if and only if \(n-3<\beta<n-2\). For \(n=3\), the exact range is \(0<\beta<1\).
\end{corollary}

\begin{proof}
The proof of \cref{cor:bounded-distance-power} gives the exact local
condition \(\beta>n-3\).  On a bounded set separated from \(\Gamma\),
the kernel and the weight are bounded.  Write
$|\Gamma|=\mathcal H^1(\Gamma)$ and put
\(M_\Gamma:=\sup_{y\in\Gamma}|y|\).  If \(|x|>2M_\Gamma+1\), then
\(d_\Gamma(x)\approx|x|\) and
\[
 |\bfR^{1,n}_{ds_\Gamma}f(x)|
 \leq C|x|^{1-n}
 \int_\Gamma|f|\,ds_\Gamma
 \leq C|x|^{1-n}|\Gamma|^{1/2}\|f\|_2.
\]
The square of this bound is integrable at infinity with respect to
\(d_\Gamma(x)^\beta dx\) exactly when
\[
 \int_1^\infty r^{\beta+1-n}\,dr<\infty,
\]
or \(\beta<n-2\).  This proves sufficiency.

For necessity at infinity, take \(f=1\).  A first-order expansion of the
kernel gives \(\bfR^{1,n}_{ds_\Gamma}1(x) =-c_{1,n}|\Gamma|\frac{x}{|x|^n}+O(|x|^{-n}) \quad\text{as }|x|\to\infty\). The leading term is nonzero.  Its weighted square integral diverges when
\(\beta\geq n-2\), logarithmically at equality.  For \(n\geq3\), the local lower bound in
the preceding proof gives necessity of \(\beta>n-3\).  For \(n=2\),
this condition follows from the assumed local finiteness.
\end{proof}

For large cubes \(Q_R\) containing \(\Gamma\), \(ds_\Gamma(Q_R)\approx|\Gamma|, \quad \omega_\beta(Q_R)\approx R^{\beta+n}\). The corresponding local \(A_2^1\) product is comparable to
\(R^{\beta-n+2}\).  Together with the small-cube calculation, these
local products permit \(n-3\leq\beta\leq n-2\).  In dimension two, the
lower endpoint is excluded already by local finiteness.  At every
remaining endpoint, forward testing diverges logarithmically.  At the
upper endpoint, there is also a tail obstruction.  Indeed, for a large
cube \(Q_R\), when \(|x|\) is larger than a fixed multiple of \(R\),
we have \(\dist(x,\Gamma)\approx|x|\), and the reproducing Poisson
kernel is comparable to
\(R^{n-1}|x|^{-2(n-1)}\).  Polar coordinates show that the dual
one-tailed Poisson expression contains
\[
R^{n-1}\int_R^\infty r^{\beta-n+1}\,dr,
\]
which diverges logarithmically when \(\beta=n-2\).  Thus the one-tailed and testing conditions exclude the endpoints
allowed by the local size products.

\subsection{The Biot--Savart field}

We finish with a closely related singular integral in three dimensions.
For a vector density \(F\in L^2(\sigma;\R^3)\), define the uniformly
truncated Biot--Savart field by
\[
 \mathcal B_{\varepsilon,R,\sigma}F(x)
 :=\frac{1}{4\pi}\int_\Gamma
t_{\varepsilon,R}(x,y)
 \frac{F(y)\times(x-y)}{|x-y|^3}\,d\sigma(y),
\]
with value zero on the diagonal.  Biot--Savart integrals over curves
occur for vortex filaments and line currents; see \cite{Polanco2025}.
Here \(t_{\varepsilon,R}\) is the profile for \(\alpha=1\) and \(n=3\).

\begin{corollary}
\label{cor:Biot-Savart}
Let \(\Gamma\subset\R^3\) be a compact regular
\(C^{1,\mathrm{Dini}}\) curve, let \(\sigma\) be a finite positive Borel
measure supported on \(\Gamma\), and let \(\omega\) be a locally finite
positive Borel measure on \(\R^3\).  Then
\[
 \sup_{0<\varepsilon<R}
 \|\mathcal B_{\varepsilon,R,\sigma}F\|_{L^2(\omega;\R^3)}
 \leq C\|F\|_{L^2(\sigma;\R^3)}
\]
holds for every vector density \(F\) if and only if the four
Muckenhoupt conditions and the forward and dual Riesz testing conditions
in \cref{thm:main}, with \(\alpha=1\), hold for \((\sigma,\omega)\).
The Biot--Savart operator norm is comparable to
\(\mathfrak N_{\bfR^{1,3}}\).
\end{corollary}

\begin{proof}
Fix $0<\varepsilon<R$, and write $\mathcal B_\sigma$ and
$\mathbf R_\sigma$ for the corresponding tangent-line truncated
Biot--Savart and $1$-fractional Riesz operators in $\R^3$.
Because the Riesz numerator is \(y-x\), while the Biot--Savart numerator is
\(F(y)\times(x-y)\), there is a nonzero dimensional constant \(c_B\)
such that, if \(F=\sum_{k=1}^3F_ke_k\), then
\[
 \mathcal B_\sigma F
 =c_B\sum_{k=1}^3e_k\times\mathbf R_\sigma F_k.
\]
Thus
\(\|\mathcal B_\sigma\|\leq
|c_B|\sqrt3\|\mathbf R_\sigma\|\).  Conversely,
apply \(\mathcal B_\sigma\) to \(fe_k\).  Since
\[
 \mathcal B_\sigma(fe_k)=c_Be_k\times\mathbf R_\sigma f,
 \qquad
 \sum_{k=1}^3|e_k\times v|^2=2|v|^2,
\]
we obtain \(2|c_B|^2\|\mathbf R_\sigma f\|_{L^2(\omega)}^2 \leq3\|\mathcal B_\sigma\|^2\|f\|_{L^2(\sigma)}^2\). These comparisons hold uniformly for the same truncations.  The result
now follows from \cref{thm:main}.
\end{proof}

In particular, if \(d\sigma=ds_\Gamma\) and
\(d\omega_\beta=d_\Gamma^\beta dx\), then
\cref{cor:whole-space-distance-power,cor:Biot-Savart} give the exact
range \(0<\beta<1\) in \(\R^3\).  A physical filament density has the
form \(F=f\tau\), where \(\tau\) is the unit tangent.  The upper estimate
for such densities follows at once from \cref{cor:Biot-Savart}.  The
two-weight equivalence in that corollary is stated for all vector
densities; for the restricted family \(f\tau\), the cross product alone
does not recover every component of the Riesz vector.

\section*{Acknowledgements}

Ji Li is supported by the Australian Research Council through grant
DP260100485.  Chun-Yen Shen is supported in part by the National Science
and Technology Council of Taiwan through grant 111-2115-M-002-010-MY5.
Eric T. Sawyer is supported in part by the Natural Sciences and Engineering
Research Council of Canada.  Brett D. Wick is partially supported by the
National Science Foundation under grant DMS-2349868 and by the Australian
Research Council under grant DP260201083.


\begin{thebibliography}{99}


\bibitem{DavidJourne}
G.~David and J.-L. Journ\'e,
\emph{A boundedness criterion for generalized Calder\'on--Zygmund
operators},
Ann. of Math. (2) \textbf{120} (1984), no.~2, 371--397.

\bibitem{GjerdeKumarNordbotten2021}
I.~G. Gjerde, K.~Kumar and J.~M. Nordbotten,
\emph{A mixed approach to the Poisson problem with line sources},
SIAM J. Numer. Anal. \textbf{59} (2021), no.~2, 1117--1139.

\bibitem{GjerdeKumarNordbottenWohlmuth2019}
I.~G. Gjerde, K.~Kumar, J.~M. Nordbotten and B.~Wohlmuth,
\emph{Splitting method for elliptic equations with line sources},
ESAIM Math. Model. Numer. Anal. \textbf{53} (2019), no.~5,
1715--1739.

\bibitem{HytonenHilbert}
T.~Hyt\"onen,
\emph{The two-weight inequality for the Hilbert transform with general
measures},
Proc. Lond. Math. Soc. (3) \textbf{117} (2018), no.~2, 483--526.


\bibitem{DukeII}
M.~T. Lacey,
\emph{Two-weight inequality for the Hilbert transform:
A real variable characterization, II},
Duke Math. J. \textbf{163} (2014), no.~15, 2821--2840.

\bibitem{HilbertEnergy}
M.~T. Lacey, E.~T. Sawyer and I.~Uriarte-Tuero,
\emph{A two weight inequality for the Hilbert transform assuming
an energy hypothesis},
J. Funct. Anal. \textbf{263} (2012), no.~2, 305--363.

\bibitem{DukeI}
M.~T. Lacey, E.~T. Sawyer, C.-Y. Shen and I.~Uriarte-Tuero,
\emph{Two-weight inequality for the Hilbert transform:
A real variable characterization, I},
Duke Math. J. \textbf{163} (2014), no.~15, 2795--2820.


\bibitem{Hardy}
B.~Muckenhoupt,
\emph{Hardy's inequality with weights},
Studia Math. \textbf{44} (1972), 31--38.


\bibitem{NTVHilbert}
F.~Nazarov, S.~Treil and A.~Volberg,
\emph{Two weight estimate for the Hilbert transform and corona
decomposition for non-doubling measures},
preprint, arXiv:1003.1596.

\bibitem{Ojea2024}
I.~Ojea,
\emph{Anisotropic regularity for elliptic problems with Dirac measures
as data},
J. Math. Anal. Appl. \textbf{535} (2024), no.~1,
Paper No.~128104.

\bibitem{Polanco2025}
J.~I. Polanco,
\emph{Fast and accurate evaluation of Biot--Savart integrals over
spatial curves in periodic domains},
SIAM J. Sci. Comput. \textbf{47} (2025), no.~4, B822--B845.

\bibitem{SawyerMaximal}
E.~T. Sawyer,
\emph{A characterization of a two-weight norm inequality for maximal
operators},
Studia Math. \textbf{75} (1982), no.~1, 1--11.

\bibitem{SawyerFractional}
E.~T. Sawyer,
\emph{A characterization of two weight norm inequalities for fractional
and Poisson integrals},
Trans. Amer. Math. Soc. \textbf{308} (1988), no.~2, 533--545.

\bibitem{FailureEnergy}
E.~T. Sawyer, C.-Y. Shen and I.~Uriarte-Tuero,
\emph{A note on failure of energy reversal for classical fractional
singular integrals},
Int. Math. Res. Not. IMRN (2015), no.~19, 9888--9920.


\bibitem{EnergyCommon}
E.~T. Sawyer, C.-Y. Shen and I.~Uriarte-Tuero,
\emph{A two weight theorem for fractional singular integrals with an
energy side condition, quasicube testing and common point masses},
Preprint, \href{https://arxiv.org/abs/1505.07816v3}{arXiv:1505.07816v3}, 2016.


\bibitem{CurvePaper}
E.~T. Sawyer, C.-Y. Shen and I. Uriarte-Tuero,
\emph{The two weight $T1$ theorem for fractional Riesz transforms when one
measure is supported on a curve},
J. Anal. Math. \textbf{142} (2020), no.~2, 453--520.

\bibitem{GoodLambda}
E.~T. Sawyer, C.-Y. Shen and I. Uriarte-Tuero,
\emph{A good-$\lambda$ lemma, two weight $T1$ theorems without weak
boundedness, and a two weight accretive global $Tb$ theorem}, in
\emph{Harmonic Analysis, Partial Differential Equations and Applications},
Appl. Numer. Harmon. Anal., Birkh\"auser, Cham, 2017, 125--164.


\end{thebibliography}
\end{document}